\documentclass[reqno,english]{amsart}
\usepackage{amsfonts,amsmath,latexsym,verbatim,amscd,mathrsfs,color,array}
\usepackage[colorlinks=true]{hyperref}
\usepackage{amsmath,amssymb,amsthm,amsfonts,graphicx,color}
\usepackage{amssymb}
\usepackage{pdfsync}
\usepackage{epstopdf}
\usepackage{cite}
\usepackage{graphicx}

\newcommand{\R}{\mathbb R}

\newcommand{\bt}{\begin{theorem}}
\newcommand{\et}{\end{theorem}}
\newcommand{\bl}{\begin{lemma}}
\newcommand{\el}{\end{lemma}}
\newcommand{\bd}{\begin{definition}}
\newcommand{\ed}{\end{definition}}
\newcommand{\bc}{\begin{corollary}}
\newcommand{\ec}{\end{corollary}}
\newcommand{\bp}{\begin{proof}}
\newcommand{\ep}{\end{proof}}
\newcommand{\bx}{\begin{example}}
\newcommand{\ex}{\end{example}}
\newcommand{\bi}{\begin{exercise}}
\newcommand{\ei}{\end{exercise}}
\newcommand{\bo}{\begin{prop}}
\newcommand{\eo}{\end{prop}}
\newcommand{\br}{\begin{remark}}
\newcommand{\er}{\end{remark}}
\newcommand{\be}{\begin{equation}}
\newcommand{\ee}{\end{equation}}
\newcommand{\ba}{\begin{align}}
\newcommand{\ea}{\end{align}}
\newcommand{\bn}{\begin{enumerate}}
\newcommand{\en}{\end{enumerate}}
\newcommand{\bg}{\begin{align*}}
\newcommand{\bcs}{\begin{cases}}
\newcommand{\ecs}{\end{cases}}

\newcommand{\bean}{\begin{eqnarray*}}
\newcommand{\eean}{\end{eqnarray*}}

\newtheorem{definition}{Definition}[section]
\newtheorem{theorem}{Theorem}[section]
\newtheorem{lemma}{Lemma}[section]

\newtheorem{prop}{Proposition}[section]
\newtheorem{remark}{Remark}[section]
\numberwithin{equation}{section}
\begin{document}
\title[Sign-changing blowing-up solutions for the critical heat equation]{Sign-changing blowing-up solutions for the critical nonlinear heat equation}

\author[M. del Pino]{Manuel del Pino}
\address{\noindent   Department of Mathematical Sciences University of Bath,
Bath BA2 7AY, United Kingdom \\
and  Departamento de
Ingenier\'{\i}a  Matem\'atica-CMM   Universidad de Chile,
Santiago 837-0456, Chile}
\email{m.delpino@bath.ac.uk}

\author[M. del Pino]{Monica Musso}
\address{\noindent   Department of Mathematical Sciences University of Bath,
Bath BA2 7AY, United Kingdom \\
and Departamento de Matem\'aticas, Universidad Cat\'olica de Chile, Macul 782-0436, Chile}
\email{m.musso@bath.ac.uk}

\author[J. Wei]{Juncheng Wei}\address{\noindent Department of Mathematics, University of British Columbia, Vancouver, B.C., Canada, V6T 1Z2} \email{jcwei@math.ubc.ca}
\author[Y. Zheng]{Youquan Zheng}\address{\noindent School of Mathematics, Tianjin University, Tianjin 300072, P. R. China} \email{zhengyq@tju.edu.cn}
\begin{abstract}
Let $\Omega$ be a smooth bounded domain in $\mathbb{R}^n$ and denote the regular part of the Green's function on $\Omega$ with Dirichlet boundary condition as $H(x,y)$. Assume that $q \in \Omega$ and $n\geq 5$. We prove that there exists an integer $k_0$ such that for any integer $k\geq k_0$ there exist initial data $u_0$ and smooth parameter functions $\xi(t)\to q$, $0<\mu(t)\to 0$  as $t\to +\infty$ such that the solution $u_q$ of the critical nonlinear heat equation
\begin{equation*}
\begin{cases}
u_t = \Delta u + |u|^{\frac{4}{n-2}}u\text{ in } \Omega\times (0, \infty),\\
u = 0\text{ on } \partial \Omega\times (0, \infty),\\
u(\cdot, 0) = u_0 \text{ in }\Omega,
\end{cases}
\end{equation*}
has the form
\begin{equation*}
u_q(x, t) \approx \mu(t)^{-\frac{n-2}{2}}\left(Q_k\left(\frac{x-\xi(t)}{\mu(t)}\right) - H(x, q)\right),
\end{equation*}
where the profile $Q_k$ is the non-radial sign-changing solution of the Yamabe equation
\begin{equation*}
\Delta Q + |Q|^{\frac{4}{n-2}}Q = 0\text{ in }\mathbb{R}^n,
\end{equation*}
constructed in \cite{delpinomussofrankpistoiajde2011}. In dimension 5 and 6, we also prove the stability of $u_q(x, t)$.
\end{abstract}
\maketitle
\section{Introduction}
Let $\Omega$ be a smooth bounded domain in $\mathbb{R}^n$ with $n\geq 3$. We consider the following critical nonlinear heat equation
\begin{equation}\label{e:main}
\begin{cases}
u_t = \Delta u + |u|^{\frac{4}{n-2}}u\text{ in } \Omega\times (0, \infty),\\
u = 0\text{ on } \partial \Omega\times (0, \infty),\\
u(\cdot, 0) = u_0 \text{ in }\Omega,
\end{cases}
\end{equation}
for a function $u:\overline{\Omega}\times [0, \infty)\to \mathbb{R}$ and smooth initial datum $u_0$ satisfying $u_0|_{\partial \Omega} = 0$.

\medskip

Problem (\ref{e:main}) can be viewed as a special case of the well-known Fujita equation
\begin{equation}\label{1.1}
u_t = \Delta u + |u|^{p-1}u
\end{equation}
with $ p>1$, which appears in many applied disciplines and become a prototype for the analysis of singularity formation in nonlinear parabolic equations.
A large amount of literature has been devoted to this problem on the asymptotic behaviour and blowing-up solutions after Fujita's seminal work \cite{fujitajfsuts1966blowup}. See,  for example, \cite{CollotMerleRaphaelCMP2017}, \cite{cortazar2016green}, \cite{delmussoweitype2}, \cite{delmussowei3d}, \cite{GigaKohn1987}, \cite{GigaKohn1985}, \cite{GigaKohn1989}, \cite{GigaMatsuiSasayama2004}, \cite{GigaMatsuiSasayama2004a}, \cite{matanomerlejfa2011threshold}, \cite{MM1}, \cite{MM2}, \cite{merlezaagduke1997stability}, \cite{schweyerjfa2012typeii} and references therein. We refer the the interested readers to \cite{quittnersouplet2007superlinear} for the corresponding background and a comprehensive survey of the results until 2007. Blowing-up phenomena for problem (\ref{1.1}) is very sensitive to the exponent $p$, the critical case $p=\frac{n+2}{n-2}$ is special in several ways, positive steady state solutions do not exist if $p< \frac{n+2}{n-2}$. Radial and positive global solutions must go to zero and bounded, see \cite{PQ1}, \cite{PQ2}, \cite{quittnersouplet2007superlinear}, they exist in the case $p> \frac{n+2}{n-2}$ with infinite energy, see
\cite{Gui1992}. Infinite time blowing-up solutions exist in that case but they exhibit entirely different nature, see \cite{PY1}, \cite{PY2}.

The motivation of this paper is twofolds. In \cite{cortazar2016green}, Cortazar, del Pino and Musso proved the following result.
Suppose $n\geq 5$, denote the Green's function of the Laplacian $\Delta$ in $\Omega$ with Dirichlet boundary value as $G(x, y)$ and the regular part of $G(x, y)$ as  $H(x,y)$.
Let $q_1,\cdots, q_k$ be $k$ distinct points in $\Omega$ such that the matrix
\begin{eqnarray}\label{positivematrix}
\hat{\mathcal{G}}(q) = \left[
\begin{matrix}
H(q_1, q_1)&-G(q_1, q_2)&\cdots & -G(q_1, q_k)\\
-G(q_2, q_1)&H(q_2, q_2)&\cdots &-G(q_2, q_k)\\
\vdots&\vdots&\ddots&\vdots\\
-G(q_k,q_1)&-G(q_k, q_2)&\cdots&H(q_k, q_k)
\end{matrix}
\right].
\end{eqnarray}
is positive definite. They proved the existence of $u_0$ and smooth parameter functions $\xi_j(t)\to q_j$, $0<\mu_j(t)\to 0$, as $t\to +\infty$, $j = 1, \cdots, k$, such that (\ref{e:main}) has an infinite time blowing-up {\it positive} solution $u_q$ whose shape can be approximately described as
\begin{equation*}
u_q \approx\sum_{j= 1}^k\alpha_{n}\left(\frac{\mu_j(t)}{\mu_j^2(t) + |x-\xi_j(t)|^2}\right)^{\frac{n-2}{2}},
\end{equation*}
with $\mu_j(t) = \beta_jt^{-\frac{1}{n-4}}(1+o(1))$, as $t \to \infty$, for
 some positive constants $\beta_j$. The profile of $u_q$ is a super-position of functions that are obtained from a fixed profile
\begin{equation}\label{positivesolution}
U(x) = \alpha_n\left(\frac{1}{1+|x|^2}\right)^{\frac{n-2}{2}},
\end{equation}
properly scaled by $\mu_j (t)$ and translated by $\xi_j (t)$. The function $U$
is the unique radially symmetric entire solution for the Yamabe equation
\begin{equation}\label{e:sign-changing bubble}
\Delta Q + |Q|^{\frac{4}{n-2}}Q = 0\text{ in }\mathbb{R}^n.
\end{equation}
The aim of this paper is to explore the possibility to construct sign-changing solutions to (\ref{e:main}) which blows-up, as $t\to \infty$,
in the form of the profile of a sign-changing entire solution to the time-independent limit problem (\ref{e:sign-changing bubble}).
Pohozaev's identity tells us that any sign-changing solution of (\ref{e:sign-changing bubble}) is non-radial. The existence of non-radial sign-changing and with arbitrary large energy elements of $\Sigma:=\left\{ Q \in {\mathcal D}^{1,2} (\R^n) \backslash \{0\}: Q \text{ satisfies } (\ref{e:sign-changing bubble})\right\}$ was first proved by W. Ding \cite{Ding1986} using variational arguments. Indeed, using stereographic projection to $S^n$, (\ref{e:sign-changing bubble}) transforms into
$$
\Delta_{S^n} v + \frac{n(n-2)}{4} ( |v|^{\frac 4{n-2}} v - v) = 0 \text{ in } S^n,
$$
(see, for example, \cite{Sch-Yau}, \cite{Heb}), Ding proved the existence of infinitely many critical points to the corresponding energy functional in the space of  functions satisfying
$$  v(x)= v(|x_1|,|x_2|), \quad x= (x_1, x_2) \in  S^n \subset \R^{n+1}= \R^k\times \R^{n+1-k}, \quad k\ge 2.$$
More explicit constructions of sign-changing solutions to (\ref{e:sign-changing bubble}) were obtained in \cite{delpinomussofrankpistoiajde2011}, \cite{delpinomussofrankpistoiajde2013}. Furthermore, \cite{MussoWei2015} proves the rigidity results (non-degeneracy) of the solutions found in \cite{delpinomussofrankpistoiajde2011}, \cite{delpinomussofrankpistoiajde2013}. Classification of solutions in $\Sigma$ plays an important role in the soliton resolution conjecture for energy critical wave equation, for example, \cite{KenigMerle2016}, \cite{KenigMerle2017} and the references therein. Therefore, a natural question is: does the infinite time blowing-up phenomenon for problem (\ref{e:main}) occur with sign-changing profiles? In this paper we show that the sign-changing blowing-up solutions with basic cell constructed in \cite{delpinomussofrankpistoiajde2011} do exist.

Our starting point is the sign-changing solutions $Q$ of (\ref{e:sign-changing bubble}) constructed in \cite{delpinomussofrankpistoiajde2011} and \cite{delpinomussofrankpistoiajde2013}. Let us describe these solutions more precisely. In \cite{delpinomussofrankpistoiajde2011}, it was proven that there exists a large positive integer $k_0$ such that $\forall k \geq k_0$, a solution $Q = Q_k$ of (\ref{e:sign-changing bubble}) exists. Furthermore, if we define the energy functional by
\begin{equation*}\label{energy1}
E(u) = {1\over 2}\int_{\R^n} |\nabla u|^2 dx - {1\over p+1} \int_{\R^n} |u|^{p+1}dx, \quad p = \frac{n+2}{n-2},
\end{equation*}
then we have
$$
E (Q_k ) =  \left\{ \begin{matrix}   (k+1) \, S_n \, \left( 1+ O(k^{2-n})  \right)  & \hbox{ if } n\geq 4\, , \\ & \\     (k+1) \, S_3 \, \left( 1+ O(k^{-1} |\log k|^{-1} \right) & \hbox{ if } n=3 \end{matrix}
  \right.
$$
as $k \to \infty$. Here $S_n$ is a positive constant depending on $n$. The function $Q=Q_k$ decays like the radial symmetrical solution $U(x)$ defined in (\ref{positivesolution}) at infinity, that is to say, we have
\begin{equation}\label{vanc1}
\lim_{|x| \to \infty} |x|^{n-2} \, Q_k (x ) = \left[ {4\over n(n-2) } \right]^{n-2 \over 4} \, 2^{n-2 \over 2} \left( 1+ d_k \right)
\end{equation}
where
$$
 d_k = \left\{ \begin{matrix}   O(k^{-1} )  & \hbox{ if } n\geq 4\, , \\ & \\     O(k^{-1} |\log k|^2 )  & \hbox{ if } n=3 \end{matrix}
  \right. \ \quad {\mbox {as}} \quad k \to \infty.
$$
Furthermore, we have
\begin{equation*}\label{at0}
Q(x) = \left[ n (n-2) \right]^{n-2 \over 4} \, \left ( 1- {n-2 \over 2} |x|^2 + O(|x|^3 ) \right) \quad {\mbox {as}} \quad |x|\to 0
\end{equation*}
and there exists $\eta >0$ (depending only on $k_0$) such that for any $k$,
\begin{equation*}\label{at00}
\eta \leq Q(x) \leq Q(0) \quad {\mbox {for all}} \quad |x|\leq {1\over 2}.
\end{equation*}
On the other hand, $Q=Q_k$ is invariant under rotation of angle ${2\pi \over k}$ in the $x_1 , x_2$ plane, i.e.,
\begin{equation}\label{sim00}
Q( e^{2\pi \over k} \bar x , x' ) = Q(\bar x , x' ), \quad \bar x= (x_1, x_3) , \quad x'= (x_3, \ldots , x_n ).
\end{equation}
It is also even in the $x_j$-coordinates, for any $j=2, \cdots , n$ and invariant under the Kelvin's transformation, namely, we have
\begin{equation}\label{sim22}
Q (x_1,\ldots,x_j, \ldots, x_{n}  ) =  Q (x_1,\ldots,-x_j, \ldots, x_{n}  ),\quad  j=2,\ldots,n
\end{equation}
and
\begin{equation}\label{sim33}
Q (x )\  = \  |x|^{2-n} Q (|x|^{-2} x).
\end{equation}

It was proved in \cite{MussoWei2015} that these solutions are non-degenerate. More precisely, fix one solution $Q=Q_k$ and define the linearized operator of (\ref{e:sign-changing bubble}) at $Q$ as
\begin{equation}\label{defL}
L(\phi ) = \Delta \phi + p |Q|^{p-1}  \phi.
\end{equation}
The invariance of any solution of (\ref{e:sign-changing bubble}) under dilation (if $u$ satisfies (\ref{e:sign-changing bubble}), then the function $\mu^{-\frac{n-2}{2}} u(\mu^{-1} x)$ solves (\ref{e:sign-changing bubble}) for all $\mu >0$), under translation (if $u$ solves (\ref{e:sign-changing bubble}), then $u(x+\xi)$ also solves (\ref{e:sign-changing bubble}) for $\xi\in\mathbb{R}^n$), together with the invariance (\ref{sim00}), (\ref{sim22}), (\ref{sim33}) produce natural kernel functions $\varphi$ of $L$, that is to say, we have
$$
L(\varphi ) = 0.
$$
These are $3n$ linearly independent functions defined as follows:
\begin{equation}\label{capitalzeta0}
z_0 (x) = {n-2 \over 2} Q(x) + \nabla Q (x) \cdot x ,
\end{equation}
\begin{equation}\label{capitalzetaj}
z_\alpha (x)  = {\partial \over \partial x_\alpha  } Q(x) , \quad {\mbox {for}} \quad \alpha=1, \ldots , n,
\end{equation}
\begin{equation}\label{capitalzeta2}
z_{n+1} (x) = -x_2  {\partial \over \partial x_1 } Q(x) + x_1  {\partial \over \partial x_2  } Q(x),
\end{equation}
\begin{equation}\label{chico1}
z_{n+2} (x) = -2 x_1 z_0 (x) + |x|^2 z_1 (x) , \quad z_{n+3} (x) =  -2 x_2 z_0 (x) + |x|^2 z_2 (x)
\end{equation}
and, for $l=3, \ldots , n$
\begin{equation}
\label{chico2}
z_{n+l+1} (x) = -x_l z_1 (x) + x_1 z_l (x), \quad z_{2n+l-1} (x) =  -x_l z_2 (x) + x_2 z_l (x).
\end{equation}
Indeed, direct computations yield that
$$
L(z_\alpha ) = 0 , \quad {\mbox {for all}} \quad \alpha = 0 , 1 , \ldots , 3n-1.
$$
The function $z_0$ defined by (\ref{capitalzeta0}) is from the invariance of (\ref{e:sign-changing bubble}) under dilation $\mu^{-{\frac{n-2}{2}}} Q (\mu^{-1} x)$. $z_i$, $i=1, \ldots , n$ defined by (\ref{capitalzetaj}) are due to the invariance of (\ref{e:sign-changing bubble}) under translation $Q (x+\xi)$. The function $z_{n+1}$ in (\ref{capitalzeta2}) is generated from the invariance of $Q$ with respect to rotation in the $(x_1 , x_2)$-plane. The functions $z_{n+2}$ and $z_{n+3}$ in (\ref{chico1}) are generated from the invariance of (\ref{e:sign-changing bubble}) with respect to the Kelvin transformation (\ref{sim33}). The functions in (\ref{chico2}) are due to the invariance of (\ref{e:sign-changing bubble}) under rotations in the $(x_1, x_l)$-plane, $(x_2, x_l)$-plane respectively.

Let us recall that the Green's function $G(x,y)$ is defined by the following Dirichlet boundary value problem
\begin{equation*}
\begin{cases}
-\Delta G(x,y) = c(n)\delta(x-y)&\text{ in }\Omega, \\
G(\cdot, y) = 0&\text{ on }\partial\Omega,
\end{cases}
\end{equation*}
where $\delta(x)$ is the Dirac measure at the origin and $c(n)$ is a constant depending on $n$ satisfying
\begin{equation*}
-\Delta \Gamma(x) = c(n)\delta(x),\quad \Gamma(x) = \frac{Q(0)}{|x|^{n-2}} = \frac{\left[n(n-2)\right]^{\frac{n-2}{4}}}{|x|^{n-2}}.
\end{equation*}
Denote the regular part of $G(x,y)$ as $H(x,y)$, namely, $H(x,y)$ satisfies the following problem
\begin{equation*}
\begin{cases}
-\Delta H(x,y) = 0&\text{ in }\Omega,\\
H(\cdot, y) = \Gamma(\cdot - y)&\text{ in }\partial\Omega.
\end{cases}
\end{equation*}
Our main result can be stated as follows.
\begin{theorem}\label{t:main}
Assume $n\geq 5$ and $q$ is a point in $\Omega$. There exists an integer $k_0$ such that, for any $k\geq k_0$, there exist an initial datum $u_0$ and smooth parameter functions $\xi(t)\to q$, $0<\mu(t)\to 0$, as $t\to +\infty$, such that the solution $u_q$ to (\ref{e:main}) has form
\begin{equation}\label{solutionfinal}
u_q(x, t) =\mu(t)^{-\frac{n-2}{2}}\left(Q_k\left(\frac{x-\xi(t)}{\mu(t)}\right) - H(x, q) + \varphi(x, t)\right),
\end{equation}
where $\varphi(x, t)$ is a bounded smooth function satisfying $\varphi(x, t)\to 0$ uniformly away from $q$ as $t\to +\infty$.
\end{theorem}

Theorem \ref{t:main} exhibits new blowing-up phenomena where the profile of bubbling is sign-changing rather than the positive solution for the critical heat equation. In the case of positive bubbling, the linear operator around the basic cell contains exactly $n+1$ dimensional kernels corresponding to the rigidity motions (translation and dilation). However, in the case of sign-changing (non-radial) blowing-up solution, the kernel of the linearized operators at the basic cell includes not only the functions generated from dilation and translations, but also functions due to rotation around the sub-planes and Kelvin transform. Therefore we have to find enough parameter functions to adjust. Similar to the supercritical Bahri-Coron's problem in \cite{MussoWei2016}, our computations indicate that the dominated role played is still scaling and translations. Indeed, (\ref{solutionfinal}) has a more involved form, see (\ref{e2:3}) below for details. Note that in \cite{zhangmd}, sign-changing blow-up solutions were also constructed, but their basic cell is the positive radial solution $U(x)$ defined in (\ref{positivesolution}).

We believe that this is the {\em first} example of blowing-up solutions in nonlinear parabolic equations  whose core profile is non-radial. In a series of interesting papers, Duyckaerts, Kenig and Merle \cite{KenigMerle2016, KenigMerle2012} introduced the notion of nondegenracy for nonradial solutions of the equation (\ref{e:sign-changing bubble}) and obtained the profile decomposition for possible blow-up solutions for energy critical wave equation in general setting. Existence of bubbling solutions with the positive radial profile for the energy critical wave equations has been constructed in \cite{DK, KST}. However as far as we know there are no examples of noradial blow-up for energy critical wave equation.

To prove Theorem \ref{t:main}, we will use the {\it inner-outer gluing scheme} for parabolic problems. Gluing methods have been proven very useful in singular perturbation elliptic problems, for example, \cite{dkw2007concentration}, \cite{delwei2011Degiorgi}, \cite{delkowalczykweijdg2013entire}. Recently, this method has also been developed to various evolution problems, for instance, the construction of infinite time blowing-up solutions for energy critical nonlinear heat equation \cite{cortazar2016green}, \cite{delmussowei3d}, the formation of singularity to harmonic map flow \cite{davila2017singularity}, finite time blowing-up solutions for energy critical heat equation \cite{delmussoweitype2}, vortex dynamics in Euler flows \cite{davila2018Euler} and type II ancient solutions for the Yamabe flow \cite{del2012type}.

The proof consists of constructing an approximation to the solution with sufficiently small error, then we solve for a small remainder term using linearization around the bubble and the Schauder fixed-point arguments. In Section 2, we construct the first approximation with form (\ref{e2:3}). To get an approximation with fast decay far away from the point $q$, we add nonlocal terms to cancel the slow decay parts as in \cite{davila2017singularity}. Then we compute the error, in order to improve the approximation error near the point $q$, we have to use solvability conditions for the corresponding elliptic linearized operator around the sign-changing bubble. These conditions yield  systems of  ODEs for the parameters. (See equations (\ref{e5:9}).) Of these systems of the ODEs, the equation for the scaling parameter function plays a dominant role, from which we deduce the blow-up dynamics of our solutions.  After the approximate solution has been constructed, the full problem is solved as a small perturbation by the {\em inner-outer gluing scheme}, see Section 3. This consists of decomposing the perturbation term into form $\eta\tilde{\phi} +\psi$, where $\eta$ is a smooth cut-off function vanishing away from $q$. The tuple $(\tilde{\phi}, \psi)$ satisfy a coupled nonlinear parabolic system where the equation for $\psi$ is a small perturbation of the standard heat equation, and $\tilde{\phi}$ satisfies the parabolic linearized equation around the bubble.

When dealing with parabolic problems for $\tilde{\phi}$, a crucial step is to find a solution to the linearized parabolic equation around the bubble with sufficiently fast decay. However, it seems that the argument in \cite{cortazar2016green} for the positive bubbling of the critical heat equation does not work in our sign-changing case since we can not perform Fourier mode expansions. Inspired by the linear theory of \cite{davila2017singularity}, \cite{MussoSireWeiZhengZhou} and \cite{sireweizhenghalf}, our main contributions in this paper is to use blowing-up arguments based on the non-degeneracy of bubbles proved in \cite{MussoWei2015} and a removable of singularity property for the corresponding limit equation. As pointed out in \cite{KenigMerle2016}, the term $|Q|^{p-1} = |Q|^{\frac{4}{n-2}}$ in $L(\phi) = \Delta \phi + p |Q|^{p-1}\phi$ is not $C^1$ when the space dimension $n \geq 7$, as a result of this fact, the solution $\tilde{\phi}$, $\psi$ do not have Lipschitz property with respect to the parameter functions. This is the reason we use Schauder fixe-point theorem rather than Contraction Mapping Theorem to solve the inner-outer gluing parabolic system in Section 4. In dimension 5 and 6, $\tilde{\phi}$ and $\psi$ do have Lipschitz continuity with respect to the parameter functions, Theorem \ref{t:main} as well as a stability result for $u_q$ can be proved using the Contraction Mapping Theorem in the spirit of \cite{cortazar2016green}, see Section 8.

\section{Construction of the approximation}
\subsection{The basic cell.}
Let $O(n)$ be the orthogonal group of $n\times n$ matrices $M$ with real coefficients and $M^T M = I$, $SO(n) \subset O(n) $ be the special orthogonal group of all matrices in $O(n)$ satisfying $det(M)=1$. It is well known that $SO(n)$ is a compact group containing all rotations in $\R^n$,  and via isometry, it can be identified with a compact subset of $\R^{n (n-1) \over 2}$. Let $\hat S$ be the subgroup of $SO(n)$ generated by rotations in the $(x_1 , x_2)$-plane and $(x_j , x_\alpha)$-plane, for any $j=1,2$, $\alpha = 3, \ldots , n$. Then $\hat S$ is a compact manifold of dimension $2n-3$ without boundary. That is to say, there exists a smooth injective map $\chi : \hat S \to \R^{{n(n-1) \over 2}}$ such that $\chi ( \hat S)$ is a compact manifold without boundary of dimension $2n-3$ and $\chi^{-1} : \chi (\hat S ) \to \hat S$ is the smooth parametrization of $\hat S$ in a neighborhood of the identity map. Let us write
$$
\theta \in K =  \chi(\hat S), \quad R_\theta = \chi^{-1} (\theta )
$$
for a smooth compact manifold $K$ of dimension $2n-3$ and $R_\theta$ denotes a rotation map in $\hat S$.

Let $A = (\mu, \xi , a, \theta ) \in \mathbb{R}^+ \times \mathbb{R}^n \times \mathbb{R}^2 \times \mathbb{R}^{2n-3}$, define
\begin{equation} \label{defthetaA}
Q_A (x) = \mu^{-{n-2 \over 2}} \left| \eta_{A} (x)  \right|^{2-n} \,
Q \left( { R_\theta \left( {x-\xi \over \mu }  - a |{x-\xi \over \mu }|^2 \right) \over | \eta_{A} (x)|^2 }\right),
\end{equation}
where
\begin{equation}
\label{defetaA}
\eta_{A} (x) = {x-\xi \over |x-\xi |} - a {|x-\xi | \over \mu }
\end{equation}
and $Q$ is the fixed non-degenerate solution to problem (\ref{e:sign-changing bubble}) as described in the introduction.
It was proved in \cite{KenigMerle2016} that for any choice of $A$, $Q_A$ still satisfies (\ref{e:sign-changing bubble}), i.e.,
$$
\Delta Q_A + |Q_A |^{p-1} Q_A = 0 , \quad {\mbox {in}} \quad \R^n.
$$
Direct computations yield the following relations between the differentiation of $Q_A$ with respect to each component of $A$ and $z_\alpha$ defined in (\ref{capitalzeta0}), (\ref{capitalzetaj}), (\ref{capitalzeta2}), (\ref{chico1}) and (\ref{chico2}). Precisely, we have
\begin{equation}\label{z011}
z_0 (y) = - \, {\partial \over \partial \mu} \left[ Q_A (x) \right]_{| \mu = 1, \xi = 0, a =0 , \theta=0 }
\end{equation}
\begin{equation}\label{zj11}
z_\alpha  (y) = - \, {\partial \over \partial \xi_\alpha } \left[ Q_A (x) \right]_{| \mu = 1, \xi = 0, a =0 , \theta=0 },
\quad \alpha = 1, \ldots , n,
\end{equation}
\begin{equation}\label{z211}
z_{n+2} (y) =  \, {\partial \over \partial a_1} \left[ Q_A (x) \right]_{| \mu = 1, \xi = 0, a =0 , \theta=0 },
\end{equation}
\begin{equation}\label{z311}
z_{n+3} (y) =  \, {\partial \over \partial a_2} \left[ Q_A (x) \right]_{| \mu = 1, \xi = 0, a =0 , \theta=0 }.
\end{equation}
Let $\theta = (\theta_{12}, \theta_{13}, \ldots, \theta_{1n}, \theta_{23} , \ldots , \theta_{2n} ) $, where
$\theta_{ij}$ is the rotation in the $(i,j)$-plane, then we have
\begin{equation}\label{z411}
z_{n+1} (y) =  \, {\partial \over \partial \theta_{12}} \left[ Q_A (x) \right]_{| \mu = 1, \xi = 0, a =0 , \theta=0 }
\end{equation}
and, for $l=3, \ldots , n$,
\begin{equation}\label{z5110}
z_{n+l+1} (y) =  \, {\partial \over \partial \theta_{1l}} \left[ Q_A (x) \right]_{| \mu = 1, \xi = 0, a =0 , \theta=0 },
\end{equation}
\begin{equation}\label{z511}
z_{2n+l-1} (y) =  \, {\partial \over \partial \theta_{2l}} \left[ Q_A (x) \right]_{| \mu = 1, \xi = 0, a =0 , \theta=0 }.
\end{equation}
Following the definition in \cite{KenigMerle2016}, a solution $Q$ of (\ref{e:sign-changing bubble}) is non-degenerate if
\begin{equation}\label{nondeg}
{\mbox {Kernel}} (L) = {\mbox {Span}} \{ z_\alpha \, : \, \alpha = 0 , 1 , 2, \ldots , 3n-1 \},
\end{equation}
or equivalently, any bounded solution of $L(\varphi) = 0 $ is a linear combination of $z_\alpha $, $\alpha = 0 , \ldots , 3n-1$.
It was proved in \cite{MussoWei2015} that, the solution $Q$ is non-degenerate when the dimension satisfies some extra conditions.
Indeed, the authors showed that for all dimensions $n\leq 48$, any solution $Q=Q_k$ is non-degenerate, for dimension $n \geq 49$, there exists a subsequence of solutions $Q_{k_j}$ which is non-degenerate in the sense (\ref{nondeg}).

\subsection{Setting up the problem.}
Let $t_0 > 0$ be a sufficiently large constant, let us consider the heat equation
\begin{equation}\label{e3:1}
\begin{cases}
u_t = \Delta u + |u|^{\frac{4}{n-2}}u&\text{  in  }\Omega\times (t_0, \infty),\\
u = 0&\text{  in  }\partial \Omega\times (t_0, \infty).
\end{cases}
\end{equation}
Observe that the solution of (\ref{e3:1}) provides a solution $u(x, t) = u(x, t-t_0)$ to (\ref{e:main}). Given a fixed point $q\in \Omega$, we will find a solution $u(x, t)$ of equation (\ref{e3:1}) with approximate form
$$
u(x, t)\approx \mu(t)^{-\frac{n-2}{2}} Q\left(\frac{x-\xi(t)}{\mu(t)}\right).
$$
More precisely, let $A = A(t) = (\mu(t), \xi(t), a(t), \theta(t))\in \mathbb{R}^+\times\mathbb{R}^n\times \mathbb{R}^2\times \mathbb{R}^{2n-3}$ be the parameter functions and define the function
\begin{equation}
Q_{A(t)}(x) = \mu(t)^{-\frac{n-2}{2}}\left|\eta_{A(t)}(x)\right|^{2-n}Q\left(\frac{R_{\theta(t)}\left(\frac{x-\xi(t)}{\mu(t)} - a(t)\left|\frac{x-\xi(t)}{\mu(t)}\right|^2\right)}{|\eta_{A(t)}(x)|^2}\right),
\end{equation}
where
\begin{equation}
\eta_{A(t)}(x) = \frac{x-\xi(t)}{|x-\xi(t)|} - a(t)\frac{|x-\xi(t)|}{\mu(t)}
\end{equation}
and $Q$ is the non-degenerate solution for (\ref{e:sign-changing bubble}) described in Section 2.1. With abuse of notation when there is no ambiguity, here and in what follows, $A(t) = (\mu(t), \xi(t), a(t), \theta(t))$ will be abbreviated as $A = (\mu, \xi, a, \theta)$, $a$ is a vector in $\mathbb{R}^2$, $a = \begin{pmatrix}a_1\\ a_2\end{pmatrix}\in \mathbb{R}^2$, it is also a vector in $\mathbb{R}^n$, namely,
$$
a = \begin{pmatrix}a_1\\ a_2\\ 0\\ \cdots\\ 0\end{pmatrix}\in \mathbb{R}^n.
$$

To begin with, we assume that for a fixed positive function $\mu_0(t)\to 0$ ($t\to +\infty$) and a constant $\sigma > 0$, there hold
$$
\mu(t) = \mu_0(t) + O(\mu_0^{1+\sigma}(t))\quad\text{as}\quad t\to +\infty,
$$
$$
\xi(t) = q + O(\mu_0^{1+\sigma}(t))\quad\text{as}\quad t\to +\infty,
$$
$$
a(t) = O(\mu_0^{\sigma}(t))\quad\text{as}\quad t\to +\infty,
$$
$$
\theta(t) = O(\mu_0^{\sigma}(t))\quad\text{as}\quad t\to +\infty.
$$
In \cite{KenigMerle2016}, it was proven that for any choice of $A$, the function $Q_A$ still satisfies (\ref{e:sign-changing bubble}), namely
$$
\Delta Q_A + |Q_A|^{p-1}Q_A = 0\text{ in }\mathbb{R}^n.
$$
Let $\tilde{y} = \frac{R_{\theta(t)}\left(\frac{x-\xi(t)}{\mu(t)} - a(t)\left|\frac{x-\xi(t)}{\mu(t)}\right|^2\right)}{|\eta|^2}$ and $\eta =   \frac{x-\xi(t)}{|x-\xi(t)|} - a(t)\frac{|x-\xi(t)|}{\mu(t)}$, then we have the following expansion
\begin{equation*}
\begin{aligned}
|\eta|^2 &= \left|\frac{x-\xi(t)}{|x-\xi(t)|} - a(t)\frac{|x-\xi(t)|}{\mu(t)}\right|^2\\ &= 1 - 2 a(t)\cdot \left(\frac{x-\xi(t)}{\mu(t)}\right) + |a(t)|^2\frac{|x-\xi(t)|^2}{\mu^2(t)},
\end{aligned}
\end{equation*}
\begin{equation*}
\begin{aligned}
\frac{1}{|\eta|^2} &= \frac{1}{1 - 2 a(t)\cdot \left(\frac{x-\xi(t)}{\mu(t)}\right) + |a(t)|^2\frac{|x-\xi(t)|^2}{\mu^2(t)}}\\
&= 1 + 2 a(t)\cdot \left(\frac{x-\xi(t)}{\mu(t)}\right) +O\left(|a(t)|^2\frac{|x-\xi(t)|^2}{\mu^2(t)}\right)
\end{aligned}
\end{equation*}
and
\begin{equation*}
\begin{aligned}
\tilde{y} = \frac{R_{\theta(t)}\left(\frac{x-\xi(t)}{\mu(t)} - a(t)\left|\frac{x-\xi(t)}{\mu(t)}\right|^2\right)}{|\eta|^2}& = R_{\theta(t)}\left(\frac{x-\xi(t)}{\mu(t)}\right) + R_{\theta(t)} a(t)\left|\frac{x-\xi(t)}{\mu(t)}\right|^2\\
& \quad + O\left(|a|^2\frac{|x-\xi(t)|^3}{\mu^3(t)}\right).
\end{aligned}
\end{equation*}
Denote the error operator as
\begin{equation*}
S(u):= -u_t+\Delta u + |u|^{p-1}u,
\end{equation*}
with $p = \frac{n+2}{n-2}$ . Then the error of the first approximation $Q_A(x, t)$ can be computed as
\begin{eqnarray*}
\begin{aligned}
S(Q_A) &= -\frac{\partial }{\partial t}\left(Q_A(x, t)\right) = \mathcal{E}_0 + \mathcal{E}_1 + \mathcal{E}_2 + \mathcal{E}_3.
\end{aligned}
\end{eqnarray*}
For $y = \frac{x-\xi(t)}{\mu(t)}$, using Taylor expansion, the expressions of $\mathcal{E}_0$, $\mathcal{E}_1$, $\mathcal{E}_2$ and $\mathcal{E}_3$ are given below explicitly.
\begin{eqnarray*}
\begin{aligned}
\mathcal{E}_0 &= \frac{\dot{\mu}(t)}{\mu(t)}\mu^{-\frac{n-2}{2}}(t)\left|\eta\right|^{2-n}z_0(\tilde{y})  +\frac{\dot{\mu}(t)}{\mu(t)}\mu^{-\frac{n-2}{2}}(t)\left|\eta\right|^{2-n}z_0(\tilde{y})\left(2\tilde{y}\cdot R_\theta a\right)\\
&\quad - \frac{\dot{\mu}(t)}{\mu(t)}\mu^{-\frac{n-2}{2}}(t)\left|\eta\right|^{2-n}\left(\nabla Q(\tilde{y})\cdot \frac{R_\theta a}{|\eta|^2}\right)\left(\frac{|x-\xi(t)|^2}{\mu^2(t)}\right)\\
&= \frac{\dot{\mu}(t)}{\mu(t)}\mu^{-\frac{n-2}{2}}(t)z_0\left(y\right)\left(1 + \left(y\cdot a\right)F_0(\mu, \xi, a, \theta, y)\right),
\end{aligned}
\end{eqnarray*}
where $f$ are generic smooth bounded functions of the tuple $(\mu, \xi, a, \theta, y)$ which may different from one place to another, $F_0(\mu, \xi, a, \theta, y)$ is a smooth bounded function depending on $(\mu, \xi, a, \theta, y)$. Similarly, we have
\begin{eqnarray*}
\begin{aligned}
\mathcal{E}_1 &= \mu(t)^{-\frac{n-2}{2}}(n-2)|\eta|^{-n}\left(\eta\cdot a\right)\left(\frac{x-\xi}{|x-\xi|}\cdot \frac{\dot{\xi}}{\mu}\right)Q(\tilde{y})\\
&\quad + \mu^{-\frac{n-2}{2}}|\eta|^{2-n}\nabla Q(\tilde{y})\cdot \left[\frac{1}{|\eta|^2}R_\theta\left(\frac{\dot{\xi}}{\mu} - \frac{2 a (x-\xi)\cdot \dot{\xi}}{\mu^2}\right)\right]\\
&\quad + \mu^{-\frac{n-2}{2}}|\eta|^{2-n}\nabla Q(\tilde{y})\cdot \left(\tilde{y}\frac{2\eta}{|\eta|^2}\left( a\left(\frac{x-\xi}{|x-\xi|}\frac{\dot{\xi}}{\mu}\right)\right)\right)\\
& = \mu^{-\frac{n-2}{2}}\nabla Q\left(y\right)\cdot \frac{\dot{\xi}}{\mu(t)}\left(1 + \left(y\cdot a\right)F_1(\mu, \xi, a, \theta, y)\right)
\end{aligned}
\end{eqnarray*}
where $f$ are generic smooth bounded functions of the tuple $(\mu, \xi, a, \theta, y)$ which may different from one place to another, $F_1(\mu, \xi, a, \theta, y)$ is a smooth bounded function depending on $(\mu, \xi, a, \theta, y)$.
Furthermore, $\mathcal{E}_2 = \mathcal{E}_{21} + \mathcal{E}_{22}$, where
\begin{eqnarray*}
\begin{aligned}
\mathcal{E}_{21} &= -\mu^{-\frac{n-2}{2}}|\eta|^{-n}2\left(\dot{a}_1\cdot y\right)\left[\frac{n-2}{2}Q\left(\tilde{y}\right) + \nabla Q(\tilde{y})\cdot\tilde{y}\right]\\
& \quad + \mu^{-\frac{n-2}{2}}|\eta|^{-n}R_\theta\dot{a}_1\cdot\nabla Q(\tilde{y})\left|\frac{x-\xi}{\mu}\right|^2 \\
& \quad + \mu^{-\frac{n-2}{2}}|\eta|^{-n}2\left[\frac{n-2}{2}Q(\tilde{y}) + \nabla Q(\tilde{y})\cdot\tilde{y}\right]\left|\frac{x-\xi}{\mu}\right|^2a_1\dot{a}_1\\
& = \mu^{-\frac{n-2}{2}}\Bigg\{-2\left(\dot{a}_1\cdot y\right)\left[\frac{n-2}{2}Q\left(y\right) + \nabla Q\left(y\right)\cdot y\right] + \dot{a}_1\cdot\nabla Q\left(y\right)\left|y\right|^2\Bigg\}\times\\
&\quad\quad\quad\quad\quad\quad\quad\quad\quad\quad\quad\quad\quad\quad\quad\quad\quad\quad\quad \Bigg(1 + \left(y\cdot a\right)F_{21}(\mu, \xi, a, \theta, y)\Bigg)\\
& = \mu^{-\frac{n-2}{2}}z_{n+2}(y)\dot{a}_1\Bigg(1 + \left(y\cdot a\right)F_{21}(\mu, \xi, a, \theta, y)\Bigg)
\end{aligned}
\end{eqnarray*}
and
\begin{eqnarray*}
\begin{aligned}
\mathcal{E}_{22} &= -\mu^{-\frac{n-2}{2}}|\eta|^{-n}2\left(\dot{a}_2\cdot y\right)\left[\frac{n-2}{2}Q\left(\tilde{y}\right) + \nabla Q(\tilde{y})\cdot\tilde{y}\right]\\
& \quad + \mu^{-\frac{n-2}{2}}|\eta|^{-n}R_\theta\dot{a}_2\cdot\nabla Q(\tilde{y})\left|\frac{x-\xi}{\mu}\right|^2 \\
& \quad + \mu^{-\frac{n-2}{2}}|\eta|^{-n}2\left[\frac{n-2}{2}Q(\tilde{y}) + \nabla Q(\tilde{y})\cdot\tilde{y}\right]\left|\frac{x-\xi}{\mu}\right|^2a_2\dot{a}_2\\
& = \mu^{-\frac{n-2}{2}}\Bigg\{-2\left(\dot{a}_2\cdot y\right)\left[\frac{n-2}{2}Q\left(y\right) + \nabla Q\left(y\right)\cdot y\right] + \dot{a}_2\cdot\nabla Q\left(y\right)\left|y\right|^2\Bigg\}\times\\
&\quad\quad\quad\quad\quad\quad\quad\quad\quad\quad\quad\quad\quad\quad\quad\quad\quad\quad\quad \Bigg(1 + \left(y\cdot a\right)F_{22}(\mu, \xi, a, \theta, y)\Bigg)\\
& = \mu^{-\frac{n-2}{2}}z_{n+3}(y)\dot{a}_2\Bigg(1 + \left(y\cdot a\right)F_{22}(\mu, \xi, a, \theta, y)\Bigg).
\end{aligned}
\end{eqnarray*}
Here we identify the component $a_j$, $j=1,2$ of $a$ with the vector
$$
\begin{pmatrix}a_1\\ 0\\ 0\\ \cdots\\ 0\end{pmatrix}\in \mathbb{R}^n, \quad {\mbox {if}} \quad j=1, \quad \begin{pmatrix}0\\ a_2\\ 0\\ \cdots\\ 0\end{pmatrix}\in \mathbb{R}^n, \quad {\mbox {if}} \quad j=1,
$$
$f$ are generic smooth bounded functions of the tuple $(\mu, \xi, a, \theta, y)$ which may different from one place to another, $F_{21}(\mu, \xi, a, \theta, y)$ and $F_{22}(\mu, \xi, a, \theta, y)$ are smooth bounded functions depending on $(\mu, \xi, a, \theta, y)$.
Finally, $\mathcal{E}_{3} = \mathcal{E}_{3, 12} + \sum_{j=3}^n\mathcal{E}_{3, 1j} + \sum_{j=3}^n\mathcal{E}_{3, 2j}$, where
\begin{equation*}
\begin{aligned}
&\mathcal{E}_{3, 12}= \mu^{-\frac{n-2}{2}}|\eta|^{2-n}\nabla Q(\tilde{y})\cdot (i\tilde{y})\dot{\theta}_{12}\\
& = \mu^{-\frac{n-2}{2}}z_{n+1}(y)\dot{\theta}_{12}\left(1 + \left(y\cdot R_\theta a\right)F_{3,21}(\mu, \xi, a, \theta, y)\right)
\end{aligned}
\end{equation*}
and similarly, for $j=3, \cdots, n$,
\begin{equation*}
\mathcal{E}_{3, 1j}= \mu^{-\frac{n-2}{2}}z_{n+j+1}(y)\dot{\theta}_{1j}\left(1 + \left(y\cdot R_\theta a\right)F_{3,1j}(\mu, \xi, a, \theta, y)\right),
\end{equation*}
\begin{equation*}
\mathcal{E}_{3, 2j}= \mu^{-\frac{n-2}{2}}z_{2n+l-1}(y)\dot{\theta}_{2j}\left(1 + \left(y\cdot R_\theta a\right)F_{3,2j}(\mu, \xi, a, \theta, y)\right),
\end{equation*}
where $i$ is the rotation matrix with angle $\frac{\pi}{2}$ around the axes $x_1$, $x_2$ in $\mathcal{E}_{3, 12}$, around the axes $x_1$, $x_j$ in $\mathcal{E}_{3, 1j}$ and around the axes $x_2$, $x_j$ in $\mathcal{E}_{3, 2j}$ respectively, $F_{3, 12}(\mu, \xi, a, \theta, y)$, $F_{3, 1j}(\mu, \xi, a, \theta, y)$ and $F_{3, 2j}(\mu, \xi, a, \theta, y)$, $j = 3, \cdots, n$, are smooth bounded functions depending on $(\mu, \xi, a, \theta, y)$.

To perform the gluing method, the terms $\mu^{-\frac{n-2}{2}-1}\dot{\mu} z_{0}(y)$, $\mu^{-\frac{n-2}{2}-1}\dot{\xi}\cdot \nabla Q(y)$ and $\mu^{-\frac{n-2}{2}-1}\nabla Q(y)\cdot \left(iR_\theta\xi\right)\dot\theta$ do not have enough decay, inspired by \cite{davila2017singularity}, we should add nonlocal terms to cancel them out at main order. By the detailed construction of $Q$ (see \cite{delpinomussofrankpistoiajde2011}) and (\ref{vanc1}) we know that the main order of $z_{0}(y)$ is
$$
\frac{D_{n,k}(2-|y|^2)}{\left(1+|y|^2\right)^{\frac{n}{2}}}
$$
with $D_{n,k} = -\frac{n-2}{2}\left[ {4\over n(n-2) } \right]^{n-2 \over 4} \, 2^{n-2 \over 2} \left( 1+ d_k \right)$.
Therefore, we consider the following heat equation
\begin{equation}\label{2018320}
-\varphi_t + \Delta\varphi + \frac{\dot{\mu}}{\mu}\mu^{-(n-2)}\frac{D_{n,k}\left(2-\left|\frac{x - \xi}{\mu}\right|^2\right)}{\left(1+\left|\frac{x - \xi}{\mu}\right|^2\right)^{\frac{n}{2}}} = 0 \text{ in }\mathbb{R}^n\times (t_0, +\infty).
\end{equation}
By the Duhamel's principle, we known
\begin{equation}\label{nonlocalterm1}
\Phi^0(x,t) = -\int_{t_0}^{t}\int_{\mathbb{R}^n}p(t-\tilde{s}, x-y)\frac{\dot{\mu}(\tilde{s})}{\mu(\tilde{s})}\mu^{-(n-2)}(\tilde{s})\frac{D_{n,k}\left(2-\left|\frac{y - \xi(\tilde{s})}{\mu(\tilde{s})}\right|^2\right)}{\left(1+\left|\frac{y - \xi(\tilde{s})}{\mu(\tilde{s})}\right|^2\right)^{\frac{n}{2}}}dyd\tilde{s}
\end{equation}
provides a bounded solution for (\ref{2018320}). Here $p(t, x) = \frac{1}{(4\pi t)^{\frac{n}{2}}}e^{\frac{-|x|^2}{4t}}$ is the standard heat kernel for the heat operator $-\frac{\partial }{\partial t} + \Delta$ on $\mathbb{R}^n\times (t_0, +\infty)$. By the super-sub solution argument, $\Phi^0(x, t)$ satisfies the estimate $\Phi^0(x, t)\sim  \frac{\dot{\mu}}{\mu}\frac{\mu^{-n+4}}{1+|y|^{n-4}}$ (see Lemma \ref{l4:lemma4.1}).

To cancel the main order $\mu^{-\frac{n-2}{2}-1}\dot{\xi}\cdot \frac{E_{n,k}y}{(1+|y|^2)^{\frac{n}{2}}}$ of $\mu^{-\frac{n-2}{2}-1}\dot{\xi}\cdot \nabla Q(y)$ where $E_{n, k}$ is a constant depending on $n$ and $k$, for $y = \frac{x-\xi}{\mu}$, we consider the following heat equation
\begin{equation}\label{20183201}
-\varphi_t + \Delta\varphi + E_{n,k}\mu^{-(n-2)}\frac{1}{\left(1+|y|^2\right)^{\frac{n}{2}}}\frac{\dot{\xi}}{\mu}\cdot y = 0 \text{ in }\mathbb{R}^n\times (t_0, +\infty).
\end{equation}
The solution defined from the Duhamel's principle
\begin{equation*}
\begin{aligned}
\Phi^1(x,t) = - E_{n,k}\int_{t_0}^{t}\int_{\mathbb{R}^n}p(t-\tilde{s}, x-y)\mu^{-(n-2)}(\tilde{s})&\frac{\dot{\xi}(\tilde{s})\cdot \frac{y-\xi(\tilde{s})}{\mu(\tilde{s})}}{\mu(\tilde{s})}\times\\
&\quad\quad\quad\quad \frac{1}{\left(1+\left|\frac{y-\xi(\tilde{s})}{\mu(\tilde{s})}\right|^2\right)^{\frac{n}{2}}}dyd\tilde{s}
\end{aligned}
\end{equation*}
satisfies the estimate
$\Phi^1(x, t) \sim  \frac{|\dot{\xi}|}{\mu}\frac{\mu^{-n+4}}{1+|y|^{n-3}}$.

Similarly, for $i = 1, 2$,
we consider the heat equation
\begin{equation}\label{201832010}
-\varphi_t + \Delta\varphi + \mu^{-(n-2)}\frac{E_{n,k}|y|^2 - 2D_{n,k}(2-|y|^2)}{\left(1+|y|^2\right)^{\frac{n}{2}}}\dot{a}_i y_i = 0 \text{ in }\mathbb{R}^n\times (t_0, +\infty),
\end{equation}
which has a bounded solution given by
\begin{equation*}
\begin{aligned}
\Phi^{2,i}(x,t) = -\int_{t_0}^{t}\int_{\mathbb{R}^n}p(t-\tilde{s}, & x-y)\mu^{-(n-2)}(\tilde{s}) \dot{a}_i(\tilde{s}) \left(\frac{y-\xi(\tilde{s})}{\mu(\tilde{s})}\right)_i\times\\
&\quad\quad\quad\quad\frac{E_{n,k}\left|\frac{y-\xi(\tilde{s})}{\mu(\tilde{s})}\right|^2 - 2D_{n,k}\left(2-\left|\frac{y-\xi(\tilde{s})}{\mu(\tilde{s})}\right|^2\right)}
{\left(1+\left|\frac{y-\xi(\tilde{s})}{\mu(\tilde{s})}\right|^2\right)^{\frac{n}{2}}}dyd\tilde{s}
\end{aligned}
\end{equation*}
satisfies the estimate
$\Phi^{2, i}(x, t) \sim |\dot{a}_i|\frac{\mu^{-n+4}}{1+|y|^{n-5}}$.

Now we define $\Phi^*(x, t) = \Phi^0(x, t) + \Phi^1(x, t) + \sum_{i = 1}^2\Phi^{2, i}(x, t)$.
Since the final solution must sasify $u = 0$ in $\partial\Omega$, a better approximation than $Q_A(x, t)$ should be
\begin{equation}\label{e2:3}
u_{A}(x, t) = Q_A(x, t) + \mu^{\frac{n-2}{2}}\Phi^*(x, t) - \mu^{\frac{n-2}{2}} H(x, q).
\end{equation}
The error of $u_{A}$ can be computed as follows,
\begin{equation}\label{e2:3333}
S(u_{A}) = -\partial_t u_{A} + |u_{A}|^{p-1}u_{A} - |Q_A|^{p-1}Q_A + \mu^{\frac{n-2}{2}}\Delta\Phi^*(x, t).
\end{equation}
\subsection{The error $S(u_{A})$.}
Near the given point $q$, the following expansion holds.
\begin{lemma}\label{l2.1}
Consider the region $|x-q|\leq \varepsilon$ for $\varepsilon$ small enough, we have
\begin{equation*}
S(u_{A}) = \mu^{-\frac{n+2}{2}}(\mu E_{0} + \mu E_{1} + \mu E_{2} + \mu E_{3} + \mathcal{R})
\end{equation*}
with
\begin{equation*}
\begin{aligned}
E_{0}&= p|Q|^{p-1}\left[-\mu^{n-3} H(q, q) + \mu^{n-3}\Phi^0(q, t)\right]  +\dot{\mu}(t)\left(z_0(y) - \frac{D_{n,k}(2-|y|^2)}{\left(1+|y|^2\right)^{\frac{n}{2}}}\right),
\end{aligned}
\end{equation*}
\begin{equation*}
\begin{aligned}
E_{1} &= p|Q|^{p-1}\left[-\mu^{n-2}\nabla H(q, q)\right]\cdot y + p|Q|^{p-1}\left[\mu^{n-3}\Phi^1(q, t)\right]\\
  &\quad + \left(\nabla Q(y) - \frac{E_{n,k}y}{\left(1+|y|^2\right)^{\frac{n}{2}}}\right)\cdot \dot{\xi},
\end{aligned}
\end{equation*}
\begin{equation*}
\begin{aligned}
& E_{2} = p|Q|^{p-1}\left[\mu^{n-3}\Phi^{2, 1}(q, t) + \mu^{n-3}\Phi^{2, 2}(q, t)\right] \\
&\quad + \mu(t)\dot{a}_1 \left(-2y_1\left(z_0(y) - \frac{D_{n,k}(2-|y|^2)}{\left(1+|y|^2\right)^{\frac{n}{2}}}\right) + |y|^2\left(\frac{\partial}{\partial y_1} Q(y) - \frac{E_{n,k}y_1}{\left(1+|y|^2\right)^{\frac{n}{2}}}\right)\right)\\
&\quad + \mu(t)\dot{a}_2 \left(-2y_2\left(z_0(y) - \frac{D_{n,k}(2-|y|^2)}{\left(1+|y|^2\right)^{\frac{n}{2}}}\right) + |y|^2\left(\frac{\partial}{\partial y_2} Q(y) - \frac{E_{n,k}y_2}{\left(1+|y|^2\right)^{\frac{n}{2}}}\right)\right),
\end{aligned}
\end{equation*}
\begin{equation*}
\begin{aligned}
E_{3} =z_{n+1}(y)\mu\dot{\theta}_{12} + \sum_{j = 3}^n\left(z_{n+j+1}(y)\mu\dot{\theta}_{1j} + z_{2n+j-1}(y)\mu\dot{\theta}_{2j}\right)
\end{aligned}
\end{equation*}
and
\begin{equation*}
\mathcal{R} = (\mu_0^{n+2} + \mu_0^{n-1}\dot{\mu})f + \frac{\mu_0^{n-1}\vec{f}}{1+|y|^{2}}\cdot a+\frac{\mu_0^{n-2}\vec{g}}{1+|y|^{4}}\cdot(\xi-q) + \mu_0^{n}\dot{\xi}\cdot\vec{h},
\end{equation*}
where $f$, $\vec{f}$, $\vec{g}$ and $\vec{h}$ are smooth and bounded functions depending on the tuple of variables $(\mu_0^{-1}\mu, \xi, a, \theta, x-\xi)$.
\end{lemma}
\begin{proof}
Set
$$\tilde{y} = \frac{R_\theta\left(\frac{x-\xi(t)}{\mu(t)} - a\left|\frac{x-\xi(t)}{\mu(t)}\right|^2\right)}{|\eta|^2},$$
we have
\begin{equation*}
u_{A}(x, t) = \mu(t)^{-\frac{n-2}{2}}\left|\eta\right|^{2-n}Q\left(\tilde{y}\right) + \mu^{\frac{n-2}{2}}\Phi^*(x, t) - \mu^{\frac{n-2}{2}} H(x, q)
\end{equation*}
and
\begin{equation*}
S(u_{A}) = S_1 + S_2,
\end{equation*}
where
\begin{equation*}
\begin{aligned}
S_1 :=&\mathcal{E}_0 + \mathcal{E}_1 + \mathcal{E}_2 + \mathcal{E}_3 + \frac{n-2}{2}\mu^{\frac{n-4}{2}}\dot{\mu}H(x, q)\\
 &- \frac{n-2}{2}\mu^{\frac{n-4}{2}}\dot{\mu}\Phi^*(x, t) - \mu^{\frac{n-2}{2}}\partial_t\Phi^*(x, t),
\end{aligned}
\end{equation*}
\begin{equation*}
\begin{aligned}
S_2:=&\left|\mu(t)^{-\frac{n-2}{2}}\left|\eta\right|^{2-n}Q\left(\tilde{y}\right) + \mu^{\frac{n-2}{2}}\Phi^*(x, t) - \mu^{\frac{n-2}{2}} H(x, q)\right|^{p-1}\\
&\left(\mu(t)^{-\frac{n-2}{2}}\left|\eta\right|^{2-n}Q\left(\tilde{y}\right) + \mu^{\frac{n-2}{2}}\Phi^*(x, t) - \mu^{\frac{n-2}{2}} H(x, q)\right)\\ &-\mu(t)^{-\frac{n+2}{2}}\left|\eta\right|^{-2-n}\left|Q\left(\tilde{y}\right)\right|^{p-1}Q\left(\tilde{y}\right) + \mu^{\frac{n-2}{2}}\Delta\Phi^*(x, t).
\end{aligned}
\end{equation*}
Let
\begin{equation*}
\begin{aligned}
S_{2} = &\mu^{-\frac{n+2}{2}}\left|\eta\right|^{-2-n}\left[\left|Q(\tilde{y}) + \Theta\right|^{p-1}\left(Q(\tilde{y}) + \Theta\right) - \left|Q(\tilde{y})\right|^{p-1}Q(\tilde{y})\right],
\end{aligned}
\end{equation*}
and
\begin{equation}\label{e:2018317}
\begin{aligned}
\Theta =&  \mu^{n-2}|\eta|^{n-2}\Phi^*(x, t) -\mu^{n-2}|\eta|^{n-2}H(x, q).
\end{aligned}
\end{equation}
Observe that $|\Theta|\lesssim \mu_0^{n-2}$ when $\varepsilon$ is small enough, we may assume $Q(y)^{-1}|\Theta|< \frac{1}{2}$ in the considered region $|x-q| < \varepsilon$. Using Taylor's expansion, we obtain the following
\begin{equation*}
\begin{aligned}
S_{2} = &\mu^{-\frac{n+2}{2}}\left|\eta\right|^{-2-n}\left[p\left|Q(\tilde{y})\right|^{p-1}\Theta + p(p-1)\int_{0}^1(1-s)\left|Q(\tilde{y}) + s\Theta\right|^{p-2}ds\Theta^2\right].
\end{aligned}
\end{equation*}
Hence we have
\begin{equation*}
\begin{aligned}
\Theta = \mu^{n-2}|\eta|^{n-2}\Phi^*((|\eta|^2 R_{-\theta}\tilde{y} & +  a |y|^2)\mu + \xi, t)\\
 &- \mu^{n-2}|\eta|^{n-2}H((|\eta|^2 R_{-\theta}\tilde{y} +  a |y|^2)\mu + \xi, q).
\end{aligned}
\end{equation*}
We further expand as
\begin{equation*}
\begin{aligned}
\Theta =& -\mu^{n-2}|\eta|^{n-2}\left(H(q, q) - \Phi^*(q, t)\right)\\
& + ((|\eta|^2 R_{-\theta}\tilde{y} +  a |y|^2)\mu + \xi - q)\cdot\left[-\mu^{n-2}|\eta|^{n-2}\nabla \left(H(q, q) - \Phi^*(q, t)\right)\right]\\
& +\int_{0}^1\Big\{-\mu^{n-2}|\eta|^{n-2}D_x^2H(q + s((|\eta|^2 R_{-\theta}\tilde{y} +  a |y|^2)\mu + \xi - q), q)\Big\}\\
&\quad\quad\quad\quad\quad\quad\quad\quad\quad\quad\quad\quad\quad\quad[(|\eta|^2 R_{-\theta}\tilde{y} +  a |y|^2)\mu + \xi - q]^2(1-s)ds\\
& +\int_{0}^1\Big\{\mu^{n-2}|\eta|^{n-2}D_x^2\Phi^*(q + s((|\eta|^2 R_{-\theta}\tilde{y} +  a |y|^2)\mu + \xi - q), t)\Big\}\\
&\quad\quad\quad\quad\quad\quad\quad\quad\quad\quad\quad\quad\quad\quad[(|\eta|^2 R_{-\theta}\tilde{y} +  a |y|^2)\mu + \xi - q]^2(1-s)ds.
\end{aligned}
\end{equation*}
Therefore, we have
\begin{equation*}
\begin{aligned}
\Theta &= -\mu^{n-2}|\eta|^{n-2}H(q, q)-\mu^{n-1}|\eta|^{n}\nabla H(q, q)\cdot R_{-\theta}\tilde{y} \\
&\quad - \mu^{n-2}|\eta|^{n-2}\nabla H(q, q)\cdot (\xi - q) -\mu^{n-1}|\eta|^{n-2}\nabla H(q, q)\cdot a |y|^2\\
&\quad + \mu^{n-2}|\eta|^{n-2}\Phi^*(q, t)+\mu^{n-1}|\eta|^{n}\nabla \Phi^*(q, t)\cdot R_{-\theta}\tilde{y} \\
&\quad + \mu^{n-2}|\eta|^{n-2}\nabla \Phi^*(q, t)\cdot (\xi - q) -\mu^{n-1}|\eta|^{n-2}\nabla \Phi^*(q, t)\cdot a |y|^2\\
&\quad + \mu_0^{n}F(\mu_0^{-1}\mu, \xi, a, \theta, x-\xi)\\
& = -\mu^{n-2}\left(1 - 2 a\cdot \frac{x-\xi}{\mu(t)} + |a|^2\frac{|x-\xi|^2}{\mu^2(t)}\right)^{\frac{n-2}{2}}H(q, q)\\
&\quad -\mu^{n-1}\left(1 - 2 a\cdot \frac{x-\xi}{\mu(t)} + |a|^2\frac{|x-\xi|^2}{\mu^2(t)}\right)^{\frac{n}{2}}\nabla H(q, q)\cdot R_{-\theta}\tilde{y}\\
\end{aligned}
\end{equation*}
\begin{equation*}
\begin{aligned}
&\quad -\mu^{n-2}\left(1 - 2 a\cdot \frac{x-\xi}{\mu(t)} + |a|^2\frac{|x-\xi|^2}{\mu^2(t)}\right)^{\frac{n-2}{2}}\nabla H(q, q)\cdot (\xi - q)\\
&\quad -\mu^{n-1}\left(1 - 2 a\cdot \frac{x-\xi}{\mu(t)} + |a|^2\frac{|x-\xi|^2}{\mu^2(t)}\right)^{\frac{n-2}{2}}\nabla H(q, q)\cdot a |y|^2\\
&\quad +\mu^{n-2}\left(1 - 2 a\cdot \frac{x-\xi}{\mu(t)} + |a|^2\frac{|x-\xi|^2}{\mu^2(t)}\right)^{\frac{n-2}{2}}\Phi^*(q, t)\\
&\quad +\mu^{n-1}\left(1 - 2 a\cdot \frac{x-\xi}{\mu(t)} + |a|^2\frac{|x-\xi|^2}{\mu^2(t)}\right)^{\frac{n}{2}}\nabla \Phi^*(q, t)\cdot R_{-\theta}\tilde{y}\\
&\quad +\mu^{n-2}\left(1 - 2 a\cdot \frac{x-\xi}{\mu(t)} + |a|^2\frac{|x-\xi|^2}{\mu^2(t)}\right)^{\frac{n-2}{2}}\nabla \Phi^*(q, t)\cdot (\xi - q)\\
&\quad +\mu^{n-1}\left(1 - 2 a\cdot \frac{x-\xi}{\mu(t)} + |a|^2\frac{|x-\xi|^2}{\mu^2(t)}\right)^{\frac{n-2}{2}}\nabla \Phi^*(q, t)\cdot a |y|^2\\
&\quad + \mu_0^{n}F(\mu_0^{-1}\mu, \xi, a, \theta, x-\xi)\\
& = -\mu^{n-2}\left(1 + O(|a||y|)\right)H(q, q)\\
&\quad -\mu^{n-1}\left(1 + O(|a||y|)\right)\nabla H(q, q)\cdot y\left(1 + O(|a||y|)\right)\\
&\quad -\mu^{n-2}\left(1 + O(|a||y|)\right)\nabla H(q, q)\cdot (\xi - q)\\
&\quad -\mu^{n-1}\left(1 + O(|a||y|)\right)\nabla H(q, q)\cdot a |y|^2\\
&\quad +\mu^{n-2}\left(1 + O(|a||y|)\right)\Phi^*(q, t)\\
&\quad +\mu^{n-1}\left(1 + O(|a||y|)\right)\nabla \Phi^*(q, t)\cdot y\left(1 + O(|a||y|)\right)\\
&\quad +\mu^{n-2}\left(1 + O(|a||y|)\right)\nabla \Phi^*(q, t)\cdot (\xi - q)\\
&\quad +\mu^{n-1}\left(1 + O(|a||y|)\right)\nabla \Phi^*(q, t)\cdot a |y|^2\\
&\quad + \mu_0^{n}F(\mu_0^{-1}\mu, \xi, a, \theta, x-\xi)\\
&= -\mu^{n-2}H(q, q)-\mu^{n-1}\nabla H(q, q)\cdot y -\mu^{n-2}\nabla H(q, q)\cdot (\xi - q)\\
&\quad -\mu^{n-1}\nabla H(q, q)\cdot a |y|^2+\mu^{n-2}\Phi^*(q, t)\\
&\quad +\mu^{n-1}\nabla \Phi^*(q, t)\cdot y+\mu^{n-2}\nabla \Phi^*(q, t)\cdot (\xi - q) +\mu^{n-1}\nabla \Phi^*(q, t)\cdot a |y|^2\\
&\quad + \mu_0^{n}F(\mu_0^{-1}\mu, \xi, a, \theta, x-\xi) + \mu_0^{n-2}|a||y|F(\mu_0^{-1}\mu, \xi, a, \theta, x-\xi)\\
\end{aligned}
\end{equation*}
and
\begin{equation*}
\begin{aligned}
& p\left|Q(\tilde{y})\right|^{p-1}\Theta\\
&\quad = p\left|Q\left(R_\theta y + a\left|y\right|^2 + O\left(|a|^2|y|^3\right)\right)\right|^{p-1}\Theta\\
&\quad = p\left|Q(y) + \nabla Q(y)\cdot \left(a\left|y\right|^2 + \left(R_\theta y - y\right)\right) + O(|a|^2|y|^2)\right|^{p-1}\Theta\\
&\quad = p\left(\left|Q\right|^{p-1}(y) + O(|a||y|)\right)\Theta\\
&\quad = p\left(\left|Q\right|^{p-1}(y) + O(|a||y|)\right)\Bigg(-\mu^{n-2}H(q, q)-\mu^{n-1}\nabla H(q, q)\cdot y\\
&\quad\quad -\mu^{n-2}\nabla H(q, q)\cdot (\xi - q) -\mu^{n-1}\nabla H(q, q)\cdot a |y|^2\\
\end{aligned}
\end{equation*}
\begin{equation*}
\begin{aligned}
&\quad\quad +\mu^{n-2}\Phi^*(q, t)+\mu^{n-1}\nabla \Phi^*(q, t)\cdot y+\mu^{n-2}\nabla \Phi^*(q, t)\cdot (\xi - q)\\
&\quad\quad +\mu^{n-1}\nabla \Phi^*(q, t)\cdot a |y|^2 + \mu_0^{n}F(\mu_0^{-1}\mu, \xi, a, \theta, x-\xi)\\
&\quad\quad + \mu_0^{n-2}|a||y|F(\mu_0^{-1}\mu, \xi, a, \theta, x-\xi)\Bigg)\\
&\quad  = p\left|Q\right|^{p-1}(y)\Bigg(-\mu^{n-2}H(q, q)-\mu^{n-1}\nabla H(q, q)\cdot y\\
&\quad\quad -\mu^{n-2}\nabla H(q, q)\cdot (\xi - q) -\mu^{n-1}\nabla H(q, q)\cdot a |y|^2\\
&\quad\quad +\mu^{n-2}\Phi^*(q, t)+\mu^{n-1}\nabla \Phi^*(q, t)\cdot y+\mu^{n-2}\nabla \Phi^*(q, t)\cdot (\xi - q)\\
&\quad\quad +\mu^{n-1}\nabla \Phi^*(q, t)\cdot a |y|^2\Bigg) + \frac{\mu_0^{n-2}|a||y|}{1+|y|^4}F(\mu_0^{-1}\mu, \xi, a, \theta, x-\xi),
\end{aligned}
\end{equation*}
where the smooth functions $F$ are bounded in its arguments which may different from line to line.

Decompose $S_1$ as $S_1 = S_{11} + S_{12}$, where
\begin{equation*}
S_{11} :=\mathcal{E}_0 + \mathcal{E}_1 + \mathcal{E}_2 + \mathcal{E}_3 - \mu^{\frac{n-2}{2}}\partial_t\Phi^*(x, t),
\end{equation*}
\begin{equation*}
\begin{aligned}
S_{12} :=& \frac{n-2}{2}\mu^{\frac{n-4}{2}}\dot{\mu}H(x, q) - \frac{n-2}{2}\mu^{\frac{n-4}{2}}\dot{\mu}\Phi^*(x, t).
\end{aligned}
\end{equation*}
Observe that
\begin{equation*}
S_{12} = \mu_0^{\frac{n-2}{2}-1}\dot{\mu}F(\mu_0^{-1}\mu, \xi, a, \theta, x-\xi)
\end{equation*}
holds for a function $F$ smooth and bounded in their arguments. This proves the lemma.
\end{proof}

Recall that we are trying to find a solution with form
\begin{equation*}
u(x, t) = u_{A}(x,t) + \tilde{\phi}(x, t),
\end{equation*}
where $\tilde{\phi}$ is a small term compared with $u_{A}(x, t)$. By the relation $S(u_{A}+\tilde{\phi})=0$, the main equation can be written as
\begin{equation}\label{e2:27}
 -\partial_t\tilde{\phi} + \Delta\tilde{\phi} + p\left|u_{A}\right|^{p-1}\tilde{\phi} + S(u_{A}) + \tilde{N}_{A}(\tilde{\phi}),
\end{equation}
where
\begin{equation}\label{e2:28}
\tilde{N}_{A}(\tilde{\phi}) = \left|u_{A} + \tilde{\phi}\right|^{p-1}(u_{A} + \tilde{\phi})-\left|u_{A}\right|^{p-1}(u_{A} + \tilde{\phi}) - p\left|u_{A}\right|^{p-1}\tilde{\phi}.
\end{equation}
Note that around $q$, it is more convenient to use the self-similar form, so we write $\tilde{\phi}(x, t)$ as
\begin{equation}\label{e2:29}
\tilde{\phi}(x, t) = \mu(t)^{-\frac{n-2}{2}}\phi\left(\frac{x-\xi(t)}{\mu(t)}\right).
\end{equation}

\subsection{Improvement of the approximation.}
The largest term in the expansion for $\mu^{\frac{n+2}{2}}S(u_{A})$ is $\mu E_{0}$. To improve the approximation error near the point $q$, $\phi(y, t)$ should be the solution of the elliptic equation (at main order)
\begin{equation}\label{e2:31}
\Delta_y\phi_{0} + p|Q|^{p-1}(y)\phi_{0} = -\mu_{0}E_{0}\quad\text{in }\mathbb{R}^n,\quad \phi_0(y, t)\to 0\quad\text{as }|y|\to \infty.
\end{equation}
Equation (\ref{e2:31}) is an elliptic equation of form
\begin{equation}\label{e2:32}
L[\psi]:= \Delta_y\psi + p|Q|^{p-1}(y)\psi = h(y)\quad\text{in }\mathbb{R}^n,\quad \psi(y)\to 0\quad\text{as }|y|\to \infty.
\end{equation}
By the nondegeneracy of the basic cell $Q$ (see \cite{MussoWei2015}), we know that each bounded solution of $L[\psi] = 0$ in $\mathbb{R}^n$ is contained in the space
$$\mbox{span}\{z_0,\cdots, z_{3n-1}\}.$$
Standard elliptic theory tells us that problem (\ref{e2:32}) is solvable for $h(y) = O(|y|^{-m})$, $m > 2$, if and only if the $L^2$ orthogonal identities
\begin{equation*}
\int_{\mathbb{R}^n}h(y)z_i(y)dy = 0\quad\text{for all }i = 0,\cdots, 3n-1
\end{equation*}
hold.

For (\ref{e2:31}), we first consider the following condition,
\begin{equation}\label{e2:35}
\int_{\mathbb{R}^n}\mu^{\frac{n+2}{2}}S(u_{A})(y, t)z_{0}(y)dy = 0.
\end{equation}
We claim that, for suitable positive constant $b$ and a positive constant $c_n$ depending only on $n$, choosing $\mu = b\mu_0(t)$ , $\mu_0(t) = c_{n} t^{-\frac{1}{n-4}}$, (\ref{e2:35}) can be achieved at main order. Observe that $\dot{\mu}_0(t) = -\frac{1}{(n-4)c_{n}^{n-4}}\mu_0^{n-3}(t)$ and the main contribution to the left of (\ref{e2:35}) comes from the following term
\begin{equation*}
\begin{aligned}
E_{0j}= p|Q|^{p-1}\left[\mu^{n-3}\left(\Phi^0(q, t) - H(q, q)\right)\right]  +\dot{\mu}(t)\left(z_0(y) - \frac{D_{n,k}(2-|y|^2)}{\left(1+|y|^2\right)^{\frac{n}{2}}}\right).
\end{aligned}
\end{equation*}
Now let us compute the term $\Phi^0(q, t)$ which is given by (\ref{nonlocalterm1}). Note that the heat kernel function $p(t, x) = \frac{1}{(4\pi t)^{\frac{n}{2}}}e^{\frac{-|x|^2}{4t}}$ satisfies the following transformation law
$$p(t-\tilde{s},q-y)=(t-\tilde{s})^{-\frac{n}{2}}p\left(1,\frac{|q-y|}{(t-\tilde{s})^{\frac{1}{2}}}\right),$$
therefore we have
\begin{equation*}
\begin{aligned}
&\Phi^0(q, t)= -\int_{t_0}^{t}\int_{\mathbb{R}^n}p(t-\tilde{s}, q-y)\frac{\dot{\mu}(\tilde{s})}{\mu(\tilde{s})}\frac{\mu^{-(n-2)}(\tilde{s})D_{n,k}\left(2-\left|\frac{y-\xi(\tilde{s})}{\mu(\tilde{s})}\right|^2\right)}
{\left(1+\left|\frac{y-R_{\theta(\tilde{s})}\xi(\tilde{s})}{\mu(\tilde{s})}\right|^2\right)^{\frac{n}{2}}}dyd\tilde{s}\\
&= -(1+o(1))\int_{t_0}^{t}\int_{\mathbb{R}^n}p(t-\tilde{s}, q-y)\frac{\dot{\mu}(\tilde{s})}{\mu(\tilde{s})}\frac{\mu^{-(n-2)}(\tilde{s})D_{n,k}\left(2-\left|\frac{y-q}{\mu(\tilde{s})}\right|^2\right)}
{\left(1+\left|\frac{y-q}{\mu(\tilde{s})}\right|^2\right)^{\frac{n}{2}}}dyd\tilde{s}\\
&= -(1+o(1))\int_{t_0}^{t}\frac{1}{(t-\tilde{s})^{\frac{n}{2}}}d\tilde{s}\int_{\mathbb{R}^n}p\left(1, \frac{q-y}{(t-\tilde{s})^{\frac{1}{2}}}\right)\frac{\dot{\mu}(\tilde{s})}{\mu(\tilde{s})}\\
&\quad\quad\quad\quad\quad\quad\quad\quad \times\frac{\mu^{-(n-2)}(\tilde{s})\left(t-\tilde{s}\right)^{\frac{n}{2}}D_{n,k}\left(2-\left|\frac{(t-\tilde{s})^{\frac{1}{2}}}{\mu(\tilde{s})}
\frac{q-y}{(t-\tilde{s})^{\frac{1}{2}}}\right|^2\right)}
{\left(1+\left|\frac{(t-\tilde{s})^{\frac{1}{2}}}{\mu(\tilde{s})}\frac{q-y}{(t-\tilde{s})^{\frac{1}{2}}}\right|^2\right)^{\frac{n}{2}}
}d\frac{y-q_j}{(t-\tilde{s})^{\frac{1}{2}}}\\
\end{aligned}
\end{equation*}
\begin{equation*}
\begin{aligned}
&= -(1+o(1))\int_{t_0}^{t}\frac{\dot{\mu}(\tilde{s})}{\mu(\tilde{s})}\mu^{-(n-2)}(\tilde{s})d\tilde{s}\int_{\mathbb{R}^n}p\left(1, \frac{q-y}{(t-\tilde{s})^{\frac{1}{2}}}\right)\\
&\quad\quad\quad\quad\quad\quad\quad\quad\quad\quad\quad\quad\quad\quad\quad\quad\times\frac{D_{n,k}\left(2-\left|\frac{(t-\tilde{s})^{\frac{1}{2}}}{\mu(\tilde{s})}
\frac{q-y}{(t-\tilde{s})^{\frac{1}{2}}}\right|^2\right)}{\left(1+\left|\frac{(t-\tilde{s})^{\frac{1}{2}}}{\mu(\tilde{s})}
\frac{q-y}{(t-\tilde{s})^{\frac{1}{2}}}\right|^2\right)^{\frac{n}{2}}}
d\frac{y-q}{(t-\tilde{s})^{\frac{1}{2}}}\\
&= -(1+o(1))\int_{t_0}^{t}\frac{\dot{\mu}(\tilde{s})}{\mu(\tilde{s})}\mu^{-(n-2)}(\tilde{s})F\left(\frac{(t-\tilde{s})^{\frac{1}{2}}}{\mu(\tilde{s})}\right)d\tilde{s},
\end{aligned}
\end{equation*}
with
\begin{equation*}
F(a) = \int_{\mathbb{R}^n}p\left(1, x\right)\frac{D_{n,k}\left(2-a^2|x|^2\right)}
{\left(1+a^2|x|^2\right)^{\frac{n}{2}}}dx.
\end{equation*}
We claim that, for a suitable positive constant $c$ depending on $n$ and $b$, it holds that
\begin{equation}\label{e:201803312}
\Phi^0(q, t) = -(1+o(1))\int_{t_0}^{t}\frac{\dot{\mu}(\tilde{s})}{\mu(\tilde{s})}\mu^{-(n-2)}(\tilde{s})F\left(\frac{(t-\tilde{s})^{\frac{1}{2}}}{\mu(\tilde{s})}\right)d\tilde{s} = c(1+o(1)).
\end{equation}

Indeed, for a small positive constant $\delta$, decompose the integral $$\int_{t_0}^{t}\frac{\dot{\mu}(\tilde{s})}{\mu(\tilde{s})}\mu^{-(n-2)}(\tilde{s})F\left(\frac{(t-\tilde{s})^{\frac{1}{2}}}{\mu(\tilde{s})}\right)d\tilde{s}$$ as
\begin{equation*}
\begin{aligned}
&\int_{t_0}^{t}\frac{\dot{\mu}(\tilde{s})}{\mu(\tilde{s})}\mu^{-(n-2)}(\tilde{s})F\left(\frac{(t-\tilde{s})^{\frac{1}{2}}}{\mu(\tilde{s})}\right)d\tilde{s}\\
&=  \int_{t_0}^{t-\delta}\frac{\dot{\mu}(\tilde{s})}{\mu(\tilde{s})}\mu^{-(n-2)}(\tilde{s})F\left(\frac{(t-\tilde{s})^{\frac{1}{2}}}{\mu(\tilde{s})}\right)d\tilde{s}\\
&\quad + \int_{t-\delta}^{t}\frac{\dot{\mu}(\tilde{s})}{\mu(\tilde{s})}\mu^{-(n-2)}(\tilde{s})F\left(\frac{(t-\tilde{s})^{\frac{1}{2}}}{\mu(\tilde{s})}\right)d\tilde{s}\\
& := I_1 + I_2.
\end{aligned}
\end{equation*}
For $I_1$, we have $t-\tilde{s} >\delta$, therefore
\begin{equation*}
\begin{aligned}
0\leq -I_1 &\leq \frac{b^{4-n}}{(n-4)c_n^{n-4}}\int_{t_0}^{t-\delta}\mu^{-2}(\tilde{s})F\left(\frac{(t-\tilde{s})^{\frac{1}{2}}}{\mu(\tilde{s})}\right)d\tilde{s}\\
&\leq C\frac{b^{4-n}}{(n-4)c_n^{n-4}}\int_{t_0}^{t-\delta}\mu^{-2}(\tilde{s})\left|\frac{(t-\tilde{s})^{\frac{1}{2}}}{\mu(\tilde{s})}\right|^{-(n-2)}d\tilde{s}\\
&= \frac{C}{n-4}\int_{t_0}^{t-\delta}\frac{1}{\tilde{s}}\frac{1}{(t-\tilde{s})^{\frac{n-2}{2}}}d\tilde{s}\leq \frac{C}{(n-4)t_0}\frac{2}{n-4}\frac{1}{\delta^{\frac{n-4}{2}}}.
\end{aligned}
\end{equation*}
Note that we have used the definition $\mu_{0} = b c_{n} t^{-\frac{1}{n-4}}$ and the fact $|a|^{n-2}F(a)\leq C$.
For $$I_2 =\int_{t-\delta}^{t}\frac{\dot{\mu}(\tilde{s})}{\mu(\tilde{s})}\mu^{-(n-2)}(\tilde{s})F\left(\frac{(t-\tilde{s})^{\frac{1}{2}}}{\mu(\tilde{s})}\right)d\tilde{s},$$
after change of variables $\frac{(t-\tilde{s})^{\frac{1}{2}}}{\mu(\tilde{s})} = \hat{s}$, we have
$$
d\tilde{s} = -\frac{\mu(\tilde{s})}{\frac{1}{2}(t-\tilde{s})^{-\frac{1}{2}}+\dot{\mu}(\tilde{s})\hat{s}}d\hat{s}
$$
and
\begin{equation*}
\begin{aligned}
I_2&=\int_{t-\delta}^{t}\frac{\dot{\mu}(\tilde{s})}{\mu(\tilde{s})}\mu^{-(n-2)}(\tilde{s})F\left(\frac{(t-\tilde{s})^{\frac{1}{2}}}{\mu(\tilde{s})}\right)
d\tilde{s}\\ &=\int^{\frac{\delta^{\frac{1}{2}}}{\mu(t-\delta)}}_{0}\frac{\dot{\mu}(\tilde{s})}{\mu(\tilde{s})}\mu^{-(n-2)}(\tilde{s})F\left(\hat{s}\right)
\frac{\mu(\tilde{s})}{\frac{1}{2}(t-\tilde{s})^{-\frac{1}{2}}+\dot{\mu}(\tilde{s})\hat{s}}d\hat{s}.
\end{aligned}
\end{equation*}
Observe that for small $\delta$, $\frac{1}{2}(t-\tilde{s})^{-\frac{1}{2}}+\dot{\mu}(\tilde{s})\hat{s} =\frac{1}{2}(t-\tilde{s})^{-\frac{1}{2}}(1-\frac{2}{(n-4)\tilde{s}}(t-\tilde{s}))
> \frac{1}{2}(t-\tilde{s})^{-\frac{1}{2}}(1-\frac{2}{(n-4)\tilde{s}}\delta)$, $d\tilde{s} = \frac{\mu(\tilde{s})}{\frac{1}{2}(t-\tilde{s})^{-\frac{1}{2}}}(1+O(\delta))d\hat{s}$, hence
\begin{equation*}
I_2= -\frac{2 b^{4-n}}{(n-4)c_{n}^{n-4}}\left(\int^{\frac{\delta^{\frac{1}{2}}}{\mu(t-\delta)}}_{0}\hat{s}F\left(\hat{s}\right)
d\hat{s} + o(1)\right)
= -\frac{2 b^{4-n}}{(n-4)c_{n}^{n-4}}A + o(1)
\end{equation*}
when $\frac{\delta^{\frac{1}{2}}}{\mu(t-\delta)}$ is large enough. Here the constant $A = \int_{0}^\infty \tilde{s} F(\tilde{s})d\tilde{s} < +\infty$ since the dimension of the space satisfies $n > 4$. Hence we have
\begin{equation}\label{e:2018330}
\begin{aligned}
\Phi^0(q, t) &= -(1+o(1))\int_{t_0}^{t}\frac{\dot{\mu}(\tilde{s})}{\mu(\tilde{s})}\mu^{-(n-2)}(\tilde{s})F\left(\frac{(t-\tilde{s})^{\frac{1}{2}}}{\mu(\tilde{s})}\right)d\tilde{s} \\
&= \frac{2 b^{4-n}}{(n-4)c_{n}^{n-4}}A + o(1):= B b^{4-n}+ o(1)
\end{aligned}
\end{equation}
when $t_0$ is sufficiently large. Here the constant $B$ is $ B = B_{n}: = \frac{2}{(n-4)c_{n}^{n-4}}A$. This is (\ref{e:201803312}).

Direct computations yields that
\begin{equation}\label{e2:36}
\begin{aligned}
&\mu_0^{-(n-3)}(t)\int_{\mathbb{R}^n}E_{0}(y, t)z_{0}(y)dy \approx  c_1 b^{n-3}H(q, q)-\frac{2c_1 A +c_2}{(n-4)c_{n}^{n-4}}b
\end{aligned}
\end{equation}
with
\begin{equation*}
c_1 = -p\int_{\mathbb{R}^n}|Q|^{p-1}(y)z_0(y)dy \in (0, +\infty),
\end{equation*}
\begin{equation*}
c_2 = \int_{\mathbb{R}^n}\left(z_0(y) - \frac{D_{n,k}(2-|y|^2)}{\left(1+|y|^2\right)^{\frac{n}{2}}}\right)z_{0}(y)dy \in (0, +\infty).
\end{equation*}
Note that $c_1 < + \infty$ and $c_2 < +\infty$ are due to the assumption $n > 4$. We will prove
\begin{equation}\label{positiveness}
c_1 > 0, \quad c_2 > 0
\end{equation}
in the Appendix. Write
\begin{equation*}
\mu(t) = b\mu_0(t) = b c_{n}t^{-\frac{1}{n-4}}.
\end{equation*}
Then (\ref{e2:35}) can be satisfied at main order if the following hold
\begin{equation}\label{e2:39}
b^{n-2}H(q, q) -\frac{2c_1 A +c_2}{(n-4)c_{n}^{n-4}c_1}b^{2} = 0.
\end{equation}
Imposing $\frac{2c_1 A +c_2}{(n-4)c_{n}^{n-4}c_1}=\frac{2}{n-2}$, i.e.,
$$
c_{n}=\left[\frac{(2c_1 A +c_2)(n-2)}{2(n-4)c_1}\right]^{\frac{1}{n-4}},
$$
we get
\begin{equation}\label{e2:40}
\dot{\mu}_0(t) = -\frac{2c_1}{(2c_1A + c_2)(n-2)}\mu_0^{n-3}(t).
\end{equation}
By \eqref{e2:39} and \eqref{e2:40}, the constants $b$ should satisfy the relation
\begin{equation}\label{e:equationforb}
H(q, q)b^{n-3}  = \frac{2b}{n-2}.
\end{equation}
It is clear that (\ref{e:equationforb}) can be uniquely solved if and only if
\begin{equation}\label{e:conditionforH}
H(q, q) > 0,
\end{equation}
which holds from the maximum principle. Under the assumption (\ref{e:conditionforH}),
\begin{equation}\label{e:explicitequationforb}
b = \left(\frac{2}{(n-2)H(q, q)}\right)^{\frac{1}{n-4}}.
\end{equation}

Similarly, the relations
\begin{equation}
\int_{\mathbb{R}^n}\mu^{\frac{n+2}{2}}S(u_{A})(y, t)z_{i}(y)dy = 0, \quad i = 1,\cdots, 3n-1
\end{equation}
can be achieved at main order by choosing $\xi_0 = q$, $a_0 = (0, 0)$ and $\theta_0 = (0, \cdots, 0)$.

Now fix $\mu_0(t)$ and the constant $b$ satisfying (\ref{e:explicitequationforb}), denote
\begin{equation*}
\bar{\mu}_0 = b\mu_0(t).
\end{equation*}
Let $\Phi$ be the solution for (\ref{e2:31}) for $\mu = \bar{\mu}_0$ which is unique, then we have the following
\begin{equation*}
\Delta_y\Phi + p|Q|^{p-1}(y)\Phi = -\mu_{0}E_{0}[\mu_0, \dot{\mu}_0]\text{ in }\mathbb{R}^n,\,\,\Phi(y, t)\to 0\text{ as }|y|\to \infty.
\end{equation*}
From the definitions for $\mu_0$ and $b$, we obtain
\begin{equation*}
\mu_{0}E_{0}=-\gamma\mu_0^{n-2}q_0(y),
\end{equation*}
where $\gamma$ is positive,
\begin{equation}\label{e2:50}
q_0(y): = \frac{p|Q|^{p-1}(y)c_2b^{2}}{(n-4)c_{n}^{n-4}c_1} + \frac{b^{2}}{(n-4)c_{n}^{n-4}}\left(z_0(y) - \frac{D_{n,k}(2-|y|^2)}{\left(1+|y|^2\right)^{\frac{n}{2}}}\right)
\end{equation}
and $\int_{\mathbb{R}^n}q_0(y)z_{0}dy = 0$.

Let $p_0 = p_0(|y|)$ be the solution for $L(p_0) = q_0$. Then $p_0(y) = O(|y|^{-2})$ as $|y|\to\infty$ since (\ref{e2:50}) holds. Therefore,
\begin{equation}\label{e2:51}
\Phi(y, t) = \gamma\mu_0^{n-2}p_0(y).
\end{equation}
Thus the corrected approximation should be
\begin{equation}\label{e2:52}
u^*_{A}(x, t) = u_{A}(x, t) + \tilde{\Phi}(x, t)
\end{equation}
with
\begin{equation*}
\tilde{\Phi}(x, t) = \mu(t)^{-\frac{n-2}{2}}\Phi\left(\frac{x-\xi(t)}{\mu(t)}\right).
\end{equation*}

\subsection{Estimating the error $S(u^*_{A})$.}
In the region $|x-q| > \delta$, $S(u^*_{A})$ can be described as
\begin{equation}\label{awayq}
S(u^*_{A})(x, t) = \mu_0^{\frac{n-2}{2}-1}\dot{\mu}f_1 + \mu_0^{\frac{n+2}{2}}f_2 + \mu_0^{\frac{n-2}{2}}\dot{\xi}\cdot \vec{f}_1 + \mu_0^{\frac{n}{2}}\dot{a}\cdot \vec{f}_2 + \mu_0^{\frac{n}{2}}\dot{\theta}\cdot \vec{f}_3,
\end{equation}
where $f_1$, $f_2$, $\vec{f}_1$, $\vec{f}_2$ and $\vec{f}_3$ are smooth bounded functions depending on the tuple $(x, \mu_0^{-1}\mu, \xi, a, \theta)$.

In the region near the point $q$, direct computations yields that
\begin{equation}\label{sstar}
\begin{aligned}
S(u^*_{A}) &= S(u_{A}) - \mu^{-\frac{n+2}{2}}\mu_{0}E_{0}[\bar{\mu}_0, \dot{\mu}_{0}]+\mu^{-\frac{n+2}{2}}\Bigg\{-\mu^{2}\partial_t\Phi(y, t) \\
&+ \mu\dot{\mu}\left[\frac{n-2}{2}\Phi(y, t)+y\cdot\nabla_y\Phi\right] + \nabla_y\Phi(y, t)\cdot \mu\dot{\xi}\Bigg\}\\
& + \left|u_{A} + \tilde{\Phi}\right|^{p-1}(u_{A} + \tilde{\Phi}) - \left|u_{A}\right|^{p-1}u_A - p\mu^{-\frac{n+2}{2}}|Q(y)|^{p-1}\Phi(y, t),
\end{aligned}
\end{equation}
where $y=\frac{x-\xi}{\mu}$.
If $|x-q|\leq \delta$,
\begin{equation}\label{e2:55}
\mu^{\frac{n+2}{2}}S(u^*_{A}) = \mu^{\frac{n+2}{2}}S(u_{A}) - \mu_{0}E_{0}[\bar{\mu}_0, \dot{\mu}_{0}]+A(y),
\end{equation}
where
\begin{equation}\label{e2:56}
A = \mu_0^{n+4}f(\mu_0^{-1}\mu, \xi, a, \theta, \mu y) + \frac{\mu_0^{2n-4}}{1+|y|^{2}}g(\mu_0^{-1}\mu, \xi, a, \theta, \mu y),\quad y = \frac{x-\xi}{\mu}
\end{equation}
for smooth and bounded functions $f$ and $g$.

Now we write $\mu(t)$ as
\begin{equation*}
\mu(t) = \bar{\mu}_0 + \lambda(t).
\end{equation*}
From (\ref{e2:55}),
\begin{equation*}
S(u^*_{A}) = \mu^{-\frac{n+2}{2}}\left\{\mu_{0}\big(E_{0}[\mu, \dot{\mu}]- E_{0}[\bar{\mu}_0, \dot{\mu}_{0}]\big)+\lambda E_{0}[\mu, \dot{\mu}]+\mu E_{1}[\mu,\dot{\xi}]+R+A\right\}.
\end{equation*}

Observe that $\Phi^0$ is a nonlocal term depending on $\mu,$ $\dot{\mu}$ and we have
\begin{equation*}
\begin{aligned}
\mu^{n-3}\Phi^0[\bar{\mu}_0+\lambda, b\dot{\mu}_0&+\dot{\lambda}](q, t) - \mu^{n-3}\Phi^0[\bar{\mu}_0, b\dot{\mu}_0](q, t) =  -2A \dot{\lambda} - \mu_0^{n-4}(n-3)B\lambda
\end{aligned}
\end{equation*}
which can be deduced by similar arguments as (\ref{e:2018330}), one gets
\begin{equation*}
\begin{aligned}
&E_{0}[\bar{\mu}_0+\lambda, b\dot{\mu}_0+\dot{\lambda}]- E_{0}[\bar{\mu}_0, b\dot{\mu}_0]\\
&= \dot{\lambda}\left(z_0(y) - \frac{D_{n,k}(2-|y|^2)}{\left(1+|y|^2\right)^{\frac{n}{2}}}\right)-\mu_0^{n-4}p|Q|^{p-1}(y)\left[(n-3)b^{n-4}H(q,q)\lambda\right]\\
&\quad + \mu_0^{n-4} p|Q|^{p-1}(y)(n-3)B\lambda - p|Q|^{p-1}(y)2A \dot{\lambda}- \mu_0^{n-4}p|Q|^{p-1}(y)(n-3)B\lambda,
\end{aligned}
\end{equation*}
As for $\lambda E_{0}[\mu, \dot{\mu}]$, we have
\begin{equation*}
\begin{aligned}
\lambda E_{0}[\mu, \dot{\mu}]&=\lambda\dot{\lambda}\left(z_0(y) - \frac{D_{n,k}(2-|y|^2)}{\left(1+|y|^2\right)^{\frac{n}{2}}}\right) +\lambda b\Bigg[\dot{\mu}_0\left(z_0(y) - \frac{D_{n,k}(2-|y|^2)}{\left(1+|y|^2\right)^{\frac{n}{2}}}\right)\\
&\quad + p|Q|^{p-1}(y)\mu_0^{n-3}\left(-b^{n-4}H(q, q)\right)\Bigg] + p|Q|^{p-1}(y)b\mu_0^{n-3}B\lambda \\ &\quad -\mu_0^{n-4}p|Q|^{p-1}(y)f(\mu_0^{-1}\lambda)\lambda^2,
\end{aligned}
\end{equation*}
where $f$ is smooth and bounded in its arguments.

Combine all the estimates above, we get the expansion for $S(u^*_{A})$.
\begin{lemma}\label{l2.2}
In the region $|x-q|\leq \delta $ for a fixed small $\delta > 0$, set $\mu = \bar{\mu}_0 + \lambda$ with $|\lambda(t)|\leq \mu_0(t)^{1+\sigma}$ for some positive number $\sigma\in (0, n-4)$. When $t$ is large enough, we have the expansion of $S(u^*_{A})$ as
\begin{equation*}
\begin{aligned}
&S(u^*_{A})\\& = \mu^{-\frac{n+2}{2}}\Bigg\{\mu_{0}\dot{\lambda}\left(z_0(y) - \frac{D_{n,k}(2-|y|^2)}{\left(1+|y|^2\right)^{\frac{n}{2}}}-2Ap|Q|^{p-1}(y)\right)\\
&\quad -\mu_{0}\mu_0^{n-4}p|Q|^{p-1}(y)\left[(n-3)b^{n-4}H(q,q)\lambda\right]\\
&\quad + \left(\nabla Q(y) - \frac{E_{n,k}y}{\left(1+|y|^2\right)^{\frac{n}{2}}}\right)\cdot \dot{\xi} + p|Q|^{p-1}\left[-\mu^{n-2}\nabla H(q, q)\right]\cdot y \\
&\quad + \mu^2(t)\dot{a}_1 \left(-2y_1\left(z_0(y) - \frac{D_{n,k}(2-|y|^2)}{\left(1+|y|^2\right)^{\frac{n}{2}}}\right) + |y|^2\left(\frac{\partial}{\partial y_1} Q(y) - \frac{E_{n,k}y_1}{\left(1+|y|^2\right)^{\frac{n}{2}}}\right)\right)\\
&\quad + \mu^2(t)\dot{a}_2 \left(-2y_2\left(z_0(y) - \frac{D_{n,k}(2-|y|^2)}{\left(1+|y|^2\right)^{\frac{n}{2}}}\right) + |y|^2\left(\frac{\partial}{\partial y_2} Q(y) - \frac{E_{n,k}y_2}{\left(1+|y|^2\right)^{\frac{n}{2}}}\right)\right)\\
&\quad + \mu^2(t)\dot{\theta}_{12}z_{n+1}(y) + \sum_{j = 3}^n\left(\mu^2(t)\dot{\theta}_{1j}z_{n+j+1}(y) + \mu^2(t)\dot{\theta}_{2j}z_{2n+j-1}(y)\right)\Bigg\}\\
&\quad + \mu^{-\frac{n+2}{2}}\lambda b\Bigg[\dot{\mu}_0\left(z_0(y) - \frac{D_{n,k}(2-|y|^2)}{\left(1+|y|^2\right)^{\frac{n}{2}}}\right)\\
&\quad\quad\quad\quad\quad\quad\quad\quad\quad\quad\quad\quad\quad\quad\quad\quad + p|Q|^{p-1}(y)\mu_0^{n-3}\bigg(-b^{n-4}H(q, q) + B\bigg)\Bigg]\\
&\quad +\mu_0^{-\frac{n+2}{2}}\left[\mu_0^{n-4}p|Q|^{p-1}(y)f_1\lambda^2 +\frac{f_2}{1+|y|^{n-2}}\lambda\dot{\lambda} + \mu_0^{n+2}f_3 + \mu_0^{n-1}\dot{\mu}f_4\right]\\
&\quad +\mu_0^{-\frac{n+2}{2}}\left[\dot{\xi}\vec{f}_1 + \dot{\xi}\vec{f}_2 + \dot{\xi}\vec{f}_3\right]\\
&\quad +\mu_0^{-\frac{n+2}{2}}\left[\frac{\mu_0^{n}g_1}{1+|y|^{2}}+\frac{\mu_0^{2n-4}g_2}{1+|y|^{2}} + \frac{\mu_0^{n-2}g_3}{1+|y|^{4}}\lambda+\frac{\mu_0^{n-1}\vec{g}_1}{1+|y|^{2}}\cdot a +\frac{\mu_0^{n-2}\vec{g}_2}{1+|y|^{4}}\cdot(\xi-q)\right]\\,
\end{aligned}
\end{equation*}
where $x = \xi + \mu y$, $f_1$, $f_2$, $f_3$,$f_4$, $\vec{f}_1$, $\vec{f}_2$, $\vec{f}_3$, $g_1$, $g_2$, $g_3$ and $\vec{g}_1$, $\vec{g}_2$ are smooth bounded (vector) functions depending on the tuple of variables $(\mu_0^{-1}\mu, \xi, a, \theta, x)$.
\end{lemma}

\section{The inner-outer gluing procedure}
We will find a solution for (\ref{e3:1}) with form
\begin{equation*}
u = u^*_{A} + \tilde{\phi}
\end{equation*}
when $t_0$ is large enough, the function $\tilde{\phi}(x, t)$ is small compared to $u^*_A$. To this aim, we use the \textbf{inner-outer gluing procedure}.

Write
\begin{equation*}
\tilde{\phi}(x, t) = \psi(x, t) + \phi^{in}(x, t)\quad\text{where}\quad\phi^{in}(x, t): = \eta_{R}(x,t)\tilde{\phi}(x, t)
\end{equation*}
with
\begin{equation*}
\tilde{\phi}(x, t): = \mu_{0}^{-\frac{n-2}{2}}\phi\left(\frac{x-\xi}{\mu_{0}}, t\right),\quad \mu_{0}(t) = b\mu_0(t)
\end{equation*}
and
\begin{equation*}
\eta_{R}(x, t) = \eta\left(\frac{|x-\xi|}{R\mu_{0}}\right).
\end{equation*}
In above, $\eta(\tau)$ is a (smooth) cut-off function defined on the interval $[0, +\infty)$, $\eta(\tau) = 1$ for $0\leq \tau < 1$ and $\eta(\tau)= 0$ for $\tau > 2$. $R$ is a fixed number defined as
\begin{equation}\label{e3:20}
R = t_0^{\rho}\quad \text{with}\quad 0 < \rho \ll 1.
\end{equation}
Under this ansatz, problem (\ref{e3:1}) can be written as
\begin{equation}\label{e3:5}
\begin{cases}
\partial_t\tilde{\phi} = \Delta\tilde{\phi} + p(u^*_{A})^{p-1}\tilde{\phi} + \tilde{N}(\tilde{\phi}) + S(u^*_{A})&\text{in}\quad \Omega\times (t_0, \infty),\\
\tilde{\phi} = -u^*_{A}&\text{in}\quad \partial \Omega\times (t_0, \infty),
\end{cases}
\end{equation}
where $\tilde{N}_A(\tilde{\phi})=|u_{A}^*+\tilde{\phi}|^{p-1}(u_{A}^*+\tilde{\phi})-p|u_{A}^*|^{p-1}\tilde{\phi}-|u_{A}^*|^{p-1}u_{A}^*$, $S(u^*_{A})=-\partial_t \mu^*_{A}+\Delta u^*_{A}+|u^*_{A}|^{p-1}u^*_{A}$.
Let us write $S(u^*_{A})$ as
\begin{equation*}
S(u^*_{A}) = S_{A} + S^{(2)}_{A}
\end{equation*}
where
\begin{equation*}
\begin{aligned}
&S_{A}\\
& = \mu^{-\frac{n+2}{2}}\Bigg\{\mu_{0}\dot{\lambda}\left(z_0(y) - \frac{D_{n,k}(2-|y|^2)}{\left(1+|y|^2\right)^{\frac{n}{2}}}-2Ap|Q|^{p-1}(y)\right)\\
&\quad -\mu_{0}\mu_0^{n-4}p|Q|^{p-1}(y)\left[(n-3)b^{n-4}H(q,q)\lambda\right]\\
&\quad + \lambda b\Bigg[\dot{\mu}_0\left(z_0(y) - \frac{D_{n,k}(2-|y|^2)}{\left(1+|y|^2\right)^{\frac{n}{2}}}\right) + p|Q|^{p-1}(y)\mu_0^{n-3}\bigg(-b^{n-4}H(q, q) + B\bigg)\Bigg]\\
&\quad + \left(\nabla Q(y) - \frac{E_{n,k}y}{\left(1+|y|^2\right)^{\frac{n}{2}}}\right)\cdot \dot{\xi} + p|Q|^{p-1}\left[-\mu^{n-2}\nabla H(q, q)\right]\cdot y\\
&\quad + \mu^2(t)\dot{a}_1 \left(-2y_1\left(z_0(y) - \frac{D_{n,k}(2-|y|^2)}{\left(1+|y|^2\right)^{\frac{n}{2}}}\right) + |y|^2\left(\frac{\partial}{\partial y_1} Q(y) - \frac{E_{n,k}y_1}{\left(1+|y|^2\right)^{\frac{n}{2}}}\right)\right)\\
&\quad + \mu^2(t)\dot{a}_2 \left(-2y_2\left(z_0(y) - \frac{D_{n,k}(2-|y|^2)}{\left(1+|y|^2\right)^{\frac{n}{2}}}\right) + |y|^2\left(\frac{\partial}{\partial y_2} Q(y) - \frac{E_{n,k}y_2}{\left(1+|y|^2\right)^{\frac{n}{2}}}\right)\right)\\
&\quad + \mu^2(t)\dot{\theta}_{12}z_{n+1}(y) + \sum_{j = 3}^n\left(\mu^2(t)\dot{\theta}_{1j}z_{n+j+1}(y) + \mu^2(t)\dot{\theta}_{2j}z_{2n+j-1}(y)\right)\Bigg\}.
\end{aligned}
\end{equation*}
Define
\begin{equation}\label{e3:7}
V_{A} = p\left(|u^*_{A}|^{p-1} - \left|\mu^{-\frac{n-2}{2}}Q\left(\frac{x-\xi}{\mu}\right)\right|^{p-1}\right)\eta_{R} + p(1-\eta_{R})|u^*_{A}|^{p-1},
\end{equation}
then $\tilde{\phi}$ satisfies problem (\ref{e3:5}) if

\noindent (1) $\psi$ solves the \textbf{outer problem}
\begin{equation}\label{outerproblem}
\left\{
\begin{aligned}
&\partial_t\psi =\Delta\psi + V_{A}\psi + 2\nabla\eta_{R}\nabla\tilde{\phi}+ \tilde{\phi}\big(\Delta -\partial_t\big)\eta_{R} + \tilde{N}_{A}(\tilde{\phi}) + S_{out},\text{ in }\Omega\times(t_0,\infty),\\
&\psi = -u^*_{A}\quad\text{on}\quad \partial\Omega\times (t_0,\infty),
\end{aligned}
\right.
\end{equation}
with
\begin{equation}\label{e3:8}
S_{out} = S^{(2)}_{A} + (1 - \eta_{R})S_{A}.
\end{equation}

\noindent (2) $\tilde{\phi}$ solves
\begin{equation}\label{e3:9}
\eta_{R}\partial_t\tilde{\phi} = \eta_{R}\left[\Delta\tilde{\phi} + p|Q_{\mu, \xi, \theta}|^{p-1}\tilde{\phi} + p|Q_{\mu, \xi}|^{p-1}\psi + S_{A}\right]\text{ in }B_{2R\mu}(\xi)\times (t_0, \infty),
\end{equation}
for $Q_{\mu, \xi} := \mu^{-\frac{n-2}{2}}Q\left(\frac{x-\xi}{\mu}\right)$.
In the self-similar form, (\ref{e3:9}) becomes the so-called \textbf{inner problem}
\begin{equation}\label{e3:10}
\begin{aligned}
&\mu_{0}^2\partial_t\phi = \Delta_y\phi + p|Q|^{p-1}(y)\phi + \mu_{0}^{\frac{n+2}{2}}S_{A}(\xi + \mu_{0}y, t)\\
&\quad\quad + p\mu_{0}^{\frac{n-2}{2}}\frac{\mu_{0}^{2}}{\mu^{2}}|Q|^{p-1}(\frac{\mu_{0}}{\mu}y)\psi(\xi + \mu_{0}y, t) + B[\phi] + B^0[\phi]\quad\text{in}\quad B_{2R}(0)\times (t_0, \infty),
\end{aligned}
\end{equation}
where
\begin{equation}\label{e3:11}
B[\phi]:=\mu_{0}\dot{\mu}_{0}\left(\frac{n-2}{2}\phi + y\cdot\nabla_y\phi\right) + \mu_{0}\nabla\phi\cdot \dot{\xi}
\end{equation}
and
\begin{equation}\label{e3:12}
B^0[\phi]:= p\left[|Q|^{p-1}\left(\frac{\mu_{0}}{\mu}y\right)-|Q|^{p-1}(y)\right]\phi + p\left[\mu_{0}^{2}|u^{*}_{A}|^{p-1} -|Q|^{p-1}\left(\frac{\mu_{0}}{\mu}y\right)\right]\phi.
\end{equation}

\section{Scheme of the proof}
To find a solution $(\phi, \psi)$ satisfying (\ref{outerproblem}) and (\ref{e3:10}), we proceed with the following steps.
\subsection{Linear theory for (\ref{e3:10}).}
Let us rewrite problem (\ref{e3:10}) as
\begin{eqnarray}\label{e5:1}
&& \mu_{0}^{2}\partial_t\phi = \Delta_y\phi + p|Q|^{p-1}(y)\phi + H[\lambda,\xi, a, \theta, \dot{\lambda},\dot{\xi}, \dot{a}, \dot{\theta}, \phi, \psi](y,t), y\in B_{2R}(0),
\end{eqnarray}
for $t\geq t_0$, where
\begin{equation}\label{e5:2}
\begin{aligned}
H[\lambda,\xi, a, \theta, \dot{\lambda},\dot{\xi}, \dot{a}, \dot{\theta}, \phi, \psi]:=&\mu_{0}^{\frac{n+2}{2}}S_{A}(\xi + \mu_{0}y, t)+ B[\phi] + B^0[\phi]\\
&+ p\mu_{0}^{\frac{n-2}{2}}\frac{\mu_{0}^{2}}{\mu^{2}}|Q|^{p-1}(\frac{\mu_{0}}{\mu}y)\psi(\xi + \mu_{0}y, t),
\end{aligned}
\end{equation}
the terms $B[\phi]$, $B^0[\phi]$ are defined in (\ref{e3:11}), (\ref{e3:12}) respectively. Using change of variables
\begin{equation*}\label{e5:3}
t = t(\tau),\quad \frac{dt}{d\tau} = \mu_{0}^{2}(t),
\end{equation*}
(\ref{e5:1}) becomes
\begin{eqnarray}\label{e5:4}
\partial_\tau\phi = \Delta_y\phi + p|Q|^{p-1}(y)\phi + H[\lambda,\xi, a, \theta, \dot{\lambda},\dot{\xi}, \dot{a}, \dot{\theta}, \phi, \psi](y,t(\tau))
\end{eqnarray}
for $y\in B_{2R}(0)$, $\tau\geq \tau_0$. Here $\tau_0$ the (unique) positive number such that $t(\tau_0) = t_0$.
We try to find a solution $\phi$ to the following equation
\begin{equation}\label{e5:6}
\left\{
\begin{aligned}
&\partial_\tau\phi = \Delta_y\phi + p|Q|^{p-1}(y)\phi\\
&\quad\quad\quad\quad\quad + H[\lambda,\xi, a, \theta, \dot{\lambda},\dot{\xi}, \dot{a}, \dot{\theta}, \phi, \psi](y,t(\tau)),\quad y\in B_{2R}(0),\quad\tau\geq\tau_0,\\
&\phi(y,\tau_0) = \sum_{l= 1}^Ke_{l}Z_l(y),\quad y\in B_{2R}(0),
\end{aligned}
\right.
\end{equation}
for suitable constants $e_{l}$, $l = 1, \cdots, K$. Here $Z_l$ are eigenfunctions associated to negative eigenvalues of the problem
\begin{equation*}\label{e3:13}
L(\phi) + \lambda\phi = 0,\quad \phi\in L^\infty(\mathbb{R}^n).
\end{equation*}
It was proved in \cite{KenigMerle2016} that $K$ is finite and $Z_l$ satisfies
\begin{equation*}\label{e3:14}
Z_l(x) \sim \frac{e^{-\sqrt{-\lambda}|x|}}{|x|^{\frac{N-1}{2}}}\text{ as } |x|\to \infty.
\end{equation*}
Next, we prove that (\ref{e5:6}) is solvable for $\phi$, provided $\psi$ is in suitable weighted spaces and the parameter functions $\lambda$, $\xi$, $a$, $\theta$ are chosen so that the term $H[\lambda,\xi, a, \theta, \dot{\lambda},\dot{\xi}, \dot{a}, \dot{\theta}, \phi, \psi](y,t(\tau))$  in the right hand side of (\ref{e5:6}) satisfies the following $L^2$ orthogonality conditions
\begin{equation}\label{e5:7}
\int_{B_{2R}}H[\lambda,\xi, a, \theta, \dot{\lambda},\dot{\xi}, \dot{a}, \dot{\theta}, \phi, \psi](y,t(\tau))z_l(y)dy = 0,
\end{equation}
for all $\tau\geq \tau_0$, $l = 0,1,2,\cdots,3n-1$. These conditions will impose highly nonlinearity to (\ref{e5:6}), to get a solution $\phi$, we apply the Schauder fixed-point theorem. We first need a linear theory for (\ref{e5:6}).

For $R > 0$ large but fixed, consider the following initial value problem
\begin{equation}\label{e5:31}
\left\{
\begin{aligned}
&\partial_\tau\phi = \Delta\phi + p|Q|^{p-1}(y)\phi + h(y,\tau),~ y\in B_{2R}(0),~ \tau\geq \tau_0,\\
&\phi(y,\tau_0) = \sum_{l= 1}^Ke_{l}Z_l(y),~ y\in B_{2R}(0).
\end{aligned}
\right.
\end{equation}
Set
\begin{equation*}
\nu = 1 + \frac{\sigma}{n-2},
\end{equation*}
then we have $\mu_0^{n-2+\sigma}\sim \tau^{-\nu}$.
Define the weighted norm for $h$ as
\begin{equation*}
\|h\|_{\alpha, \nu}: = \sup_{\tau > \tau_0}\sup_{y\in B_{2R}}\tau^\nu(1 + |y|^\alpha)|h(y,\tau)|.
\end{equation*}
Then the following estimates for (\ref{e5:31}) hold.
\begin{prop}\label{proposition5.1}
Suppose $\alpha\in (2, n-2)$, $\nu > 0$, $\|h\|_{2+\alpha,\nu} < +\infty$ and
\begin{equation*}\label{orthogalcondition}
\int_{B_{2R}}h(y,\tau)z_j(y)dy = 0~\text{ for all }~\tau\in (\tau_0,\infty), ~j = 0, 1, \cdots, 3n-1.
\end{equation*}
Then there exist functions $\phi = \phi[h](y, \tau)$ and $(e_1,\cdots, e_K ) = (e_1[h](\tau),\cdots, e_K[h](\tau))$ satisfying (\ref{e5:31}). Furthermore, for $\tau\in (\tau_0,+\infty)$, $y\in B_{2R}(0)$, there hold
\begin{equation}\label{e5:100}
\begin{aligned}
&(1+|y|)|\nabla_y \phi(y, \tau)|+ |\phi(y,\tau)|\lesssim \tau^{-\nu}(1+|y|)^{-\alpha}\|h\|_{2+\alpha,\nu}
\end{aligned}
\end{equation}
and
\begin{equation}\label{e5:101}
|e_l[h]|\lesssim \|h\|_{2+\alpha,\nu} \text{ for }l = 1,\cdots, K.
\end{equation}
\end{prop}
Here and in the following of this paper, the symbol $a\lesssim b$ means $a\leq C b$ for some positive constant $C$ which is independent of $t$ and $t_0$.
The proof of Proposition \ref{proposition5.1} is given in Section 5.

\subsection{The orthogonality conditions (\ref{e5:7}).}
To apply Proposition \ref{proposition5.1}, we should choose the parameter functions $\lambda$, $\xi$, $a$ and $\theta$ such that (\ref{e5:7}) hold.

Let us fix a $\sigma\in (0, n-4)$. Given $h(t):(t_0, \infty)\to\mathbb{R}^k$ and $\delta > 0$, the weighted $L^\infty$ norm is defined as
\begin{equation*}\label{e4:30}
\|h\|_\delta:=\|\mu_0(t)^{-\delta}h(t)\|_{L^\infty(t_0, \infty)}.
\end{equation*}
In what follows, $\alpha$ is always a positive constant such that $\alpha > 2$ and $\alpha-2$ is small enough.
Also assume the parameter functions $\lambda$, $\xi$, $a$, $\theta$, $\dot{\lambda}$, $\dot{\xi}$, $\dot{a}$ and $\dot{\theta}$ satisfy the following constraints,
\begin{equation}\label{e4:21}
\|\dot{\lambda}(t)\|_{n-3+\sigma} + \|\dot{\xi}(t)\|_{n-3+\sigma} + \|\dot{a}(t)\|_{n-4+\sigma} + \|\dot{\theta}(t)\|_{n-4+\sigma}\leq \frac{c}{R^{\alpha-2}},
\end{equation}
\begin{equation}\label{e4:22}
\|\lambda(t)\|_{1+\sigma} + \|\xi(t)-q\|_{1+\sigma} + \|a(t)\|_{\sigma} + \|\theta(t)\|_{\sigma}\leq \frac{c}{R^{\alpha-2}},
\end{equation}
here $c$ is a positive constant which is independent of $R$, $t$ and $t_0$.
Let us define the norm $\|\phi\|_{n-2+\sigma,\alpha}$ of $\phi$ as the least number $M>0$ such that
\begin{equation}\label{e4:24}
(1+|y|)|\nabla_y \phi(y, t)| + |\phi(y, t)|\leq M\frac{\mu_0^{n-2+\sigma}}{1+|y|^\alpha}
\end{equation}
and $\|\psi\|_{**,\beta,\alpha}$ is the least $M > 0$ such that
\begin{equation}\label{e4:400}
|\psi(x, t)|\leq M\frac{t^{-\beta}}{1+|y|^{\alpha-2}},\quad y = \frac{|x-\xi|}{\mu}
\end{equation}
holds. Here $\beta = \frac{n-2}{2(n-4)} + \frac{\sigma}{n-4}$. We suppose $\phi$ and $\psi$ satisfy
\begin{equation}\label{e4:25}
\|\phi\|_{n-2+\sigma,\alpha}\leq ct_0^{-\varepsilon}
\end{equation}
and
\begin{equation*}
\|\psi\|_{**, \beta, \alpha}\leq \frac{ct_0^{-\varepsilon}}{R^{\alpha-2}}
\end{equation*}
for some small $\varepsilon > 0$, respectively.

Then we have the following result.
\begin{prop}\label{l5:1}
(\ref{e5:7}) is equivalent to
\begin{equation}\label{e5:9}
\left\{
\begin{aligned}
&\dot{\lambda} + \frac{1+(n-4)}{(n-4)t}\lambda = \Pi_0[\lambda,\xi, a, \theta, \dot{\lambda},\dot{\xi}, \dot{a}, \dot{\theta}, \phi, \psi](t),\\
&\dot{\xi}_l = \Pi_l[\lambda,\xi, a, \theta, \dot{\lambda},\dot{\xi}, \dot{a}, \dot{\theta}, \phi, \psi](t), \quad l = 1,\cdots, n,\\
&\dot{\theta}_{12} = \mu_0^{-1}\Pi_{n+1}[\lambda,\xi, a, \theta, \dot{\lambda},\dot{\xi}, \dot{a}, \dot{\theta}, \phi, \psi](t),\\
&\dot{a}_1 = \mu_0^{-1}\Pi_{n+2}[\lambda,\xi, a, \theta, \dot{\lambda},\dot{\xi}, \dot{a}, \dot{\theta}, \phi, \psi](t),\\
&\dot{a}_2 = \mu_0^{-1}\Pi_{n+3}[\lambda,\xi, a, \theta, \dot{\lambda},\dot{\xi}, \dot{a}, \dot{\theta}, \phi, \psi](t),\\
&\dot{\theta}_{1l} = \mu_0^{-1}\Pi_{n+l+1}[\lambda,\xi, a, \theta, \dot{\lambda},\dot{\xi}, \dot{a}, \dot{\theta}, \phi, \psi](t),\quad l = 3, \cdots, n\\
&\dot{\theta}_{2l} = \mu_0^{-1}\Pi_{2n+l-1}[\lambda,\xi, a, \theta, \dot{\lambda},\dot{\xi}, \dot{a}, \dot{\theta}, \phi, \psi](t), \quad l = 3, \cdots, n.
\end{aligned}
\right.
\end{equation}
The terms in the right hand side of (\ref{e5:9}) can be written as
\begin{equation*}
\begin{aligned}
&\Pi_0[\lambda,\xi, a, \theta, \dot{\lambda},\dot{\xi}, \dot{a}, \dot{\theta}, \phi, \psi](t) = \frac{t_0^{-\varepsilon}}{R^{\alpha-2}}\mu_0^{n-3 + \sigma}(t)f_0(t) + \frac{t_0^{-\varepsilon}}{R^{\alpha-2}}\\
&\quad \Theta_0\left[\dot{\lambda},\dot{\xi}, \mu_0\dot{a}, \mu_0\dot{\theta}, \mu_0^{n-4}(t)\lambda, \mu_0^{n-4}(\xi-q), \mu_0^{n-3}a, \mu_0^{n-3}\theta, \mu_0^{n-3+\sigma}\phi, \mu_0^{\frac{n-2}{2}+\sigma}\psi\right](t)
\end{aligned}
\end{equation*}
and for $l = 1, \cdots, 3n-1$,
\begin{equation*}
\begin{aligned}
&\Pi_l[\lambda,\xi, a, \theta, \dot{\lambda},\dot{\xi}, \dot{a}, \dot{\theta}, \phi, \psi](t)\\
&= \mu_0^{n-2}c_l\left[b^{n-2}\nabla H(q, q)\right]+ \mu_0^{n-2+\sigma}(t)f_l(t) + \frac{t_0^{-\varepsilon}}{R^{\alpha-2}}\\
&\quad \Theta_l\left[\dot{\lambda},\dot{\xi}, \mu_0\dot{a}, \mu_0\dot{\theta}, \mu_0^{n-4}(t)\lambda, \mu_0^{n-4}(\xi-q), \mu_0^{n-3}a, \mu_0^{n-3}\theta, \mu_0^{n-3+\sigma}\phi, \mu_0^{\frac{n-2}{2}+\sigma}\psi\right](t),
\end{aligned}
\end{equation*}
where $c_l$ are suitable constants, $f_l(t)$ and $\Theta_l[\cdots](t)$ ($l = 0, \cdot, 3n-1$) are bounded smooth functions for $t\in [t_0,\infty)$.
\end{prop}
The proof of Proposition \ref{l5:1} is given in Section 6.

\subsection{The outer problem.}
Let us consider the out problem (\ref{outerproblem}),
\begin{equation}\label{e4:main}
\left\{
\begin{aligned}
&\partial_t\psi =\Delta\psi + V_{A}\psi + 2\nabla\eta_{R}\nabla\tilde{\phi}+ \tilde{\phi}\big(\Delta -\partial_t\big)\eta_{R}\\
&\quad\quad\quad\quad\quad\quad\quad\quad\quad\quad\quad\quad\,\,\,\, + \tilde{N}_{A}(\tilde{\phi}) + S_{out}, \text{ in }\Omega\times(t_0,\infty),\\
&\psi = -u^*_{A}\quad\text{on}\quad \partial\Omega\times (t_0,\infty), \quad \psi(t_0, \cdot) = \psi_0\quad\text{in }\Omega,
\end{aligned}
\right.
\end{equation}
with a smooth and small initial datum $\psi_0$.

To apply the Schauder fixed-point theorem to (\ref{e4:main}) and get a solution $\psi$, we first consider the corresponding linear problem
\begin{equation}\label{e4:3}
\begin{cases}
\partial_t\psi =
\Delta\psi + V_{A}\psi + f(x, t)\quad&\text{in}\quad\Omega\times (t_0, \infty),\\
\psi = g\quad&\text{on}\quad \partial\Omega\times (t_0,\infty),\\
\psi(t_0, \cdot) = h\quad&\text{in}\quad\Omega,
\end{cases}
\end{equation}
where $f(x, t)$, $g(x, t)$ and $h(x)$ are smooth functions, $V_{\mu, \xi}$ is defined in (\ref{e3:7}). We denote $\|f\|_{*, \gamma, 2+\alpha}$ as the least $M > 0$ such that
\begin{equation}\label{e4:2}
|f(x, t)|\leq M\frac{\mu^{-2}t^{-\gamma}}{1+|y|^{2+\varsigma}},\quad y = \frac{x-\xi}{\mu}
\end{equation}
for given $\varsigma$, $\gamma > 0$.
Then the following {\itshape a priori} estimate holds for problem \eqref{e4:3}.
\begin{prop}\label{l4:lemma4.1}
Suppose $\|f\|_{*, \gamma, 2+\varsigma} < +\infty$ for some constants $\varsigma$, $\gamma > 0$, $0 < \varsigma \ll 1$, $\|h\|_{L^\infty(\Omega)} < +\infty$ and
$\|\tau^\gamma g(x, \tau)\|_{L^\infty(\partial\Omega\times (t_0,\infty))} < +\infty$. Let $\phi = \psi[f, g, h]$ be the unique solution of (\ref{e4:3}), then there exists $\delta = \delta(\Omega) > 0$ small such that, for all $(x,t)$, one has
\begin{equation}\label{e4:40}
\begin{aligned}
|\psi(x, t)|&\lesssim \|f\|_{*, \gamma, 2+\varsigma}\frac{t^{-\gamma}}{1+|y|^{\varsigma}} + e^{-\delta(t-t_0)}\|h\|_{L^\infty(\Omega)} \\
&\quad  + t^{-\gamma}\|\tau^\gamma g(x, \tau)\|_{L^\infty(\partial\Omega\times (t_0,\infty))}, \quad y=\frac{x-\xi}{\mu}
\end{aligned}
\end{equation}
and
\begin{equation}\label{e4:5}
|\nabla\psi(x, t)|\lesssim \|f\|_{*, \gamma, 2+\varsigma}\frac{\mu^{-1}t^{-\gamma}}{1+|y|^{\varsigma+1}}\text{ for } |y|\leq R.
\end{equation}
\end{prop}
The proof is the same as Lemma 4.1 in \cite{cortazar2016green}, so we omit it.
This result will be applied to problem (\ref{e4:main}), as a first step, we establish the following estimates for
\begin{equation*}
f[\psi](x, t)= 2\nabla\eta_{R}\nabla\tilde{\phi}+ \tilde{\phi}\big(\Delta -\partial_t\big)\eta_{R} + \tilde{N}_{A}(\tilde{\phi}) + S_{out}.
\end{equation*}
\begin{prop}\label{propositionestimate}
We have
\begin{itemize}
\item[(1)]
\begin{equation}\label{e4:37}
|S_{out}(x, t)|\lesssim \frac{t_0^{-\varepsilon}}{R^{\alpha-2}}\frac{\mu^{-2}\mu_0^{\frac{n-2}{2}+\sigma}(t)}{1+|y|^{\alpha}},
\end{equation}
\item[(2)]
\begin{equation}\label{e4:38}
\begin{aligned}
&\left|2\nabla\eta_{R}\nabla\tilde{\phi}+ \tilde{\phi}\big(\Delta -\partial_t\big)\eta_{R}\right|\lesssim \frac{1}{R^{\alpha-2}}\|\phi\|_{n-2+\sigma,\alpha}\frac{\mu^{-2}\mu_0^{\frac{n-2}{2}+\sigma}(t)}{1+|y|^{\alpha}},
\end{aligned}
\end{equation}
\item[(3)]
\begin{equation}\label{e4:39}
\begin{aligned}
&\tilde{N}_{A}(\tilde{\phi}) \lesssim\\
&\left\{
\begin{aligned}
    t_0^{-\varepsilon}(\|\phi\|^2_{n-2+\sigma,\alpha}+\|\psi\|^2_{**,\beta,\alpha})\frac{1}{R^{\alpha-2}}\frac{\mu^{-2}\mu_0^{\frac{n-2}{2}+\sigma}(t)}{1+|y|^{\alpha}},
    & \quad \text{when } 6 \geq n,\\
    t_0^{-\varepsilon}(\|\phi\|^p_{n-2+\sigma,\alpha}+\|\psi\|^p_{**,\beta,\alpha})\frac{1}{R^{\alpha-2}}\frac{\mu_j^{-2}\mu_0^{\frac{n-2}{2}+\sigma}(t)}{1+|y|^{\alpha}},
    & \quad \text{when } 6 < n.\\
  \end{aligned}
\right.
\end{aligned}
\end{equation}
\end{itemize}
\end{prop}
The proof of Proposition \ref{propositionestimate} is given in Section 7.

\subsection{Proof of Theorem \ref{t:main}: solving the inner-outer gluing system.}
Let us formulate the whole problem into a fixed point problem.

{\bf Fact 1}. Let $h$ be a function satisfying $\|h\|_{n-3+\sigma} \lesssim \frac{1}{R^{\alpha-2}}$. The solution for
\begin{equation}\label{e5:14}
\dot{\lambda} + \frac{1+(n-4)}{(n-4)t}\lambda = h(t)
\end{equation}
can be expressed as follows
\begin{equation}\label{e5:15}
\lambda(t) = t^{-\frac{1+(n-4)}{(n-4)}}\left[d + \int_{t_0}^t\tau^{\frac{1+(n-4)}{(n-4)}}h(\tau)d\tau\right],
\end{equation}
with $d$ be an arbitrary constant. Therefore, it holds that
\begin{equation*}
\|t^{\frac{1+\sigma}{n-4}}\lambda(t)\|_{L^\infty(t_0,\infty)}\lesssim t_0^{-\frac{(n-4) - \sigma}{n-4}}d + \|h\|_{n-3+\sigma}
\end{equation*}
and
\begin{equation*}
\|\dot{\lambda}(t)\|_{n-3+\sigma}\lesssim t_0^{-\frac{(n-4) - \sigma}{n-4}}d + \|h\|_{n-3+\sigma}.
\end{equation*}

Set $\Lambda(t) = \dot{\lambda}(t)$, then we have
\begin{equation}\label{e5:17}
\Lambda + \frac{1}{t}\frac{1+(n-4)}{(n-4)}\int_{t}^\infty\Lambda(s)ds = h(t),
\end{equation}
which defines a bounded linear operator $\mathcal{L}_1: h\to \Lambda$ associating the solution $\Lambda$ of (\ref{e5:17}) to any $h$ satisfying $\|h\|_{n-3+\sigma} < +\infty$. Moreover, the operator $\mathcal{L}_1$ is continuous between the space $L^\infty(t_0, \infty)$ endowed with the $\|\cdot\|_{n-3+\sigma}$-topology.

For any $h: [t_0,\infty)\to \mathbb{R}^n$ with $\|h\|_{n-3+\sigma} < +\infty$, the solution of
\begin{equation}\label{e5:19}
\dot{\xi} = \mu_0^{n-2}c\left[b^{n-2}\nabla H(q, q) \right] + h(t)
\end{equation}
can be written as
\begin{equation}\label{e5:20}
\xi(t) = \xi^0(t) + \int_{t}^\infty h(s)ds,
\end{equation}
where
\begin{equation*}
\xi^0(t) = q + c\left[-b^{n-2}\nabla H(q, q)\right]\int_{t}^\infty\mu_0^{n-2}(s)ds.
\end{equation*}
Thus
\begin{equation*}
|\xi(t) - q|\lesssim t^{-\frac{2}{n-4}} + t^{-\frac{1+\sigma}{n-4}}\|h\|_{n-3+\sigma}
\end{equation*}
and
\begin{equation*}
\|\dot{\xi} - \dot{\xi}^0\|_{n-3+\sigma}\lesssim \|h\|_{n-3+\sigma}.
\end{equation*}
Define $\Xi(t) = \dot{\xi}(t) - \dot{\xi}^0$, then (\ref{e5:20}) defines a continuous linear operator $\mathcal{L}_2: h\to \Xi$ in the $\|\cdot\|_{n-3+\sigma}$-topology.

Similarly, from Proposition \ref{l5:1}, we can define $\mathcal{L}_3: h\to \Gamma:= \dot{a}(t)$ and $\mathcal{L}_4: h\to \Upsilon:= \dot{\theta}(t)$ which are continuous linear operators in the $\|\cdot\|_{n-4+\sigma}$-topology.

Note that ($\lambda$, $\xi$, $a$, $\theta$) is a solution of (\ref{e5:9}) if ($\Lambda = \dot{\lambda}(t)$, $\Xi = \dot{\xi}(t) - \dot{\xi}^0(t)$, $\Gamma:= \dot{a}(t)$, $\Upsilon:= \dot{\theta}(t)$) is a fixed point of the following problem
\begin{equation}\label{e5:23}
(\Lambda, \Xi, \Gamma, \Upsilon) = \mathcal{T}_0(\Lambda, \Xi, \Gamma, \Upsilon)
\end{equation}
where
\begin{equation*}
\begin{aligned}
\mathcal{T}_0: &= \left(\mathcal{L}_1(\hat{\Pi}_1[\Lambda, \Xi, \Gamma, \Upsilon, \phi, \psi], \mathcal{L}_2(\hat{\Pi}_2[\Lambda, \Xi, \Gamma, \Upsilon, \phi, \psi]), \right.\\
&\quad\quad\quad\quad\quad\quad\quad\quad\quad\quad\quad \left.\mathcal{L}_3(\hat{\Pi}_3[\Lambda, \Xi, \Gamma, \Upsilon, \phi, \psi], \mathcal{L}_4(\hat{\Pi}_4[\Lambda, \Xi, \Gamma, \Upsilon, \phi, \psi]\right)\\
& := \left(\bar{A}_1(\Lambda, \Xi, \Gamma, \Upsilon, \phi, \psi), \bar{A}_2(\Lambda, \Xi, \Gamma, \Upsilon, \phi, \psi), \bar{A}_3(\Lambda, \Xi, \Gamma, \Upsilon, \phi, \psi),\right.\\
&\left.\quad\quad\quad\quad\quad\quad\quad\quad\quad\quad\quad\quad\quad\quad\quad\quad\quad\quad\quad\quad\quad\quad\quad \bar{A}_4(\Lambda, \Xi, \Gamma, \Upsilon, \phi, \psi)\right)
\end{aligned}
\end{equation*}
with
\begin{equation*}
\hat{\Pi}_l[\Lambda, \Xi, \Gamma, \Upsilon, \phi, \psi] := \Pi_l\left[\int_{t}^\infty \Lambda, q + \int_{t}^\infty\Xi, \mu_0\int_{t}^\infty \Gamma, \int_{t}^\infty \Upsilon,  \Lambda, \Xi, \mu_0\Gamma, \Upsilon, \phi, \psi\right]
\end{equation*}
for $l = 0, 1, \cdots, 3n-1$.

{\bf Fact 2}. Proposition \ref{proposition5.1} tells us that there exists a linear operator $\mathcal{T}_1$ associating to the solution of (\ref{e5:31}) for any function $h(y,\tau)$ with $\|h\|_{2 + \alpha, \nu}$-bounded. Thus the solution of problem (\ref{e5:4}) is a fixed point of the problem
\begin{equation}
\phi = \mathcal{T}_1(H[\lambda,\xi, a, \theta, \dot{\lambda},\dot{\xi}, \dot{a}, \dot{\theta}, \phi, \psi](y,t(\tau))).
\end{equation}

{\bf Fact 3}. Proposition \ref{l4:lemma4.1} defines a linear operator $\mathcal{T}_2$ which associates any given functions $f(x, t)$, $g(x, t)$ and $h(x)$ to the corresponding solution $\psi = \mathcal{T}_2(f,g,h)$ for problem (\ref{e4:3}). Denote $\psi_1(x, t) := \mathcal{T}_2(0, -u^*_{A}, \psi_0)$. From (\ref{e2:52}), (\ref{e2:3}) and (\ref{e2:51}), $\forall x\in \partial\Omega$, one has
\begin{equation*}\label{e4:32}
|u^*_{A}(x,t)|\lesssim \mu_0^{\frac{n+2}{2}}(t).
\end{equation*}
From Lemma \ref{l4:lemma4.1},
\begin{equation*}\label{e4:33}
|\psi_1|\lesssim e^{-\delta(t-t_0)}\|\psi_0\|_{L^\infty(\mathbb{R}^n)} + t^{-\beta}\mu_0(t_0)^{2-\sigma}\text{ where }\beta = \frac{n-2}{2(n-4)}+\frac{\sigma}{n-4}.
\end{equation*}
Therefore, $\psi+\psi_1$ is a solution to (\ref{e4:main}) if $\psi$ is a fixed point of the following operator
\begin{equation*}\label{e4:34}
\mathcal{A}(\psi):=\mathcal{T}_2(f[\psi],0,0),
\end{equation*}
with
\begin{equation}\label{e4:35}
f[\psi]= 2\nabla\eta_{R}\nabla\tilde{\phi}+ \tilde{\phi}\big(\Delta -\partial_t\big)\eta_{R} + \tilde{N}_{A}(\tilde{\phi}) + S_{out}.
\end{equation}
That is to say, we have to solve the fixed point problem
\begin{equation}
\psi = \mathcal{T}_2(f[\psi],0,0)
\end{equation}

From {\bf Fact 1-3}, to prove Theorem \ref{t:main}, we should solve the following fixed point problem with unknowns $(\phi, \psi, \lambda, \xi, a, \theta, \dot{\lambda}, \dot{\xi}, \dot{a}, \dot{\theta})$,
\begin{equation}\label{inner_outer_gluing_system}
\left\{
\begin{aligned}
&(\Lambda, \Xi, \Gamma, \Upsilon) = \mathcal{T}_0(\Lambda, \Xi, \Gamma, \Upsilon),\\
&\phi = \mathcal{T}_1(H[\lambda,\xi, a, \theta, \dot{\lambda},\dot{\xi}, \dot{a}, \dot{\theta}, \phi, \psi](y,t(\tau))),\\
&\psi = \mathcal{T}_2(f(\psi),0,0).
\end{aligned}
\right.
\end{equation}
where
\begin{equation*}\label{e4:350}
f(\psi)= 2\nabla\eta_{R}\nabla\tilde{\phi}+ \tilde{\phi}\big(\Delta -\partial_t\big)\eta_{R} + \tilde{N}_{A}(\tilde{\phi}) + S_{out}.
\end{equation*}
To find a fixed point, we will use the Schauder fixed-point theorem in the set
\begin{equation*}\label{convexsetforschauder}
\begin{aligned}
\mathcal{B} &= \Bigg\{(\phi, \psi, \lambda, \xi, a, \theta, \dot{\lambda}, \dot{\xi}, \dot{a}, \dot{\theta}): R^{\alpha-2}\|\dot{\lambda}(t)\|_{n-3+\sigma} + R^{\alpha-2}\|\dot{\xi}(t)\|_{n-3+\sigma}\\
&\quad\quad\quad  + R^{\alpha-2}\|\dot{a}(t)\|_{n-4+\sigma} + R^{\alpha-2}\|\dot{\theta}(t)\|_{n-4+\sigma} + R^{\alpha-2}\|\lambda(t)\|_{1+\sigma}\\
&\quad\quad\quad  + R^{\alpha-2}\|\xi(t)-q\|_{1+\sigma} + R^{\alpha-2}\|a\|_{\sigma} + R^{\alpha-2}\|\theta\|_{\sigma} + t_0^{\varepsilon}R^{\alpha-2}\|\psi\|_{**, \beta, \alpha}\\
&\quad\quad\quad  + t_0^{\varepsilon}\|\phi\|_{n-2+\sigma, \alpha} \leq c\Bigg\}
\end{aligned}
\end{equation*}
for some large but fixed positive constant $c$.

Let
\begin{equation*}
K := \max\{\|f_0\|_{n-3+\sigma}, \|f_1\|_{n-3+\sigma},\cdots, \|f_{3n-1}\|_{n-3+\sigma}\}
\end{equation*}
where $f_0$, $f_1$, $\cdots$, $f_{3n-1}$ are the functions defined in Lemma \ref{l5:1}. Then we have
\begin{equation*}
\begin{aligned}
&\left|t^{\frac{n-3+\sigma}{n-4}}\bar{A}_i(\Lambda, \Xi, \Gamma, \Upsilon, \phi, \psi)\right|\\
&\lesssim t_0^{-\frac{(n-4) - \sigma}{n-4}}d + \frac{1}{R^{\alpha-2}}\|\phi\|_{n-2+\sigma, a} + \frac{1}{R^{\alpha-2}}\|\psi\|_{**, \beta, \alpha} + \frac{K}{R^{\alpha-2}}\\
&\quad + \frac{1}{R^{\alpha-2}}\|\Lambda\|_{n-3+\sigma} + \frac{1}{R^{\alpha-2}}\|\Xi\|_{n-3+\sigma}
\end{aligned}
\end{equation*}
Thus, for $d$ satisfying $t_0^{-\frac{(n-4)-\sigma}{n-4}}d < \frac{K}{R^{\alpha-2}}$, $\mathcal{T}_0(\mathcal{B})\subset \mathcal{B}$ (choose the constant $\rho$ in (\ref{e3:20}) sufficiently small).

On the set $\mathcal{B}$, it is clear that
\begin{equation*}\label{e6:1}
\begin{aligned}
\left|H[\lambda,\xi, a, \theta, \dot{\lambda},\dot{\xi}, \dot{a}, \dot{\theta}, \phi, \psi](y,t(\tau))\right|\lesssim t_0^{-\varepsilon}\frac{\mu_0^{n-2+\sigma}}{1+|y|^{2+\alpha}}
\end{aligned}
\end{equation*}
From Proposition \ref{proposition5.1}, $\mathcal{T}_1(\mathcal{B})\subset \mathcal{B}$ holds.

Similarly, Proposition \ref{propositionestimate} ensures that $\mathcal{T}_2(\mathcal{B})\subset \mathcal{B}$. Therefore the operator $\mathcal{T}$ defined in (\ref{inner_outer_gluing_system}) maps the set $\mathcal{B}$ into itself. Since $\lambda$, $\xi$, $a$, $\theta$, $\dot{\lambda}$, $\dot{\xi}$, $\dot{a}$, $\dot{\theta}$, $\phi$ and $\psi$ decay uniformly when $t\to +\infty$, this fact combines with the standard parabolic estimate ensures that $\mathcal{T}$ is compact. By the Schauder fixed-point theorem, we conclude that (\ref{inner_outer_gluing_system}) has a fixed point in $\mathcal{B}$. That is to say, we find a solution to the system of outer problem (\ref{outerproblem}) and inner problem (\ref{e3:10}), which provides a solution to (\ref{e:main}). This completes the proof of Theorem \ref{t:main}.

\section{Proof of Proposition \ref{proposition5.1}}
In the following, we assume that $h = h(y, \tau)$ is a function defined on $\mathbb{R}^n$ which is zero outside the ball $B_{2R}(0)$ for all $\tau > \tau_0$.
As a first step to the proof of proposition \ref{proposition5.1}, we have the following
\begin{lemma}\label{l5:3}
Suppose $\alpha\in (2, n-2)$, $\nu > 0$, $\|h\|_{2+\alpha,\nu} < +\infty$ and
\begin{equation*}
\int_{\mathbb{R}^n}h(y,\tau)z_j(y)dy = 0~\text{ for all }~\tau\in (\tau_0,\infty), ~j = 0, 1, \cdots, 3n-1.
\end{equation*}
Then for any $\tau_1 > \tau_0$ large enough, the solution $(\phi(y,\tau), c_1(\tau), \cdots, c_K(\tau))$ to the following problem
\begin{eqnarray}\label{e5:32}
\left\{
\begin{aligned}
&\partial_\tau\phi = \Delta\phi + p|Q|^{p-1}(y)\phi + h(y,\tau)-\sum_{l=1}^Kc_l(\tau)Z_l(y),~ y\in \mathbb{R}^n,~\tau\geq \tau_0,\\
&\int_{\mathbb{R}^n}\phi(y,\tau)Z_l(y)dy = 0~\mbox{ for all }~ \tau\in (\tau_0,+\infty), ~ l = 1,\cdots, K,\\
&\phi(y,\tau_0) = 0,~y\in \mathbb{R}^n,
\end{aligned}
\right.
\end{eqnarray}
satisfies
\begin{equation}\label{e5:35}
\|\phi(y,\tau)\|_{\alpha,\tau_1}\lesssim \|h\|_{2+\alpha,\tau_1}
\end{equation}
and $\forall l = 1,\cdots, K$,
\begin{equation*}
|c_l(\tau)|\lesssim \tau^{-\nu}R^\alpha\|h\|_{2+\alpha,\tau_1}~ \text{ for }~\tau\in (\tau_0,\tau_1).
\end{equation*}
Here $\|h\|_{b,\tau_1}:=\sup_{\tau\in (\tau_0,\tau_1)}\tau^\nu\|(1+|y|^b)h\|_{L^\infty(\mathbb{R}^n)}$.
\end{lemma}
\begin{proof}
(\ref{e5:32}) is equivalent to
\begin{eqnarray}\label{e5:99}
\left\{
\begin{aligned}
&\partial_\tau\phi = \Delta\phi + p|Q|^{p-1}(y)\phi + h(y,\tau)-\sum_{l=1}^Kc_l(\tau)Z_l(y),~ y\in \mathbb{R}^n,~ \tau\geq \tau_0,\\
&\phi(y,\tau_0) = 0,~ y\in \mathbb{R}^n
\end{aligned}
\right.
\end{eqnarray}
with $c_l(\tau)$ given by the following relation
\begin{equation*}
c_l(\tau) \int_{\mathbb{R}^n}|Z_l(y)|^2dy =  \int_{\mathbb{R}^n}h(y,\tau)Z_l(y)dy, \quad l = 1,\cdots, K.
\end{equation*}
Then
\begin{equation}\label{e5:102}
|c_l(\tau)|\lesssim \tau^{-\nu}R^\alpha\|h\|_{2+\alpha,\tau_1}
\end{equation}
holds for $\tau\in (\tau_0,\tau_1)$. Therefore we are left with the proof of (\ref{e5:35}) for the solution $\phi$ of equation (\ref{e5:99}). Inspired by Lemma 4.5 of \cite{davila2017singularity}, the linear theory of \cite{MussoSireWeiZhengZhou} and \cite{sireweizhenghalf}, we use the blowing-up argument.

First, we have {\bf Claim}: given $\tau_1 > \tau_0$, $\|\phi\|_{\alpha, \tau_1} < +\infty$ holds. Indeed, given $R_0 > 0$, the standard parabolic theory ensures that there is a constant $K_1 = K_1(R_0,\tau_1)$ such that
\begin{equation*}
|\phi(y,\tau)|\leq K_1 \quad\text{in }B_{R_0}(0)\times (\tau_0, \tau_1].
\end{equation*}
Let us fix $R_0 > 0$ large enough and take $K_2 > 0$ large enough, then $K_2\rho^{-\alpha}$ is a super-solution of (\ref{e5:99}) when $\rho > R_0$. Therefore, for any $\tau_1 > 0$, $|\phi|\leq 2K_2\rho^{-\alpha}$ and $\|\phi\|_{\alpha,\tau_1} < +\infty$. Next, we prove the following identities,
\begin{equation}\label{e5:33}
\int_{\mathbb{R}^n}\phi(y, \tau)z_j(y)dy = 0\text{ for all }\tau\in (\tau_0,\tau_1),~ j= 0,1,\cdots, 3n-1
\end{equation}
and
\begin{equation}\label{e5:33(1)}
\int_{\mathbb{R}^n}\phi(y, \tau)Z_l(y)dy = 0\text{ for all }\tau\in (\tau_0,\tau_1),~ l = 1,\cdots, K.
\end{equation}
Indeed, $(\ref{e5:33(1)})$ follows from the definition of $c_l(\tau)$. Let us test (\ref{e5:99}) with $z_j\eta$, where $\eta(y) = \eta_0(|y|/\tilde{R}),\,\, j = 0, 1,\cdots, 3n-1,$  $\tilde{R}$ is a positive constant and $\eta_0$ is a smooth cut-off function defined by
\begin{equation*}
\eta_0(r) =
\left\{
\begin{aligned}
1,~ \mbox{for}~ r < 1,\\
0,~ \mbox{for}~ r > 2.
\end{aligned}
\right.
\end{equation*}
Then we have
\begin{equation*}
\int_{\mathbb{R}^n}\phi(\cdot, \tau) z_j\eta = \int_{0}^\tau ds\int_{\mathbb{R}^n}(\phi(\cdot, s) L_0[\eta z_j] + h z_j\eta - \sum_{l=1}^Kc_l(s)Z_lz_j\eta).
\end{equation*}
Furthermore,
\begin{equation*}
\begin{aligned}
&\int_{\mathbb{R}^n}\bigg(\phi L_0[\eta z_j] + h z_j\eta - \sum_{l=1}^Kc_l(s)Z_lz_j\eta\bigg)\\
&= \int_{\mathbb{R}^n}\phi \bigg(z_j\Delta\eta + 2\nabla\eta\nabla z_j\bigg)\\
&\quad - h z_j(1-\eta) + \sum_{l=1}^Kc_l(s)Z_lz_j(1-\eta)\\
&= O(\tilde{R}^{-\varepsilon})
\end{aligned}
\end{equation*}
holds uniformly on $\tau\in (\tau_0,\tau_1)$ for a small positive number $\varepsilon$. Letting $\tilde{R}\to +\infty$, we get (\ref{e5:33}).
Finally, we claim that when $\tau_1 > \tau_0$ is large enough, for any solution $\phi$ of (\ref{e5:99}) satisfying $\|\phi\|_{\alpha,\tau_1} < +\infty$, (\ref{e5:33}) and (\ref{e5:33(1)}), there holds
\begin{equation}\label{e5:34}
\|\phi\|_{\alpha,\tau_1}\lesssim \|h\|_{2+\alpha,\tau_1}.
\end{equation}
This proves (\ref{e5:35}).

To prove estimate (\ref{e5:34}), we use the contradiction arguments. Suppose there are sequences $\tau_1^k\to +\infty$ and $\phi_k$, $h_k$, $c_l^k$ ($l = 1, \cdots, K$) satisfying the following parabolic problem
\begin{equation*}\label{e5:36}
\left\{
\begin{aligned}
&\partial_\tau\phi_k = \Delta\phi_k + p|Q|^{p-1}(y)\phi_k + h_k - \sum_{l=1}^Kc^k_l(s)Z_l(y),~ y\in \mathbb{R}^n,~ \tau\geq \tau_0,\\
&\int_{\mathbb{R}^n}\phi_k(y, \tau) z_j(y)dy = 0\text{ for all }\tau\in (\tau_0,\tau_1^k), j= 0, 1,\cdots, 3n-1,\\
&\int_{\mathbb{R}^n}\phi_k(y, \tau) Z_l(y)dy = 0\text{ for all }\tau\in (\tau_0,\tau_1),~ l = 1,\cdots, K,\\
&\phi_k(y,\tau_0) = 0, y\in \mathbb{R}^n
\end{aligned}
\right.
\end{equation*}
and
\begin{equation}\label{e5:38}
\|\phi_k\|_{\alpha,\tau_1^k}=1,\quad \|h_k\|_{2+\alpha,\tau_1^k}\to 0.
\end{equation}
By (\ref{e5:102}), we obtain $\sup_{\tau\in (\tau_0, \tau_1^k)}\tau^\nu c^k_l(\tau)\to 0$, $l = 1,\dots, K$.
First, we claim that the following holds
\begin{equation}\label{e5:37}
\sup_{\tau_0 < \tau < \tau_1^k}\tau^\nu|\phi_k(y,\tau)|\to 0
\end{equation}
uniformly on compact subsets of $\mathbb{R}^n$. Indeed, if for some $|y_k|\leq M$, $\tau_0 < \tau_2^k < \tau_1^k$,
\begin{equation*}
(\tau_2^k)^\nu|\phi_k(y_k,\tau_2^k)|\geq \frac{1}{2},
\end{equation*}
then we have $\tau_2^k\to +\infty$. Now, define
\begin{equation*}
\tilde{\phi}_n(y,\tau) = (\tau_2^k)^\nu\phi_n(y,\tau_2^k + \tau).
\end{equation*}
Then
\begin{equation*}
\partial_\tau\tilde{\phi}_k = L[\tilde{\phi}_k] + \tilde{h}_k - \sum_{l=1}^K\tilde{c}_l^k(\tau)Z_l(y)\text{ in }\mathbb{R}^n\times (\tau_0-\tau_2^k,0],
\end{equation*}
with $\tilde{h}_k\to 0$, $\tilde{c}_l^k\to 0$ ($l = 1, \cdots, K$) uniformly on compact subsets in $\mathbb{R}^n\times (-\infty, 0]$, moreover, we have
\begin{equation*}
|\tilde{\phi}_k(y,\tau)|\leq \frac{1}{1+|y|^\alpha}\text{ in }\mathbb{R}^n\times (\tau_0-\tau_2^k,0].
\end{equation*}
Using the dominant convergence theorem and the fact that $\alpha\in (2, n-2)$, $\tilde{\phi}_k\to\tilde{\phi}$ uniformly on compact subsets in $\mathbb{R}^n\times (-\infty, 0]$ for a function $\tilde{\phi}\neq 0$ satisfying
\begin{equation}\label{e5:103}
\left\{
\begin{aligned}
&\partial_\tau\tilde{\phi} = \Delta\tilde{\phi} + p|Q|^{p-1}(y)\tilde{\phi}~~\text{ in }~~\mathbb{R}^n\times (-\infty, 0],\\
&\int_{\mathbb{R}^n}\tilde{\phi}(y, \tau) z_j(y)dy = 0\text{ for all }\tau\in (-\infty, 0], ~ j= 0, 1,\cdots, 3n-1,\\
&\int_{\mathbb{R}^n}\tilde{\phi}(y, \tau) Z_l(y)dy = 0\text{ for all }\tau\in (-\infty, 0], ~ l= 1, \cdots, K,\\
&|\tilde{\phi}(y,\tau)|\leq \frac{1}{1+|y|^\alpha}\text{ in }\mathbb{R}^n\times (-\infty, 0],\\
&\tilde{\phi}(y,\tau_0) = 0, \quad y\in \mathbb{R}^n.
\end{aligned}
\right.
\end{equation}
Now we claim that $\tilde{\phi} = 0$, which contradicts to the fact that $\tilde{\phi}\neq 0$. Standard parabolic regularity tells us that $\tilde{\phi}(y,\tau)$ is $C^{2, \varrho}$ for some $\varrho\in (0, 1)$. Then a scaling argument shows that
\begin{equation*}
(1+|y|)|\nabla\tilde{\phi}| + |\tilde{\phi}_\tau| + |\Delta\tilde{\phi}|\lesssim (1+|y|)^{-2-\alpha}.
\end{equation*}
Differentiating (\ref{e5:103}) with respect to $\tau$, we have $\partial_\tau\tilde{\phi}_\tau = \Delta\tilde{\phi}_\tau + p|Q|^{p-1}(y)\tilde{\phi}_\tau$ and
\begin{equation*}
(1+|y|)|\nabla\tilde{\phi}_\tau| + |\tilde{\phi}_{\tau\tau}| + |\Delta\tilde{\phi}_\tau|\lesssim (1+|y|)^{-4-\alpha}.
\end{equation*}
Furthermore, it holds that
\begin{equation*}
\frac{1}{2}\partial_\tau\int_{\mathbb{R}^n}|\tilde{\phi}_\tau|^2 + B(\tilde{\phi}_\tau, \tilde{\phi}_\tau) = 0,
\end{equation*}
where
\begin{equation*}
B(\tilde{\phi}, \tilde{\phi}) = \int_{\mathbb{R}^n}\left[|\nabla\tilde{\phi}|^2 - p|Q|^{p-1}(y)|\tilde{\phi}|^2\right]dy.
\end{equation*}
Since $\int_{\mathbb{R}^n}\tilde{\phi}(y, \tau) z_j(y)dy = 0$ and $\int_{\mathbb{R}^n}\tilde{\phi}(y, \tau)Z_l(y)dy = 0$ hold $\forall\tau\in (-\infty, 0]$, $j= 0, 1,\cdots, 3n-1 $, $l= 1,\cdots, K$, we have $B(\tilde{\phi}, \tilde{\phi})\geq 0$. Note that
\begin{equation*}
\int_{\mathbb{R}^n}|\tilde{\phi}_\tau|^2 = -\frac{1}{2}\partial_\tau B(\tilde{\phi}, \tilde{\phi}).
\end{equation*}
Combine the above facts, we get,
\begin{equation*}
\partial_\tau\int_{\mathbb{R}^n}|\tilde{\phi}_\tau|^2 \leq 0,\quad \int_{-\infty}^0d\tau\int_{\mathbb{R}^n}|\tilde{\phi}_\tau|^2 < +\infty.
\end{equation*}
Hence $\tilde{\phi}_\tau = 0$. Thus $\tilde{\phi}$ is independent of $\tau$, $L[\tilde{\phi}] = 0$. Since $\tilde{\phi}$ is bounded, from the nondegeneracy of $L$,  $\tilde{\phi}$ is a linear combination of the kernel functions $z_j$, $j = 0, 1,\cdots, 3n-1$. But $\int_{\mathbb{R}^n}\tilde{\phi} z_j = 0$, $j = 0,1,\cdots, 3n-1$, we get $\tilde{\phi} = 0$, a contradiction. Therefore (\ref{e5:37}) holds.

From (\ref{e5:38}), there exists a sequence $y_k$ with $|y_k|\to +\infty$ such that
\begin{equation*}
(\tau_2^k)^\nu(1+|y_k|^\alpha)|\phi_k(y_k, \tau_2^k)|\geq \frac{1}{2}.
\end{equation*}
Let
\begin{equation*}
\tilde{\phi}_k(z, \tau):=(\tau_2^k)^\nu|y_k|^\alpha\phi_k(y_k+|y_k|z,|y_k|^{2}\tau + \tau_2^k),
\end{equation*}
then
\begin{equation*}
\partial_\tau \tilde{\phi}_k = \Delta\tilde{\phi}_k + a_k\tilde{\phi}_k + \tilde{h}_k(z,\tau),
\end{equation*}
with
\begin{equation*}
\tilde{h}_k(z,\tau) = (\tau_2^k)^\nu|y_k|^{2+\alpha}h_k(y_k+|y_k|z,|y_k|^{2}\tau + \tau_2^k).
\end{equation*}
From the assumptions on $h_k$, one gets
\begin{equation*}
|\tilde{h}_k(z,\tau)| \lesssim o(1)|\hat{y}_k+z|^{-2-\alpha}((\tau_2^k)^{-1}|y_k|^{2}\tau + 1)^{-\nu}
\end{equation*}
with
\begin{equation*}
\hat{y}_k = \frac{y_k}{|y_k|}\to -\hat{e}
\end{equation*}
and $|\hat{e}|= 1$. Hence $\tilde{h}_k(z,\tau)\to 0$ uniformly on compact subsets in $\mathbb{R}^n\setminus\{\hat{e}\}\times (-\infty, 0]$. $a_k$ has the same property as $\tilde{h}_k(z,\tau)$. Furthermore, $|\tilde{\phi}_k(0, \tau_0)|\geq \frac{1}{2}$ and
\begin{equation*}
|\tilde{\phi}_k(z,\tau)| \lesssim |\hat{y}_k+z|^{-\alpha}\left((\tau_2^k)^{-1}|y_k|^{2}\tau + 1\right)^{-\nu}.
\end{equation*}
Hence one may assume $\tilde{\phi}_k\to \tilde{\phi}\neq 0$ uniformly on compact subsets in $\mathbb{R}^n\setminus\{\hat{e}\}\times (-\infty,0]$ for $\tilde{\phi}$ satisfying
\begin{equation}\label{e5:39}
\tilde{\phi}_\tau = \Delta\tilde{\phi}\quad\text{in }\mathbb{R}^n\setminus\{\hat{e}\}\times (-\infty,0]
\end{equation}
and
\begin{equation}\label{e5:400}
|\tilde{\phi}(z,\tau)|\leq |z-\hat{e}|^{-\alpha}\quad\text{in }\mathbb{R}^n\setminus\{\hat{e}\}\times (-\infty,0].
\end{equation}
Similar to Lemma 5.2 of \cite{sireweizhenghalf}, functions $\tilde{\phi}$ satisfying (\ref{e5:39}) and (\ref{e5:400}) is zero, which is a contradiction to the fact that $\tilde{\phi}\neq 0$. This concludes the validity of (\ref{e5:34}). Indeed, set
$$
u(\rho, t) = (\rho^2 + Ct)^{-\alpha/2} + \frac{\varepsilon}{\rho^{n-2}}.
$$
Then
$$
-u_t + \Delta u < (\rho^2 + Ct)^{-\alpha/2 -1}[\alpha(\alpha+2-n)+\frac{C}{2}\alpha]<0, \text{ if } \alpha < n-2 - \frac{C}{2}.
$$
For any $\alpha < n-2$, we can always find a fixed $C > 0$ such that $\alpha<n-2-\frac{C}{2}$.
Hence $u(|z-\hat{e}|,\tau +M)$ is a positive super-solution of (\ref{e5:400}) in
$(0,\infty)\times [-M, 0]$. Via the comparison principle, $|\tilde{\phi}(z,\tau)| \leq 2 u(|z-\hat{e}|,\tau +M)$. Letting $M\to +\infty$ we get
$$
|\tilde{\phi}(z,\tau)|\leq\frac {2\varepsilon}{|z-\hat{e}|^{n-2}}.
$$
Since $\varepsilon > 0$ is arbitrary, we conclude that $\tilde{\phi}=0$. The proof is completed.
\end{proof}
{\it Proof of Proposition \ref{proposition5.1}}.
First let us consider the following problem
\begin{equation*}
\left\{
\begin{aligned}
&\partial_\tau\phi = \Delta\phi + p|Q|^{p-1}(y)\phi + h(y,\tau) - \sum_{l=1}^Kc_l(\tau)Z_l,~ y\in \mathbb{R}^n,~ \tau\geq \tau_0,\\
&\phi(y,\tau_0) = 0,~ y\in \mathbb{R}^n.
\end{aligned}
\right.
\end{equation*}
Let $(\phi(y,\tau),c_1(\tau),\cdots, c_K(\tau))$ be the unique solution to problem (\ref{e5:32}). By Lemma \ref{l5:3}, for $\tau_1 > \tau_0$ large enough, there hold
\begin{equation*}
|\phi(y,\tau)|\lesssim\tau^{-\nu}(1+|y|)^{-\alpha}\|h\|_{2+\alpha, \tau_1}\text{ for all }\tau\in (\tau_0, \tau_1), \,\,y\in \mathbb{R}^n
\end{equation*}
and
\begin{equation*}
|c_l(\tau)|\leq \tau^{-\nu}R^{\alpha}\|h\|_{2+\alpha,\tau_1}\text{ for all }\tau\in (\tau_0,\tau_1),\quad l = 1,\cdots, K.
\end{equation*}
From the assumptions of the proposition, for an arbitrary $\tau_1$, $\|h\|_{2+\alpha,\nu} < +\infty$ and $\|h\|_{2+\alpha, \tau_1}\leq \|h\|_{2+\alpha,\nu}$ hold. Therefore, one has
\begin{equation*}
|\phi(y,\tau)|\lesssim\tau^{-\nu}(1+|y|)^{-\alpha}\|h\|_{2+\alpha,\nu}\text{ for all }\tau\in (\tau_0, \tau_1),\,\, y\in \mathbb{R}^n
\end{equation*}
and
\begin{equation*}
|c_l(\tau)|\leq \tau^{-\nu}R^{\alpha}\|h\|_{2+\alpha,\nu}\text{ for all }\tau\in (\tau_0, \tau_1), \quad l = 1\cdots, K.
\end{equation*}
From the arbitrariness of $\tau_1$, we have
\begin{equation*}
|\phi(y,\tau)|\lesssim\tau^{-\nu}(1+|y|)^{-\alpha}\|h\|_{2+\alpha,\nu}\text{ for all }\tau\in (\tau_0, +\infty),\,\, y\in \mathbb{R}^n
\end{equation*}
and
\begin{equation*}
|c_l(\tau)|\leq \tau^{-\nu}R^{\alpha}\|h\|_{2+\alpha,\nu}\text{ for all }\tau\in (\tau_0, +\infty), \quad l = 1\cdots, K.
\end{equation*}
Using the parabolic regularity results and a scaling argument, we get (\ref{e5:100}) and (\ref{e5:101}).\qed

\section{Proof of Proposition \ref{l5:1}}
The following integral identities will be useful in the computation of this section.
\begin{lemma}\label{e:orthogalityofkernels}
As $k \to +\infty$, for $j = 0, \cdots, 3n-1$, we have
\begin{equation*}
\int_{\mathbb{R}^n}\left(z_0(y) - \frac{D_{n,k}(2-|y|^2)}{\left(1+|y|^2\right)^{\frac{n}{2}}}\right)z_j(y)dy = \left\{ \begin{matrix}  a_{0,0} + O(k^{-1} )  & \hbox{ if } j = 0\, , \\ & \\     O(k^{-1})  & \hbox{ if } j \neq 0,\end{matrix}
  \right.
\end{equation*}
\begin{equation*}
\int_{\mathbb{R}^n}\left(\frac{\partial}{\partial y_1} Q(y) - \frac{E_{n,k}y_1}{\left(1+|y|^2\right)^{\frac{n}{2}}}\right)z_j(y)dy = \left\{ \begin{matrix}a_{1,1}+O(k^{-1}) & \hbox{ if } j = 1\, , \\
&\\  a_{1,n+2} +O(k^{-1}) &\text{ if } j = n+2,\,\,\\
&\\  O(k^{-1})&\hbox{ if } j \neq 1, n+2,
\end{matrix}
\right.
\end{equation*}
\begin{equation*}
\int_{\mathbb{R}^n}\left(\frac{\partial}{\partial y_2} Q(y) - \frac{E_{n,k}y_2}{\left(1+|y|^2\right)^{\frac{n}{2}}}\right)z_j(y)dy = \left\{ \begin{matrix}a_{2,2}+O(k^{-1}) & \hbox{ if } j = 2\, , \\
&\\ a_{2,n+3}+O(k^{-1})&\text{ if } j = n+3,\,\,\\
&\\ O(k^{-1})&\hbox{ if } j \neq 2, n+3.
\end{matrix}
\right.
\end{equation*}
For $i = 3, \cdots, n$, $j = 0, \cdots, 3n-1$, we have
\begin{equation*}
\int_{\mathbb{R}^n}\left(\frac{\partial}{\partial y_i} Q(y) - \frac{E_{n,k}y_i}{\left(1+|y|^2\right)^{\frac{n}{2}}}\right)z_j(y)dy = \left\{
\begin{matrix}
a_{i,i} + O(k^{-1} )  & \hbox{ if } j = i\, , \\
& \\     O(k^{-1})  & \hbox{ if } j \neq i.
\end{matrix}
\right.
\end{equation*}
Furthermore,
\begin{equation*}
\int_{\mathbb{R}^n}z_{n+1}(y)z_j(y)dy = \left\{
\begin{matrix}
a_{n+1,n+1} + O(k^{-1} )  & \hbox{ if } j = n+1\, , \\
& \\     O(k^{-1})  & \hbox{ if } j \neq n+1,
\end{matrix}
\right.
\end{equation*}
\begin{equation*}
\begin{aligned}
\int_{\mathbb{R}^n}\left(-2y_1\left(z_0(y) - \frac{D_{n,k}(2-|y|^2)}{\left(1+|y|^2\right)^{\frac{n}{2}}}\right)\right. +& \left. |y|^2\left(\frac{\partial}{\partial y_1} Q(y) - \frac{E_{n,k}y_1}{\left(1+|y|^2\right)^{\frac{n}{2}}}\right)\right)z_j(y)dy\\
&= \left\{
\begin{matrix}
a_{n+2, 1} + O(k^{-1}) & \hbox{ if } j = 1\, , \\
& \\    a_{n+2, n+2} + O(k^{-1})  & \hbox{ if } j = n+2\, , \\
& \\     O(k^{-1})  & \hbox{ if } j \neq 1, n+2,\end{matrix}
\right.
\end{aligned}
\end{equation*}
\begin{equation*}
\begin{aligned}
\int_{\mathbb{R}^n}\left(-2y_2\left(z_0(y) - \frac{D_{n,k}(2-|y|^2)}{\left(1+|y|^2\right)^{\frac{n}{2}}}\right)\right. +& \left. |y|^2\left(\frac{\partial}{\partial y_2} Q(y) - \frac{E_{n,k}y_2}{\left(1+|y|^2\right)^{\frac{n}{2}}}\right)\right)z_j(y)dy\\
&= \left\{
\begin{matrix}
a_{n+3, 2} + O(k^{-1}) & \hbox{ if } j = 2\, , \\
&\\ a_{n+3, n+3} + O(k^{-1} )  & \hbox{ if } j = n+3\, , \\
& \\     O(k^{-1})  & \hbox{ if } j \neq 2, n+3.\end{matrix}
\right.
\end{aligned}
\end{equation*}
For $i = 3, \cdots, n$,
\begin{equation*}
\int_{\mathbb{R}^n}z_{n+i+1}(y)z_j(y)dy = \left\{
\begin{matrix}
a_{n+i+1,n+i+1} + O(k^{-1} )  & \hbox{ if } j = n+i+1\, , \\
& \\     O(k^{-1})  & \hbox{ if } j \neq n+i+1,
\end{matrix}
\right.
\end{equation*}
\begin{equation*}
\int_{\mathbb{R}^n}z_{2n+i-1}(y)z_j(y)dy = \left\{
\begin{matrix}
a_{2n+i-1,2n+i-1} + O(k^{-1} )  & \hbox{ if } j = 2n+i-1\, , \\
& \\     O(k^{-1})  & \hbox{ if } j \neq 2n+i-1.
\end{matrix}
\right.
\end{equation*}
In the above, $a_{i, j}$ are positive constants depending on $n$ and $k$, the matrices
\begin{equation*}
\begin{pmatrix}
a_{1,1} & a_{1, n+2}\\
a_{n+2, 1} & a_{n+2, n+2}
\end{pmatrix}, \quad
\begin{pmatrix}
a_{2,2} & a_{2, n+3}\\
a_{n+3, 2} & a_{n+3, n+3}
\end{pmatrix}
\end{equation*}
are invertible.
\end{lemma}
The proof of this lemma is given in the Appendix.

\subsection{The equation for $\lambda$.} We consider (\ref{e5:7}) for $l = 0$.
\begin{lemma}\label{l5:100}
When $l = 0$, (\ref{e5:7}) is equivalent to
\begin{equation}\label{e5:900}
\begin{aligned}
\dot{\lambda} &+ \frac{1+(n-4)}{(n-4)t}\lambda + O\left(\frac{1}{k}\right)\dot{\xi} + O\left(\frac{1}{k}\right)\mu_0\left(\dot{a}_1+\dot{a}_2\right) \\
& + O\left(\frac{1}{k}\right)\mu_0\left(\dot{\theta}_{12} + \sum_{j = 3}^n(\dot{\theta}_{1j} + \dot{\theta}_{2j})\right) = \Pi_0[\lambda,\xi, a, \theta, \dot{\lambda},\dot{\xi}, \dot{a}, \dot{\theta}, \phi, \psi](t).
\end{aligned}
\end{equation}
The right hand side term of (\ref{e5:900}) can be expressed as
\begin{equation*}
\begin{aligned}
&\Pi_0[\lambda,\xi, a, \theta, \dot{\lambda},\dot{\xi}, \dot{a}, \dot{\theta}, \phi, \psi](t) = \frac{t_0^{-\varepsilon}}{R^{\alpha-2}}\mu_0^{n-3 + \sigma}(t)f_0(t)+ \frac{t_0^{-\varepsilon}}{R^{\alpha-2}}\times\\
&\quad \Theta_0\left[\dot{\lambda},\dot{\xi}, \mu_0\dot{a}, \mu_0\dot{\theta}, \mu_0^{n-4}(t)\lambda, \mu_0^{n-4}(\xi-q), \mu_0^{n-3}a, \mu_0^{n-3}\theta, \mu_0^{n-3+\sigma}\phi, \mu_0^{\frac{n-2}{2}+\sigma}\psi\right](t)
\end{aligned}
\end{equation*}
where $f_0(t)$ and $\Theta_0[\cdots](t)$ are bounded smooth functions for $t\in [t_0,\infty)$.
\end{lemma}
\begin{proof}
We compute
\begin{eqnarray*}
\int_{B_{2R}}H[\lambda,\xi, a, \theta, \dot{\lambda},\dot{\xi}, \dot{a}, \dot{\theta}, \phi, \psi](y,t(\tau))z_{0}(y)dy,
\end{eqnarray*}
where $H[\lambda,\xi, a, \theta, \dot{\lambda},\dot{\xi}, \dot{a}, \dot{\theta}, \phi, \psi](y,t(\tau))$ is defined in (\ref{e5:2}). Write
\begin{equation}\label{e:601}
\begin{aligned}
&\mu_{0}^{\frac{n+2}{2}}S_{A}(\xi + \mu_{0}y, t)\\
&\quad =\left(\frac{\mu_{0}}{\mu}\right)^{\frac{n+2}{2}}[\mu_{0}S_1(z, t) + \lambda b S_2(z, t) + \mu S_3(z, t) + \mu^2S_4(z, t) + \mu^2S_5(z, t)]_{z = \xi + \mu y}\\
&\quad\quad+\left(\frac{\mu_{0}}{\mu}\right)^{\frac{n+2}{2}}\mu_{0}[S_1(\xi+\mu_{0}y, t)-S_1(\xi+\mu y, t)]\\
&\quad\quad+\left(\frac{\mu_{0}}{\mu}\right)^{\frac{n+2}{2}}\lambda b [S_2(\xi+\mu_{0}y, t)-S_2(\xi+\mu y, t)]\\
&\quad\quad+\left(\frac{\mu_{0}}{\mu}\right)^{\frac{n+2}{2}}\mu [S_3(\xi+\mu_{0}y, t)-S_3(\xi +\mu y, t)]\\
&\quad\quad+\left(\frac{\mu_{0}}{\mu}\right)^{\frac{n+2}{2}}\mu^2 [S_4(\xi+\mu_{0}y, t)-S_4(\xi +\mu y, t)]\\
&\quad\quad+\left(\frac{\mu_{0}}{\mu}\right)^{\frac{n+2}{2}}\mu^2 [S_5(\xi+\mu_{0}y, t)-S_5(\xi +\mu y, t)],
\end{aligned}
\end{equation}
where
\begin{equation*}
\begin{aligned}
S_1(z) &= \dot{\lambda}\left(z_0\left(\frac{z-\xi}{\mu}\right) - \frac{D_{n,k}\left(2 - \left|\frac{z-\xi}{\mu}\right|^2\right)}{\left(1+\left|\frac{z-\xi}{\mu}\right|^2\right)^{\frac{n}{2}}}-2Ap|Q|^{p-1}\left(\frac{z-\xi}{\mu}\right)\right)\\
&\quad -\mu_0^{n-4}p|Q|^{p-1}\left(\frac{z-\xi}{\mu}\right)\left[(n-3)b^{n-4}H(q, q)\lambda\right],
\end{aligned}
\end{equation*}
\begin{equation*}
\begin{aligned}
S_2(z) &= \dot{\mu}_0\left(z_0\left(\frac{z-\xi}{\mu}\right) - \frac{D_{n,k}\left(2 - \left|\frac{z-\xi}{\mu}\right|^2\right)}{\left(1+\left|\frac{z-\xi}{\mu}\right|^2\right)^{\frac{n}{2}}}\right)\\
 &\quad + p|Q|^{p-1}\left(\frac{z-\xi}{\mu}\right)\mu_0^{n-3}\bigg(-b^{n-4}H(q, q) + B\bigg),
\end{aligned}
\end{equation*}
\begin{equation*}
\begin{aligned}
S_3(z) &= \left(\nabla Q\left(\frac{z-\xi}{\mu}\right) -\frac{E_{n,k}\frac{z-\xi}{\mu}}{\left(1+\left|\frac{z-\xi}{\mu}\right|^2\right)^{\frac{n}{2}}}\right)\cdot \dot{\xi}\\ &\quad + p|Q|^{p-1}\left(\frac{z-\xi}{\mu}\right)\left[-\mu^{n-2}\nabla H(q, q)\right]\cdot \left(\frac{z-\xi}{\mu}\right),
\end{aligned}
\end{equation*}
\begin{equation*}
\begin{aligned}
S_4(z) &= \dot{a}_1 \left(-2\left(\frac{z-\xi}{\mu}\right)_1\left(z_0\left(\frac{z-\xi}{\mu}\right) - \frac{D_{n,k}(2-\left|\frac{z-\xi}{\mu}\right|^2)}{\left(1+\left|\frac{z-\xi}{\mu}\right|^2\right)^{\frac{n}{2}}}\right)\right.\\
&\quad\quad\quad\quad\quad\quad\quad\quad\quad\quad \left. + \left|\frac{z-\xi}{\mu}\right|^2\left(\frac{\partial}{\partial y_1} Q\left(\frac{z-\xi}{\mu}\right) - \frac{E_{n,k}\left(\frac{z-\xi}{\mu}\right)_1}{\left(1+\left|\frac{z-\xi}{\mu}\right|^2\right)^{\frac{n}{2}}}\right)\right)\\
&\quad + \dot{a}_2 \left(-2\left(\frac{z-\xi}{\mu}\right)_2\left(z_0\left(\frac{z-\xi}{\mu}\right) - \frac{D_{n,k}(2-\left|\frac{z-\xi}{\mu}\right|^2)}{\left(1+\left|\frac{z-\xi}{\mu}\right|^2\right)^{\frac{n}{2}}}\right)\right.\\
&\quad\quad\quad\quad\quad\quad\quad\quad\quad\quad \left.+ \left|\frac{z-\xi}{\mu}\right|^2\left(\frac{\partial}{\partial y_2} Q\left(\frac{z-\xi}{\mu}\right) - \frac{E_{n,k}\left(\frac{z-\xi}{\mu}\right)_2}{\left(1+\left|\frac{z-\xi}{\mu}\right|^2\right)^{\frac{n}{2}}}\right)\right)
\end{aligned}
\end{equation*}
and
\begin{equation*}
\begin{aligned}
&S_5(z) = z_{n+1}\left(\frac{z-\xi}{\mu}\right)\dot{\theta}_{12} + \sum_{j = 3}^n\left(z_{n+j+1}\left(\frac{z-\xi}{\mu}\right)\dot{\theta}_{1j} + z_{2n+j-1}\left(\frac{z-\xi}{\mu}\right)\dot{\theta}_{2j}\right).
\end{aligned}
\end{equation*}
Direct computations yield that
\begin{equation*}
\begin{aligned}
\int_{B_{2R}}S_1(\xi+\mu y)z_{0}(y)dy &= (2Ac_1 +c_2)(1+O(R^{2-n})+O(R^{-2}))\dot{\lambda}\\
&\quad  + c_1(1+O(R^{-2}))\mu_0^{n-4}\left[(n-3)b^{n-4}H(q, q)\lambda\right],
\end{aligned}
\end{equation*}
\begin{equation*}
\int_{B_{2R}}S_2(\xi+\mu y)z_{0}(y)dy = O(R^{2-n}+R^{-2})\mu_{0}^{n-3},
\end{equation*}
\begin{equation*}
\int_{B_{2R}}S_3(\xi +\mu y)z_{0}(y)dy = O\left(\frac{1}{k}\right)\dot{\xi} + O(1 + R^{-2})\mu_{0}^{n-2},
\end{equation*}
\begin{equation*}
\begin{aligned}
\int_{B_{2R}}S_4(\xi +\mu y)z_{0}(y)dy &= O\left(\frac{1}{k}\right)\left(\dot{a}_1 + \dot{a}_2\right),
\end{aligned}
\end{equation*}
and
\begin{equation*}
\begin{aligned}
\int_{B_{2R}}S_5(\xi +\mu y)z_{0}(y)dy &= O\left(\frac{1}{k}\right)\left(\dot{\theta}_{12} + \sum_{j = 3}^n\left(\dot{\theta}_{1j}+ \dot{\theta}_{2j}\right)\right).
\end{aligned}
\end{equation*}
Since $\frac{\mu_{0}}{\mu} = (1 + \frac{\lambda}{\mu_{0}})^{-1}$, for $l = 1, 2, 3, 4, 5$, we have the following estimates
\begin{equation*}
\begin{aligned}
&\int_{B_{2R}}[S_l(\xi+\mu_{0}y, t)-S_l(\xi+\mu y, t)]z_{0}(y)dy \\
&\quad = g(t,\frac{\lambda}{\mu_0})\dot{\lambda} + g(t, \frac{\lambda}{\mu_0})\dot{\xi} + g(t,\frac{\lambda}{\mu_0})\left(\dot{a}_1+\dot{a}_2\right)\\
&\quad\quad + g(t,\frac{\lambda}{\mu_0})\left(\dot{\theta}_{12} + \sum_{j = 3}^n\left(\dot{\theta}_{1j} + \dot{\theta}_{2j}\right)\right)\\
&\quad\quad + g(t,\frac{\lambda}{\mu_0})\mu_0^{n-4}\left(\lambda + (\xi -q) + a_1 + a_2 + \theta_{12} + \sum_{j = 3}^n\left(\theta_{1j} + \theta_{2j}\right)\right)\\
&\quad\quad+ \mu_0^{n-3+\sigma}f(t),
\end{aligned}
\end{equation*}
where $f$ and $g$ are smooth, bounded functions satisfying $g(\cdot, s)\sim s$ as $s\to 0$. Thus
\begin{equation*}
\begin{aligned}
& c\left(\frac{\mu}{\mu_{0}}\right)^{\frac{n+2}{2}}\mu_{0}^{-1}\int_{B_{2R}}\mu_{0}^{\frac{n+2}{2}}S_{A}(\xi + \mu_{0}y, t)z_{0}(y)dy \\
&=\left[\dot{\lambda} + \frac{1+(n-4)}{(n-4)t}\lambda\right] + \left(O\left(\frac{1}{k}\right)+ t_0^{-\varepsilon}g(t,\frac{\lambda}{\mu_0})\right)\dot{\xi}\\
&\quad + \left(O\left(\frac{1}{k}\right)+ t_0^{-\varepsilon}g(t,\frac{\lambda}{\mu_0})\right)\mu\left(\dot{a}_1+\dot{a}_2\right)\\
&\quad + \left(O\left(\frac{1}{k}\right)+ t_0^{-\varepsilon}g(t,\frac{\lambda}{\mu_0})\right)\mu\left(\dot{\theta}_{12} + \sum_{j = 3}^n\left(\dot{\theta}_{1j} + \dot{\theta}_{2j}\right)\right)\\
&\quad + g(t,\frac{\lambda}{\mu_0})\mu_0^{n-4}\left(\lambda + (\xi -q) + \mu a_1 + \mu a_2 + \mu\theta_{12} + \mu\sum_{j = 3}^n\left(\theta_{1j} + \theta_{2j}\right)\right)
\end{aligned}
\end{equation*}
for smooth bounded functions $g$ satisfying $g(\cdot, s)\sim s$ as $s\to 0$.

Let us compute the term
$$p\mu_{0}^{\frac{n-2}{2}}(1 + \frac{\lambda}{\mu_{0}})^{-2}\int_{B_{2R}}|Q|^{p-1}(\frac{\mu_{0}}{\mu}y)\psi(\xi + \mu_{0}y, t)z_{0}(y)dy.$$
Its principal part is $I: = \int_{B_{2R}}|Q|^{p-1}(y)\psi(\xi + \mu_{0}y, t)z_{0}(y)dy$. From (\ref{e4:400}), we have $I = \frac{t_0^{-\varepsilon}}{R^{\alpha-2}}\mu_0^{\frac{n-2}{2}+\sigma}f(t)$ for a smooth bounded function $f$.

Furthermore, we have
\begin{equation*}
\int_{B_{2R}}B[\phi](y, t)z_{0}(y)dy = \frac{t_0^{-\varepsilon}}{R^{\alpha-2}}\Big[\mu_0^{n-3+\sigma}(t)\ell[\phi](t) + \dot{\xi}\ell[\phi](t)\Big]
\end{equation*}
and
\begin{equation*}
\int_{B_{2R}}B^0[\phi](y, t)z_0(y)dy = \frac{t_0^{-\varepsilon}}{R^{\alpha-2}}\mu_0^{n-2+\sigma}g\left(\frac{\lambda}{\mu_0}\right)[\phi](t)
\end{equation*}
for smooth bouned function $g(s)$ with $g(s)\sim s$ ($s\to 0$) and $\ell[\phi](t)$ is bounded smooth in $t$.

Combine the above estimations, we have the validity of the lemma.
\end{proof}
\subsection{The equation for $\xi$.}
Now we compute (\ref{e5:7}) for $l = 1, \cdots, n$.
\begin{lemma}\label{l5:21}
For $l = 1$, (\ref{e5:7}) is equivalent to
\begin{equation}\label{e5:181000}
\begin{aligned}
a_{1,1}\dot{\xi}_1 &+ a_{n+2, 1}\mu_0\dot{a}_1 + O\left(\frac{1}{k}\right)\dot{\lambda} + O\left(\frac{1}{k}\right)\mu_0\dot{a}_2\\
&+ O\left(\frac{1}{k}\right)\mu_0\left(\dot{\theta}_{12} + \sum_{j = 3}^n(\dot{\theta}_{1j} + \dot{\theta}_{2j})\right) = \Pi_1[\lambda,\xi, a, \theta, \dot{\lambda},\dot{\xi}, \dot{a}, \dot{\theta}, \phi, \psi](t).
\end{aligned}
\end{equation}
For $l = 2$, (\ref{e5:7}) is equivalent to
\begin{equation}\label{e5:181001}
\begin{aligned}
a_{2,2}\dot{\xi}_2 &+ a_{n+3, 2}\mu_0\dot{a}_2 + O\left(\frac{1}{k}\right)\dot{\lambda} + O\left(\frac{1}{k}\right)\mu_0\dot{a}_1\\
&+ O\left(\frac{1}{k}\right)\mu_0\left(\dot{\theta}_{12} + \sum_{j = 3}^n(\dot{\theta}_{1j} + \dot{\theta}_{2j})\right)
 = \Pi_2[\lambda,\xi, a, \theta, \dot{\lambda},\dot{\xi}, \dot{a}, \dot{\theta}, \phi, \psi](t).
\end{aligned}
\end{equation}
For $l = 3,\cdots, n$, (\ref{e5:7}) is equivalent to
\begin{equation}\label{e5:181}
\begin{aligned}
\dot{\xi}_l + O\left(\frac{1}{k}\right)\dot{\lambda} + O\left(\frac{1}{k}\right)\mu_0\left(\dot{a}_1+\dot{a}_2\right) &+ O\left(\frac{1}{k}\right)\mu_0\left(\dot{\theta}_{12} + \sum_{j = 3}^n(\dot{\theta}_{1j} + \dot{\theta}_{2j})\right)\\
& = \Pi_l[\lambda,\xi, a, \theta, \dot{\lambda},\dot{\xi}, \dot{a}, \dot{\theta}, \phi, \psi](t).
\end{aligned}
\end{equation}
For $l = 1,\cdots, n$,
\begin{equation*}
\begin{aligned}
&\Pi_{l}[\lambda,\xi,  a, \theta, \dot{\lambda},\dot{\xi}, \dot{a}, \dot{\theta}, \phi, \psi](t)\\
&\quad= \mu_0^{n-2}c_l\left[b^{n-2}\nabla H(q, q)\right]+ \mu_0^{n-2+\sigma}(t)f(t) + \frac{t_0^{-\varepsilon}}{R^{\alpha-2}}\Theta_l\\
&\quad\quad \left[\dot{\lambda},\dot{\xi}, \mu_0\dot{a}, \mu_0\dot{\theta}, \mu_0^{n-4}(t)\lambda, \mu_0^{n-4}(\xi-q), \mu_0^{n-3}a, \mu_0^{n-3}\theta, \mu_0^{n-3+\sigma}\phi, \mu_0^{\frac{n-2}{2}+\sigma}\psi\right](t),
\end{aligned}
\end{equation*}
where $c_l$ is a positive constant, $f(t)$ and $\Theta_l$ are smooth bounded for $t\in [t_0, \infty)$.
\end{lemma}
\begin{proof}
We compute
\begin{eqnarray*}
\int_{B_{2R}}H[\lambda,\xi, a, \theta, \dot{\lambda},\dot{\xi}, \dot{a}, \dot{\theta}, \phi, \psi](y,t(\tau))z_{l}(y)dy,
\end{eqnarray*}
where $H[\lambda,\xi, a, \theta, \dot{\lambda},\dot{\xi}, \dot{a}, \dot{\theta}, \phi, \psi](y,t(\tau))$ is defined in (\ref{e5:2}). Expand $\mu_{0}^{\frac{n+2}{2}}S_{A}(\xi + \mu_{0}y, t)$ as (\ref{e:601}), by direct computations, we have
\begin{equation*}
\int_{B_{2R}}S_1(\xi+\mu y)z_{l}(y)dy = O\left(\frac{1}{k}\right)\left(\dot{\lambda} + \mu_0^{n-4}\lambda\right),
\end{equation*}
\begin{equation*}
\int_{B_{2R}}S_2(\xi+\mu y)z_{l}(y)dy = O\left(\frac{1}{k}\right)\left(\dot{\mu}_0 + \mu_0^{n-3}\right),
\end{equation*}
\begin{equation*}
\begin{aligned}
\int_{B_{2R}}S_3(\xi +\mu y)z_{l}(y)dy &= (1+O(R^{-n}))a_{l, l}\dot{\xi}_l\\
&\quad -(1+O(R^{-2}))p\int_{\mathbb{R}^n}|Q|^{p-1}y_lz_l(y)dy\mu^{n-2}\nabla H(q, q),
\end{aligned}
\end{equation*}
\begin{equation*}
\begin{aligned}
\int_{B_{2R}}S_4(\xi & +\mu y)z_{l}(y)dy =\\
&\left\{
\begin{matrix}
a_{n+2, 1}(1+O(R^{4-n}))\dot{a}_1 + O\left(\frac{1}{k}\right)(1+O(R^{4-n}))\dot{a}_2 & \hbox{ if } l = 1\, , \\
&\\ a_{n+2, 2}(1+O(R^{4-n}))\dot{a}_2 + O\left(\frac{1}{k}\right)(1+O(R^{4-n}))\dot{a}_1 & \hbox{ if } l = 2\, , \\
&\\ O\left(\frac{1}{k}\right)(1+O(R^{4-n}))\left(\dot{a}_1+\dot{a}_2\right) & \hbox{ if } l = 3,\cdots, n
\end{matrix}
\right.
\end{aligned}
\end{equation*}
and
\begin{equation*}
\begin{aligned}
\int_{B_{2R}}S_5(\xi +\mu y)z_{l}(y)dy &= O\left(\frac{1}{k}\right)\left(\dot{\theta}_{12} + \sum_{j = 3}^n\left(\dot{\theta}_{1j}+ \dot{\theta}_{2j}\right)\right).
\end{aligned}
\end{equation*}
Since $\frac{\mu_{0}}{\mu} = (1 + \frac{\lambda}{\mu_{0}})^{-1}$, for $j = 1, 2, 3, 4, 5$, we have
\begin{equation*}
\begin{aligned}
&\quad~\int_{B_{2R}}[S_j(\xi+\mu_{0}y, t)-S_j(\xi+\mu y, t)]z_{l}(y)dy \\
&= g(t,\frac{\lambda}{\mu_0})\dot{\lambda} + g(t, \frac{\lambda}{\mu_0})\dot{\xi} + g(t,\frac{\lambda}{\mu_0})\left(\dot{a}_1+\dot{a}_2\right)\\
&\quad + g(t,\frac{\lambda}{\mu_0})\left(\dot{\theta}_{12} + \sum_{j = 3}^n\left(\dot{\theta}_{1j} + \dot{\theta}_{2j}\right)\right)\\
&\quad + g(t,\frac{\lambda}{\mu_0})\mu_0^{n-4}\left(\lambda + (\xi -q) + a_1 + a_2 + \theta_{12} + \sum_{j = 3}^n\left(\theta_{1j} + \theta_{2j}\right)\right)\\ &\quad + \mu_0^{n-3+\sigma}f(t),
\end{aligned}
\end{equation*}
where $f$ and $g$ are smooth, bounded functions satisfying $g(\cdot, s)\sim s$ as $s\to 0$. Thus
\begin{equation*}
\begin{aligned}
& c\left(\frac{\mu}{\mu_{0}}\right)^{\frac{n+2}{2}}\mu_{0}^{-1}\int_{B_{2R}}\mu_{0}^{\frac{n+2}{2}}S_{A}(\xi + \mu_{0}y, t)z_{l}(y)dy\\
&= \left[\dot{\xi} + \frac{-p\int_{\mathbb{R}^n}|Q|^{p-1}y_lz_l(y)dy}{\int_{\mathbb{R}^n}|z_l|^2dy}b^{n-2}\mu_0^{n-2}\right] + \left(O\left(\frac{1}{k}\right)+ t_0^{-\varepsilon}g(t,\frac{\lambda}{\mu_0})\right)\dot{\xi}\\
&\quad + \left(O\left(\frac{1}{k}\right)+ t_0^{-\varepsilon}g(t,\frac{\lambda}{\mu_0})\right)\mu_0\left(\dot{a}_1+\dot{a}_2\right)\\
&\quad + \left(O\left(\frac{1}{k}\right)+ t_0^{-\varepsilon}g(t,\frac{\lambda}{\mu_0})\right)\mu_0\left(\dot{\theta}_{12} + \sum_{j = 3}^n\left(\dot{\theta}_{1j} + \dot{\theta}_{2j}\right)\right)\\
&\quad + g(t,\frac{\lambda}{\mu_0})\mu_0^{n-4}\left(\lambda + (\xi -q) + \mu_0 a_1 + \mu_0 a_2 + \mu_0\theta_{12} + \mu_0\sum_{j = 3}^n\left(\theta_{1j} + \theta_{2j}\right)\right),
\end{aligned}
\end{equation*}
for smooth bounded functions $g$ satisfying $g(\cdot, s)\sim s$ as $s\to 0$.

The computations for the term $$p\mu_{0}^{\frac{n-2}{2}}(1 + \frac{\lambda}{\mu_{0}})^{-2}\int_{B_{2R}}|Q|^{p-1}(\frac{\mu_{0}}{\mu}y)\psi(\xi + \mu_{0}y, t)z_{l}(y)dy,$$
$B[\phi]$ and $B^0[\phi]$ are similar to that of Lemma \ref{l5:100}.
\end{proof}
\subsection{The equation for $\theta_{12}$.}
Now we compute (\ref{e5:7}) for $l = n + 1$.
\begin{lemma}\label{l5:22}
For $l = n + 1$, (\ref{e5:7}) is equivalent to
\begin{equation}\label{e5:182}
\begin{aligned}
\mu_0\dot\theta_{12} + O\left(\frac{1}{k}\right)\dot{\lambda} + O\left(\frac{1}{k}\right)\dot{\xi} & + O\left(\frac{1}{k}\right)\mu_0\left(\dot{a}_1+\dot{a}_2\right) + O\left(\frac{1}{k}\right)\mu_0\left(\sum_{j = 3}^n\dot{\theta}_{1j} + \dot{\theta}_{2j}\right)\\
& = \Pi_{n+1}[\lambda,\xi, a, \theta, \dot{\lambda},\dot{\xi}, \dot{a}, \dot{\theta}, \phi, \psi](t),
\end{aligned}
\end{equation}
\begin{equation*}
\begin{aligned}
&\Pi_{n+1}[\lambda,\xi, a, \theta, \dot{\lambda},\dot{\xi}, \dot{a}, \dot{\theta}, \phi, \psi](t)  = \mu_0^{n-2+\sigma}(t)f(t)+ \frac{t_0^{-\varepsilon}}{R^{\alpha-2}}\Theta_{n+1}\\
&\quad \left[\dot{\lambda},\dot{\xi}, \mu_0\dot{a}, \mu_0\dot{\theta}, \mu_0^{n-4}(t)\lambda, \mu_0^{n-4}(\xi-q), \mu_0^{n-3}a, \mu_0^{n-3}\theta, \mu_0^{n-3+\sigma}\phi, \mu_0^{\frac{n-2}{2}+\sigma}\psi\right](t),
\end{aligned}
\end{equation*}
where $f(t)$ and $\Theta_{n+1}$ are smooth bounded for $t\in [t_0, \infty)$.
\end{lemma}
\begin{proof}
We compute
\begin{eqnarray*}
\int_{B_{2R}}H[\lambda,\xi, a, \theta, \dot{\lambda},\dot{\xi}, \dot{a}, \dot{\theta}, \phi, \psi](y,t(\tau))z_{n+1}(y)dy,
\end{eqnarray*}
where $H[\lambda,\xi, a, \theta, \dot{\lambda},\dot{\xi}, \dot{a}, \dot{\theta}, \phi, \psi](y,t(\tau))$ is defined in (\ref{e5:2}). Expand $\mu_{0}^{\frac{n+2}{2}}S_{A}(\xi + \mu_{0}y, t)$ as (\ref{e:601}), by direct computations, we have
\begin{equation*}
\int_{B_{2R}}S_1(\xi+\mu y)z_{n+1}(y)dy = O\left(\frac{1}{k}\right)\left(\dot{\lambda} + \mu_0^{n-4}\lambda\right),
\end{equation*}
\begin{equation*}
\int_{B_{2R}}S_2(\xi+\mu y)z_{n+1}(y)dy = O\left(\frac{1}{k}\right)\left(\dot{\mu}_0 + \mu_0^{n-3}\right),
\end{equation*}
\begin{equation*}
\begin{aligned}
\int_{B_{2R}}S_3(\xi +\mu y)z_{n+1}(y)dy &= O\left(\frac{1}{k}\right)\dot{\xi} + O(1 + R^{-1})\mu_{0}^{n-2},
\end{aligned}
\end{equation*}
\begin{equation*}
\int_{B_{2R}}S_4(\xi +\mu y)z_{n+1}(y)dy = O\left(\frac{1}{k}\right)(1+O(R^{-2}))\left(\dot{a}_1+\dot{a}_2\right)
\end{equation*}
and
\begin{equation*}
\begin{aligned}
\int_{B_{2R}}S_5(\xi +\mu y)z_{n+1}(y)dy &= a_{n+1, n+1}(1 + O(R^{2-n}))\dot{\theta}_{12} + O\left(\frac{1}{k}\right)\sum_{j = 3}^n\left(\dot{\theta}_{1j}+ \dot{\theta}_{2j}\right).
\end{aligned}
\end{equation*}
Since $\frac{\mu_{0}}{\mu} = (1 + \frac{\lambda}{\mu_{0}})^{-1}$, for $j = 1, 2, 3, 4, 5$, we have
\begin{equation*}
\begin{aligned}
&\quad~\int_{B_{2R}}[S_l(\xi+\mu_{0}y, t)-S_l(\xi+\mu y, t)]z_{n+1}(y)dy \\
&\quad=g(t,\frac{\lambda}{\mu_0})\dot{\lambda} + g(t, \frac{\lambda}{\mu_0})\dot{\xi} + g(t,\frac{\lambda}{\mu_0})\left(\dot{a}_1+\dot{a}_2\right)\\
&\quad\quad + g(t,\frac{\lambda}{\mu_0})\left(\dot{\theta}_{12} + \sum_{j = 3}^n\left(\dot{\theta}_{1j} + \dot{\theta}_{2j}\right)\right)\\
&\quad\quad + g(t,\frac{\lambda}{\mu_0})\mu_0^{n-4}\left(\lambda + (\xi -q) + a_1 + a_2 + \theta_{12} + \sum_{j = 3}^n\left(\theta_{1j} + \theta_{2j}\right)\right)\\ &\quad\quad + \mu_0^{n-3+\sigma}f(t),
\end{aligned}
\end{equation*}
where $f$ and $g$ are smooth, bounded functions satisfying $g(\cdot, s)\sim s$ as $s\to 0$. Thus
\begin{equation*}
\begin{aligned}
& c\left(\frac{\mu}{\mu_{0}}\right)^{\frac{n+2}{2}}\mu_{0}^{-1}\int_{B_{2R}}\mu_{0}^{\frac{n+2}{2}}S_{A}(\xi + \mu_{0}y, t)z_{n+1}(y)dy\\
& = \mu_0\dot{\theta}_{12} + \left(O\left(\frac{1}{k}\right)+ t_0^{-\varepsilon}g(t,\frac{\lambda}{\mu_0})\right)\dot{\xi}\\
&\quad + \left(O\left(\frac{1}{k}\right)+ t_0^{-\varepsilon}g(t,\frac{\lambda}{\mu_0})\right)\mu_0\left(\dot{a}_1+\dot{a}_2\right)\\
&\quad + \left(O\left(\frac{1}{k}\right)+ t_0^{-\varepsilon}g(t,\frac{\lambda}{\mu_0})\right)\mu_0\sum_{j = 3}^n\left(\dot{\theta}_{1j} + \dot{\theta}_{2j}\right)\\
&\quad + g(t,\frac{\lambda}{\mu_0})\mu_0^{n-4}\left(\lambda + (\xi -q) + \mu_0 a_1 + \mu_0 a_2 + \mu_0\theta_{12} + \mu_0\sum_{j = 3}^n\left(\theta_{1j} + \theta_{2j}\right)\right),
\end{aligned}
\end{equation*}
for smooth bounded functions $g$ satisfying $g(\cdot, s)\sim s$ as $s\to 0$.

The computations for the term $$p\mu_{0}^{\frac{n-2}{2}}(1 + \frac{\lambda}{\mu_{0}})^{-2}\int_{B_{2R}}|Q|^{p-1}(\frac{\mu_{0}}{\mu}y)\psi(\xi + \mu_{0}y, t)z_{n+1}(y)dy,$$
$B[\phi]$ and $B^0[\phi]$ are similar to that of Lemma \ref{l5:100}.
\end{proof}
\subsection{The equation for $a_1$ and $a_2$.}
Now we compute (\ref{e5:7}) for $l = n + 2, n+3$.
\begin{lemma}\label{l5:23}
For $l = n + 2, n+3$, (\ref{e5:7}) is equivalent to
\begin{equation}\label{e5:183}
\begin{aligned}
&a_{1, n+2}\dot{\xi}_1 + a_{n+2, n+2}\mu_0\dot{a}_{1} + O\left(\frac{1}{k}\right)\dot{\lambda} + O\left(\frac{1}{k}\right)\dot{\xi}  + O\left(\frac{1}{k}\right)\mu_0\dot{a}_2 \\
&\quad + O\left(\frac{1}{k}\right)\mu_0\left(\dot{\theta}_{12} + \sum_{j = 3}^n(\dot{\theta}_{1j} + \dot{\theta}_{2j})\right)  = \Pi_{n+2}[\lambda,\xi, a, \theta, \dot{\lambda},\dot{\xi}, \dot{a}, \dot{\theta}, \phi, \psi](t),
\end{aligned}
\end{equation}
\begin{equation}\label{e5:18}
\begin{aligned}
&a_{2, n+3}\dot{\xi}_2 + a_{n+3, n+3}\mu_0\dot{a}_{2} + O\left(\frac{1}{k}\right)\dot{\lambda} + O\left(\frac{1}{k}\right)\dot{\xi}  + O\left(\frac{1}{k}\right)\mu_0\dot{a}_1 \\
&\quad + O\left(\frac{1}{k}\right)\mu_0\left(\dot{\theta}_{12} + \sum_{j = 3}^n(\dot{\theta}_{1j} + \dot{\theta}_{2j})\right)  = \Pi_{n+3}[\lambda,\xi, a, \theta, \dot{\lambda},\dot{\xi}, \dot{a}, \dot{\theta}, \phi, \psi](t),
\end{aligned}
\end{equation}
\begin{equation*}
\begin{aligned}
&\Pi_{n+2}[\lambda,\xi, a, \theta, \dot{\lambda},\dot{\xi}, \dot{a}, \dot{\theta}, \phi, \psi](t) = \mu_0^{n-2+\sigma}(t)f(t) + \frac{t_0^{-\varepsilon}}{R^{\alpha-2}}\Theta_{n+2}\\
&\quad \left[\dot{\lambda},\dot{\xi}, \mu_0\dot{a}, \mu_0\dot{\theta}, \mu_0^{n-4}(t)\lambda, \mu_0^{n-4}(\xi-q), \mu_0^{n-3}a, \mu_0^{n-3}\theta, \mu_0^{n-3+\sigma}\phi, \mu_0^{\frac{n-2}{2}+\sigma}\psi\right](t),
\end{aligned}
\end{equation*}
\begin{equation*}
\begin{aligned}
&\Pi_{n+3}[\lambda,\xi, a, \theta, \dot{\lambda},\dot{\xi}, \dot{a}, \dot{\theta}, \phi, \psi](t) = \mu_0^{n-2+\sigma}(t)f(t)+ \frac{t_0^{-\varepsilon}}{R^{\alpha-2}}\Theta_{n+3}\\
&\quad \left[\dot{\lambda},\dot{\xi}, \mu_0\dot{a}, \mu_0\dot{\theta}, \mu_0^{n-4}(t)\lambda, \mu_0^{n-4}(\xi-q), \mu_0^{n-3}a, \mu_0^{n-3}\theta, \mu_0^{n-3+\sigma}\phi, \mu_0^{\frac{n-2}{2}+\sigma}\psi\right](t),
\end{aligned}
\end{equation*}
where $f(t)$ and $\Theta_{n+2}$, $\Theta_{n+3}$ are smooth bounded functions for $t\in [t_0, \infty)$.
\end{lemma}
\begin{proof}
We compute
\begin{eqnarray*}
\int_{B_{2R}}H[\lambda,\xi, a, \theta, \dot{\lambda},\dot{\xi}, \dot{a}, \dot{\theta}, \phi, \psi](y,t(\tau))z_{n+2}(y)dy,
\end{eqnarray*}
where $H[\lambda,\xi, a, \theta, \dot{\lambda},\dot{\xi}, \dot{a}, \dot{\theta}, \phi, \psi](y,t(\tau))$ is defined in (\ref{e5:2}). Expand $\mu_{0}^{\frac{n+2}{2}}S_{A}(\xi + \mu_{0}y, t)$ as (\ref{e:601}), by direct computations, we have
\begin{equation*}
\int_{B_{2R}}S_1(\xi+\mu y)z_{n+2}(y)dy = O\left(\frac{1}{k}\right)\left(\dot{\lambda} + \mu_0^{n-4}\lambda\right),
\end{equation*}
\begin{equation*}
\int_{B_{2R}}S_2(\xi+\mu y)z_{n+2}(y)dy = O\left(\frac{1}{k}\right)\left(\dot{\mu}_0 + \mu_0^{n-3}\right),
\end{equation*}
\begin{equation*}
\begin{aligned}
\int_{B_{2R}}S_3(\xi +\mu y)z_{n+2}(y)dy &= a_{2, n+2}\dot{\xi} + O(1 + \log R)\mu_{0}^{n-2},
\end{aligned}
\end{equation*}
\begin{equation*}
\int_{B_{2R}}S_4(\xi +\mu y)z_{n+2}(y)dy = a_{n+2, n+2}\dot{a}_1+ O\left(\frac{1}{k}\right)(1+O(R^{-2}))\dot{a}_2,
\end{equation*}
\begin{equation*}
\begin{aligned}
\int_{B_{2R}}S_5(\xi +\mu y)z_{n+2}(y)dy &= O\left(\frac{1}{k}\right)\left(\dot{\theta}_{12} + \sum_{j = 3}^n\left(\dot{\theta}_{1j}+ \dot{\theta}_{2j}\right)\right).
\end{aligned}
\end{equation*}
Since $\frac{\mu_{0}}{\mu} = (1 + \frac{\lambda}{\mu_{0}})^{-1}$, for $l = 1, 2, 3, 4, 5$, we have
\begin{equation*}
\begin{aligned}
&\quad~\int_{B_{2R}}[S_l(\xi+\mu_{0}y, t)-S_l(\xi+\mu y, t)]z_{n+2}(y)dy \\
&\quad=g(t,\frac{\lambda}{\mu_0})\dot{\lambda} + g(t, \frac{\lambda}{\mu_0})\dot{\xi} + g(t,\frac{\lambda}{\mu_0})\left(\dot{a}_1+\dot{a}_2\right)\\
&\quad\quad + g(t,\frac{\lambda}{\mu_0})\left(\dot{\theta}_{12} + \sum_{j = 3}^n\left(\dot{\theta}_{1j} + \dot{\theta}_{2j}\right)\right)\\
&\quad\quad + g(t,\frac{\lambda}{\mu_0})\mu_0^{n-4}\left(\lambda + (R_\theta\xi -q) + a_1 + a_2 + \theta_{12} + \sum_{j = 3}^n\left(\theta_{1j} + \theta_{2j}\right)\right)\\ &\quad\quad + \mu_0^{n-3+\sigma}f(t),
\end{aligned}
\end{equation*}
where $f$ and $g$ are smooth, bounded functions satisfying $g(\cdot, s)\sim s$ as $s\to 0$. Thus
\begin{equation*}
\begin{aligned}
& c\left(\frac{\mu}{\mu_{0}}\right)^{\frac{n+2}{2}}\mu_{0}^{-2}\int_{B_{2R}}\mu_{0}^{\frac{n+2}{2}}S_{A}(\xi + \mu_{0}y, t)z_{n+2}(y)dy\\
& = a_{2, n+3}\dot{\xi} + a_{n+2, n+2}\mu_0\dot{a}_{1}\\
&\quad  + \left(O\left(\frac{1}{k}\right)+ t_0^{-\varepsilon}g(t,\frac{\lambda}{\mu_0})\right)\dot{\xi} + \left(O\left(\frac{1}{k}\right)+ t_0^{-\varepsilon}g(t,\frac{\lambda}{\mu_0})\right)\mu_0\dot{a}_2\\
&\quad  + \left(O\left(\frac{1}{k}\right)+ t_0^{-\varepsilon}g(t,\frac{\lambda}{\mu_0})\right)\mu_0\left(\dot{\theta}_{12} +\sum_{j = 3}^n\left(\dot{\theta}_{1j} + \dot{\theta}_{2j}\right)\right)\\
&\quad + g(t,\frac{\lambda}{\mu_0})\mu_0^{n-4}\left(\lambda + (\xi -q) + \mu_0 a_1 + \mu_0 a_2 + \mu_0\theta_{12} + \mu_0\sum_{j = 3}^n\left(\theta_{1j} + \theta_{2j}\right)\right),
\end{aligned}
\end{equation*}
for smooth bounded functions $g$ satisfying $g(\cdot, s)\sim s$ as $s\to 0$.

The computations for the term $$p\mu_{0}^{\frac{n-2}{2}}(1 + \frac{\lambda}{\mu_{0}})^{-2}\int_{B_{2R}}|Q|^{p-1}(\frac{\mu_{0}}{\mu}y)\psi(\xi + \mu_{0}y, t)z_{n+2}(y)dy,$$
$B[\phi]$ and $B^0[\phi]$ are similar to that of Lemma \ref{l5:100}. This proves (\ref{e5:183}). The proof of (\ref{e5:18}) is similar.
\end{proof}

\subsection{The equation for $\theta_{1l}$ and $\theta_{2l}$, $l = 3,\cdots, n$.}
Now we compute (\ref{e5:7}) for $l = n + 4,\cdots, 3n-1$.
\begin{lemma}\label{l5:24}
For $l = 3, \cdots, n$, (\ref{e5:7}) is equivalent to
\begin{equation}\label{e5:184}
\begin{aligned}
\mu_0\dot\theta_{1l} &+ O\left(\frac{1}{k}\right)\dot{\lambda} + O\left(\frac{1}{k}\right)\dot{\xi} + O\left(\frac{1}{k}\right)\mu_0\left(\dot{a}_1+\dot{a}_2\right)\\
& + O\left(\frac{1}{k}\right)\mu_0\left(\dot\theta_{12} + \sum_{j\neq l}\dot\theta_{1j} + \sum_{j =3}^n\dot\theta_{2j}\right) = \Pi_{n+l+1}[\lambda,\xi, a, \theta, \dot{\lambda},\dot{\xi}, \dot{a}, \dot{\theta}, \phi, \psi](t),
\end{aligned}
\end{equation}
\begin{equation}\label{e5:185}
\begin{aligned}
\mu_0\dot\theta_{2l} &+ O\left(\frac{1}{k}\right)\dot{\lambda} + O\left(\frac{1}{k}\right)\dot{\xi} + O\left(\frac{1}{k}\right)\mu_0\left(\dot{a}_1+\dot{a}_2\right)\\
& + O\left(\frac{1}{k}\right)\mu_0\left(\dot\theta_{12} + \sum_{j\neq l}\dot\theta_{2j} + \sum_{j =3}^n\dot\theta_{1j}\right) = \Pi_{2n+l-1}[\lambda,\xi, a, \theta, \dot{\lambda},\dot{\xi}, \dot{a}, \dot{\theta}, \phi, \psi](t),
\end{aligned}
\end{equation}
\begin{equation*}
\begin{aligned}
&\Pi_{n+l+1}[\lambda,\xi, a, \theta, \dot{\lambda},\dot{\xi}, \dot{a}, \dot{\theta}, \phi, \psi](t)  = \mu_0^{n-2+\sigma}(t)f(t)+ \frac{t_0^{-\varepsilon}}{R^{\alpha-2}}\Theta_{n+l+1}\\
&\quad \left[\dot{\lambda},\dot{\xi}, \mu_0\dot{a}, \mu_0\dot{\theta}, \mu_0^{n-4}(t)\lambda, \mu_0^{n-4}(\xi-q), \mu_0^{n-3}a, \mu_0^{n-3}\theta, \mu_0^{n-3+\sigma}\phi, \mu_0^{\frac{n-2}{2}+\sigma}\psi\right](t),
\end{aligned}
\end{equation*}
\begin{equation*}
\begin{aligned}
&\Pi_{2n+l-1}[\lambda,\xi, a, \theta, \dot{\lambda},\dot{\xi}, \dot{a}, \dot{\theta}, \phi, \psi](t)  = \mu_0^{n-2+\sigma}(t)f(t)+ \frac{t_0^{-\varepsilon}}{R^{\alpha-2}}\Theta_{2n+l-1}\\
&\quad \left[\dot{\lambda},\dot{\xi}, \mu_0\dot{a}, \mu_0\dot{\theta}, \mu_0^{n-4}(t)\lambda, \mu_0^{n-4}(\xi-q), \mu_0^{n-3}a, \mu_0^{n-3}\theta, \mu_0^{n-3+\sigma}\phi, \mu_0^{\frac{n-2}{2}+\sigma}\psi\right](t),
\end{aligned}
\end{equation*}
where $f(t)$ and $\Theta_{n+l+1}$, $\Theta_{2n+l-1}$ are smooth bounded for $t\in [t_0, \infty)$.
\end{lemma}

The proof is similar to Lemma \ref{l5:22}. Since the matrices \begin{equation*}
\begin{pmatrix}
a_{1,1} & a_{1, n+2}\\
a_{n+2, 1} & a_{n+2, n+2}
\end{pmatrix}, \quad
\begin{pmatrix}
a_{2,2} & a_{2, n+3}\\
a_{n+3, 2} & a_{n+3, n+3}
\end{pmatrix}
\end{equation*}
are invertible, equations (\ref{e5:181000}), (\ref{e5:181001}), (\ref{e5:183}) and (\ref{e5:18}) can be decoupled by inverting the coefficient matrices. Combine Lemmas \ref{l5:21}, \ref{l5:22}, \ref{l5:23}, \ref{l5:24} and \ref{e5:900}, we get the result of Proposition \ref{l5:1}.

\section{Proof of Proposition \ref{propositionestimate}}
{\it Proof of (\ref{e4:37})}. Let us recall from \eqref{e3:8} that
\begin{equation*}\label{e4:41}
S_{out} = S^{(2)}_{A} + (1 - \eta_{R})S_{A}.
\end{equation*}
From (\ref{awayq}) and Lemma \ref{l2.2}, in the region $|x-q|>\delta$ with $\delta > 0$, we have the following estimate for $S_{out}$,
\begin{equation}\label{e4:42}
|S_{out}(x, t)|\lesssim \mu_0^{\frac{n-2}{2}}(\mu_0^{2}+\mu_0^{n-4})\lesssim \mu_0^{\min(n-4,2)-(\alpha-2)-\sigma}(t_0)\frac{\mu^{-2}\mu_0^{\frac{n-2}{2}+\sigma}}{1+|y|^{\alpha}}.
\end{equation}
In the region $|x-q| \leq \delta$ with $\delta > 0$ sufficiently small, Lemma \ref{l2.2} tells us that
\begin{equation}\label{e4:43}
\left|S^{(2)}_{A}(x, t)\right|\lesssim \mu_0^{-\frac{n+2}{2}}\frac{\mu_0^{n}}{1+|y|^{2}}\lesssim  \mu_0^{2-(\alpha-2)-\sigma}(t_0)\frac{\mu^{-2}\mu_0^{\frac{n-2}{2}+\sigma}}{1+|y|^{\alpha}}.
\end{equation}
By the definition of $\eta_{R}$, if $|x-\xi|>\mu_0R$, $(1-\eta_{R})\neq 0$. Therefore we have
\begin{equation}\label{e4:44}
\left|(1 - \eta_{R})S_{A}\right|\lesssim \left(\frac{1}{R^{n-2-\alpha}}+\frac{1}{R^{4-\alpha}}\right)\frac{1}{R^{a-2}}\frac{\mu^{-2}\mu_0^{\frac{n-2}{2}+\sigma}}{1+|y|^{\alpha}}.
\end{equation}
Here the decaying assumptions \eqref{e4:21} and \eqref{e4:22} are used, respectively.  This proves the validity of (\ref{e4:37}).

{\it Proof of (\ref{e4:38})}. For the term $2\nabla\eta_{R}\nabla\tilde{\phi}$, recalling that
$$\tilde{\phi}(x, t): = \mu_{0}^{-\frac{n-2}{2}}\phi\left(\frac{x-\xi}{\mu_{0}}, t\right)$$
and the assumptions \eqref{e4:24} and \eqref{e4:25}, we have
\begin{equation}\label{4.46}
\begin{aligned}
&\left|\left(\nabla\eta_{R}\cdot\nabla\tilde\phi\right)(x,t)\right|\\
&\quad\lesssim\frac{\eta'(\left|\frac{x-\xi}{R\mu_{0}}\right|)}{R\mu_{0}}\mu_0^{-\frac{n-2}{2}}\frac{|\nabla_y \phi|}{\mu_0}\\
&\quad\lesssim\frac{\eta'(\left|\frac{x-\xi}{R\mu_{0}}\right|)}{R\mu_{0}^2}\frac{\mu_0^{{n-2\over 2}+\sigma}}{(1+ |y|^{1+ \alpha})}\|\phi\|_{n-2+\sigma, \alpha}\\
&\quad\lesssim\frac{1}{R^{a-2}}\|\phi\|_{n-2+\sigma, \alpha}\frac{\mu^{-2}\mu_0^{\frac{n-2}{2}+\sigma }}{(1+|y|^{\alpha})},
\end{aligned}
\end{equation}
where, in the region $\eta'(\left|{x-\xi\over R\mu_{0}}\right|) \not=0$, $(1+ |y| ) \sim R$, $y= {x-\xi\over \mu_{0}}$.
As for the second term $\tilde{\phi}\big(\Delta -\partial_t\big)\eta_{R}$, by direct computations, we have
\begin{equation}\label{e4:46}
\begin{aligned}
\left|\tilde{\phi}\big(\Delta -\partial_t\big)\eta_{R}\right|\lesssim & \frac{\left|\Delta\eta\left(\left|\frac{x-\xi}{R\mu_{0}}\right|\right)\right|}{R^{2}\mu_{0}^{2}}\mu_0^{-\frac{n-2}{2}}|\phi|\\
&+\left|\eta'\left(\left|\frac{x-\xi}{R\mu_{0}}\right|\right)\left(\frac{|x-\xi|}{R\mu_0^2}\dot{\mu_0}+\frac{1}{R\mu_0}\dot{\xi}\right)\right|
\mu_0^{-\frac{n-2}{2}}|\phi|.
\end{aligned}
\end{equation}
From the definition of $\tilde{\phi}$, we have the following estimate for the first term in the right hand side of (\ref{e4:46}),
\begin{equation}\label{e4:47}
\begin{aligned}
\frac{\left|\Delta\eta\left(\left|\frac{x-\xi}{R\mu_{0}}\right|\right)\right|}{R^{2}\mu_{0}^{2}}\mu_0^{-\frac{n-2}{2}}|\phi|& \lesssim  \frac{\left|\Delta\eta\left(\left|\frac{x-\xi}{R\mu_{0}}\right|\right)\right|}{R^{2}\mu_{0}^{2}}\frac{\mu_0^{\frac{n-2}{2}+\sigma}}{(1+|y|^\alpha)}
\|\phi\|_{n-2s+\sigma,\alpha}\\
&\lesssim \frac{1}{R^{a-2}}\|\phi\|_{n-2+\sigma,\alpha}\frac{\mu^{-2}\mu_0^{\frac{n-2}{2}+\sigma}(t)}{1+|y|^{\alpha}},
\end{aligned}
\end{equation}
here the fact that $\left|\Delta\eta\left(\left|\frac{x-\xi}{R\mu_{0}}\right|\right)\right|\sim \frac{1}{1+|\frac{y}{R}|^{2}}$ was used. From \eqref{e4:21}, we estimate the second term in the right hand side of (\ref{e4:46}) as
\begin{equation}\label{e4:48}
\begin{aligned}
&\left|\eta'\left(\left|\frac{x-\xi}{R\mu_{0}}\right|\right)\left(\frac{|x-\xi|\dot{\mu_0}+\mu_0\dot{\xi}}{R\mu_0^2}\right)\right|\mu_0^{-\frac{n-2}{2}}|\phi|\\
&\quad\quad\quad\quad\quad\lesssim \frac{\left|\eta'\left(\left|\frac{x-\xi}{R\mu_{0}}\right|\right)\right|}{R^{2}\mu_{0}^{2}}(\mu_0^{n-2}R^{2} +\mu_0^{n-2+\sigma}R)\mu_0^{-\frac{n-2}{2}}|\phi|\\
&\quad\quad\quad\quad\quad\lesssim \frac{1}{R^{\alpha-2}}\|\phi\|_{n-2+\sigma,\alpha}\frac{\mu^{-2}\mu_0^{\frac{n-2}{2}+\sigma}(t)}{1+|y|^{\alpha}}.
\end{aligned}
\end{equation}
From \eqref{4.46}-\eqref{e4:48}, we obtain (\ref{e4:38}).

{\it Proof of (\ref{e4:39})}. Since $p-2\geq 0$ when $n\leq 6$, we have the following
\begin{equation}\label{e4:49}
\begin{aligned}
&\tilde{N}_{A}(\psi + \psi_1 + \eta_{R}\tilde{\phi})\lesssim\\
&\quad\quad\quad\quad\quad\left\{
\begin{aligned}
&|u^*_{A}|^{p-2}\left[|\psi|^2 + |\psi_1|^2 + |\eta_{R}\tilde{\phi}|^2\right], & \quad \mbox{when}~ 6\geq n,\\
&|\psi|^p + |\psi_1|^p + |\eta_{R}\tilde{\phi}|^p, & \quad \mbox{when}~ 6< n.
\end{aligned}
\right.
\end{aligned}
\end{equation}
When $6\geq n$, there hold
\begin{equation*}\label{e4:51}
\begin{aligned}
\left|(u^*_{A})^{p-2}(\eta_{R}\tilde{\phi})^2\right|&\lesssim \frac{\mu_0^{\frac{3n}{2}-5+2\sigma}}{1+|y|^{2\alpha}}\|\phi\|^2_{n-2+\sigma, \alpha}\\
&\lesssim \mu_0^{n-2+\sigma}R^{\alpha-2}\|\phi\|^2_{n-2+\sigma, \alpha}\frac{1}{R^{\alpha-2}}\frac{\mu^{-2}\mu_0^{\frac{n-2}{2}+\sigma}}{1+|y|^{\alpha}}
\end{aligned}
\end{equation*}
and
\begin{equation*}\label{e4:52}
\begin{aligned}
\left|(u^*_{A})^{p-2}\psi^2\right|&\lesssim\mu_0^{-\frac{6-n}{2}}\frac{t^{-2\beta}}{1+|y|^{2(\alpha-2)}}\|\psi\|^2_{**,\beta, \alpha}\\
&\lesssim R^{\alpha-2}\mu_0^{n-4+\sigma + \alpha-2}\|\psi\|^2_{**,\beta,\alpha}\frac{1}{R^{\alpha-2}}\frac{\mu_j^{-2}t^{-\beta}}{1+|y|^{\alpha}}.
\end{aligned}
\end{equation*}
When $6 < n$, one has
\begin{equation*}\label{e4:510}
\begin{aligned}
\left|\eta_{R}\tilde{\phi}\right|^p&\lesssim \frac{\mu_0^{(\frac{n-2}{2}+\sigma)p}}{1+|y|^{\alpha p}}\|\phi\|^p_{n-2+\sigma, \alpha}\\
&\lesssim \mu_0^{2+(p-1)\sigma}R^{\alpha-2}\mu_0^{2}\|\phi\|^p_{n-2+\sigma, \alpha}\frac{1}{R^{\alpha-2}}\frac{\mu^{-2}\mu_0^{\frac{n-2}{2}+\sigma}}{1+|y|^{\alpha}}
\end{aligned}
\end{equation*}
and
\begin{equation*}\label{e4:520}
\begin{aligned}
\left|\psi\right|^p&\lesssim\frac{t^{-p\beta}}{1+|y|^{p(\alpha-2)}}\|\psi\|^p_{**,\beta, \alpha}\\
&\lesssim \mu^{4(1+\frac{\sigma}{n-2})+p(\alpha-2)-\alpha} R^{\alpha-2}\|\psi\|^p_{**,\beta,\alpha}\frac{1}{R^{\alpha-2}}\frac{\mu_j^{-2}\mu_0^{\frac{n-2}{2}+\sigma}}{1+|y|^{\alpha}}.
\end{aligned}
\end{equation*}
The estimate for $\psi_1$ is similar. This proves (\ref{e4:39}).\qed

\section{Stability result in dimension 5 and 6.}
In dimension 5 and 6, we have $p-1 = \frac{4}{n-2}\geq 1$. In this case, all the equations can be solved by the Contraction Mapping Theorem since the operators $\mathcal{T}_0$, $\mathcal{T}_1$ and $\mathcal{T}_2$ are Lipschitz continuous with respect to the parameter functions. Therefore, Theorem \ref{t:main} can be proved by the Contraction Mapping Theorem arguments in dimension 5 and 6, moreover, we have the following stability result.
\begin{theorem}\label{t:stability}
Assume $k_0$ is a sufficiently large integer, $n = 5, 6$ and $q$ is a point in $\Omega$, then the conclusion of Theorem \ref{t:main} holds when $k\geq k_0$. Furthermore, there exists a sub-manifold $\mathcal{M}$ with codimension $K$ in $C^1(\overline{\Omega})$ containing $u_q(x,0)$ such that, if $u_0\in \mathcal{M}$ and is sufficiently close to $u_q(x,0)$, the solution $u(x, t)$ to (\ref{e:main}) still has the form
\begin{equation*}
u(x,t) =\tilde{\lambda}(t)^{-\frac{n-2}{2}}\left(Q_k\left(\frac{x-\tilde{\xi}(t)}{\tilde{\lambda}(t)}\right)+\tilde{\varphi}(x, t)\right),
\end{equation*}
where $\tilde{q} = \lim_{t\to +\infty}\tilde{\xi}(t)$ is close to $q$.
\end{theorem}
Recalling that $K$ is the dimension of the space $V :=\{f\in \dot{H}^1(\mathbb{R}^n)| \langle L f, f\rangle < 0\}$ and $L$ is defined in (\ref{defL}). The proof is similar to \cite{cortazar2016green} and \cite{MussoSireWeiZhengZhou}, so we give a sketch here. We divide the whole process into three steps.

{\bf Step 1. Solving the outer problem (\ref{e4:main}).}
\begin{prop}\label{p4:4.1}
Assume $\lambda$, $\xi$, $a$, $\theta$, $\dot{\lambda}$, $\dot{\xi}$, $\dot{a}$ and $\dot{\theta}$ satisfy (\ref{e4:21}) and (\ref{e4:22}), $\phi$ satisfies (\ref{e4:25}), $\psi_0\in C^2(\overline{\Omega})$ and
\begin{equation*}\label{e4:26}
\|\psi_0\|_{L^\infty(\overline{\Omega})} + \|\nabla\psi_0\|_{L^\infty(\overline{\Omega})}\leq \frac{t_0^{-\varepsilon}}{R^{\alpha-2}}.
\end{equation*}
Then (\ref{e4:main}) has a unique solution $\psi = \Psi[\lambda, \xi, a, \theta, \dot{\lambda}, \dot{\xi}, \dot{a}, \dot{\theta}, \phi]$, for $y=\frac{x-\xi}{\mu_{0}}$, there exist small constants $\sigma > 0$ and $\varepsilon > 0$ such that
\begin{equation*}\label{e4:27}
|\psi(x, t)|\lesssim \frac{t_0^{-\varepsilon}}{R^{\alpha-2}}\frac{\mu_0^{\frac{n-2}{2}+\sigma}(t)}{1+|y|^{\alpha-2}} + e^{-\delta(t-t_0)}\|\psi_0\|_{L^{\infty}(\overline{\Omega})}
\end{equation*}
and
\begin{equation*}\label{e4:28}
|\nabla\psi(x, t)|\lesssim \frac{t_0^{-\varepsilon}}{R^{\alpha-2}}\frac{\mu^{-1}\mu_0^{\frac{n-2}{2}+\sigma}(t)}{1+|y|^{\alpha-1}}\text{ for } |y|\leq R
\end{equation*}
hold. Here $R$ is defined in \eqref{e3:20}.
\end{prop}

Proposition \ref{p4:4.1} is a direct consequence of Proposition \ref{l4:lemma4.1}, Proposition \ref{propositionestimate} and the Contraction Mapping Theorem, whose proof we omit here.  This result indicates that for any small initial datum $\psi_0$, (\ref{e4:main}) has a solution $\psi$.  Moreover, the following proposition clarifies the dependence of $\Psi[\lambda, \xi, a, \theta, \dot{\lambda}, \dot{\xi}, \dot{a}, \dot{\theta}, \phi]$ on the parameter functions $\lambda, \xi, a, \theta, \dot{\lambda}, \dot{\xi}, \dot{a}, \dot{\theta}, \phi$ which is proved by estimating, for instance,
\begin{equation*}\label{e4:31}
\partial_\phi\Psi[\lambda, \xi, a, \theta, \dot{\lambda}, \dot{\xi}, \dot{a}, \dot{\theta}, \phi][\bar{\phi}] = \partial_s\Psi[\lambda, \xi, a, \theta, \dot{\lambda}, \dot{\xi}, \dot{a}, \dot{\theta}, \phi+s\bar{\phi}]|_{s=0}
\end{equation*}
as a bounded linear operator between weighted parameter spaces. For simplicity, the above operator is denoted by $\partial_\phi\Psi[\bar{\phi}]$. Similarly, we define $\partial_\lambda\Psi[\bar{\lambda}]$, $\partial_\xi\Psi[\bar{\xi}]$, $\partial_a\Psi[\bar{a}]$, $\partial_\theta\Psi[\bar{\theta}]$, $\partial_{\dot{\lambda}}\Psi[\dot{\bar{\lambda}}]$, $\partial_{\dot{\xi}}\Psi[\dot{\bar{\xi}}]$, $\partial_{\dot{a}}\Psi[\dot{\bar{a}}]$ and $\partial_{\dot{\theta}}\Psi[\dot{\bar{\theta}}]$.

\begin{prop}\label{p4:4.2}
Under the assumptions of Proposition \ref{p4:4.1}, $\Psi$ depends smoothly on the parameter functions $\lambda$, $\xi$, $a$, $\theta$, $\dot{\lambda}$, $\dot{\xi}$, $\dot{a}$, $\dot{\theta}$, $\phi$, for $y=\frac{x-\xi}{\mu_{0}}$, there hold
\begin{equation}\label{e4:64}
\begin{aligned}
\left|\partial_\lambda\Psi[\bar{\lambda}]\right|&\lesssim \frac{t_0^{-\varepsilon}}{R^{\alpha-2}}\|\bar{\lambda}\|_{1+\sigma}\frac{\mu_0^{\frac{n-2}{2}-1}(t)}{1+|y|^{\alpha-2}},
\end{aligned}
\end{equation}
\begin{equation*}\label{e4:74}
\begin{aligned}
\left|\partial_\xi\Psi[\bar{\xi}]\right|\lesssim\frac{t_0^{-\varepsilon}}{R^{\alpha-2}}\left(\|\bar{\xi}\|_{1+\sigma}
\frac{\mu_0^{\frac{n-2}{2}-1}(t)}{1+|y|^{\alpha-2}}\right),
\end{aligned}
\end{equation*}
\begin{equation*}\label{e4:74a}
\begin{aligned}
\left|\partial_a\Psi[\bar{a}]\right|\lesssim\frac{t_0^{-\varepsilon}}{R^{\alpha-2}}\left(\|\bar{a}\|_{\sigma}\frac{\mu_0^{\frac{n-2}{2}-2}(t)}{1+|y|^{\alpha-2}}\right),
\end{aligned}
\end{equation*}
\begin{equation*}\label{e4:74b}
\begin{aligned}
\left|\partial_\theta\Psi[\bar{\theta}]\right|\lesssim\frac{t_0^{-\varepsilon}}{R^{\alpha-2}}\left(\|\bar{\theta}\|_{\sigma}
\frac{\mu_0^{\frac{n-2}{2}-2}(t)}{1+|y|^{\alpha-2}}\right),
\end{aligned}
\end{equation*}
\begin{equation*}\label{e4:75}
\big|\partial_{\dot{\xi}}\Psi[\dot{\bar{\xi}}](x, t)\big|\lesssim \frac{t_0^{-\varepsilon}}{R^{\alpha-2}}\|\dot{\bar{\xi}}(t)\|_{n-3+\sigma}\left(\frac{\mu_0^{-\frac{n-6}{2}-1+\sigma}(t)}{1+|y|^{\alpha-2}}\right),
\end{equation*}
\begin{equation*}\label{e4:76}
\big|\partial_{\dot{\lambda}}\Psi[\dot{\bar{\lambda}}](x, t)\big|\lesssim \frac{t_0^{-\varepsilon}}{R^{\alpha-2}}\|\dot{\bar{\lambda}}(t)\|_{n-3+\sigma}\left(\frac{\mu_0^{-\frac{n-6}{2}-1+\sigma}(t)}{1+|y|^{\alpha-2}}\right),
\end{equation*}
\begin{equation*}\label{e4:76a}
\big|\partial_{\dot{a}}\Psi[\dot{\bar{a}}](x, t)\big|\lesssim \frac{t_0^{-\varepsilon}}{R^{\alpha-2}}\|\dot{\bar{a}}(t)\|_{n-4+\sigma}\left(\frac{\mu_0^{-\frac{n-6}{2}-2+\sigma}(t)}{1+|y|^{\alpha-2}}\right),
\end{equation*}
\begin{equation*}\label{e4:76b}
\big|\partial_{\dot{\theta}}\Psi[\dot{\bar{\theta}}](x, t)\big|\lesssim \frac{t_0^{-\varepsilon}}{R^{\alpha-2}}\|\dot{\bar{\theta}}(t)\|_{n-4+\sigma}\left(\frac{\mu_0^{-\frac{n-6}{2}-2+\sigma}(t)}{1+|y|^{\alpha-2}}\right)
\end{equation*}
and
\begin{equation*}\label{e4:84}
\big|\partial_{\phi}\Psi[\bar{\phi}](x, t)\big|\lesssim \frac{1}{R^{\alpha-2}}\|\bar{\phi}(t)\|_{n-2+\sigma, \alpha}\left(\frac{\mu_0^{\frac{n-2}{2}+\sigma}(t)}{1+|y|^{\alpha-2}}\right).
\end{equation*}
\end{prop}
\begin{proof}
We prove (\ref{e4:64}). Decompose the term $\partial_\lambda\Psi[\bar{\lambda}](x, t) = Z_1 + Z$ with $Z_1 = \mathcal{T}_2(0, -\partial_{\lambda}u^*_{A}[\bar{\lambda}],0)$, where $\mathcal{T}_2$ is defined by Proposition \ref{l4:lemma4.1}. Then $Z$ is a solution of the following problem
\begin{equation}\label{e4:72}
\left\{
\begin{aligned}
&\partial_tZ =\Delta Z + V_{A}Z + \partial_{\lambda}V_A[\bar{\lambda}]\psi + \partial_{\lambda}\tilde{N}_{A}\left(\psi + \phi^{in}\right)[\bar{\lambda}] + \partial_{\lambda}S_{out}[\bar{\lambda}] ~~\text{ in }\Omega\times (t_0, \infty),\\
&Z= 0~~\text{ in }\partial\Omega\times (t_0,\infty),\\
&Z(\cdot,t_0) = 0~~\text{ in }\Omega.
\end{aligned}
\right.
\end{equation}
For any $x\in \partial\Omega$,
\begin{equation}\label{bdr}
\begin{aligned}
\left|\partial_{\lambda}u^*_{A}[\bar{\lambda}](x, t)\right|&\lesssim \mu_0^{\frac{n}{2}-1+\sigma}(t)|\bar{\lambda}(t)|\\
&\lesssim \mu_0^{\frac{n}{2}+2\sigma}(t)\|\bar{\lambda}(t)\|_{1+\sigma}.
\end{aligned}
\end{equation}
From \eqref{bdr} and Proposition \ref{l4:lemma4.1}, we obtain
\begin{equation*}
|Z_1(x, t)|\lesssim \frac{t_0^{-\varepsilon}}{R^{\alpha-2}}\left(\|\bar{\lambda}(t)\|_{1+\sigma}\frac{\mu_0^{\frac{n-2}{2}-1}(t)}{1+|y|^{\alpha-2}}\right).
\end{equation*}

To prove the estimation for $Z$, which can be viewed as a fixed point for the operator
\begin{equation}\label{fixedpointproblem}
\mathcal{A}(Z) = \mathcal{T}_2\left(g, 0, 0\right)
\end{equation}
with
\begin{equation*}
\begin{aligned}
g= \partial_{\lambda}V_A[\bar{\lambda}]\psi + \partial_{\lambda}\tilde{N}_{A}\left(\psi + \phi^{in}\right)[\bar{\lambda}] + \partial_{\lambda}S_{out}[\bar{\lambda}],
\end{aligned}
\end{equation*}
we estimate $\partial_{\lambda}S_{out}[\bar{\lambda}]$ first. In the region $|x-q| > \delta$, from (\ref{awayq}), (\ref{e4:21}) and (\ref{e4:22}), we have
\begin{equation*}
\begin{aligned}
\left|\partial_{\lambda}S_{out}[\bar{\lambda}](x, t)\right| &\lesssim \mu_0^{\frac{n-2}{2}-1}f(x,\mu_0^{-1}\mu, \xi, a, \theta)|\bar{\lambda}(t)|\\
&\lesssim \frac{t_0^{-\varepsilon}}{R^{\alpha-2}}\left(\|\bar{\lambda}(t)\|_{1+\sigma}\frac{\mu_0^{\frac{n-2}{2}-1}(t)}{1+|y|^{\alpha-2}}\right),
\end{aligned}
\end{equation*}
where the function $f$ is smooth and bounded depending on $(x,\mu_0^{-1}\mu, \xi, a, \theta)$. In the region $|x-q|\leq \delta$, from (\ref{e2:55}), we have
\begin{equation*}
\partial_{\lambda}S(u^*_{A})[\bar{\lambda}](x, t) = \partial_{\lambda}S(u_{A})[\bar{\lambda}](x, t)(1 + \mu_0f(x,\mu_0^{-1}\mu, \xi, a, \theta)),
\end{equation*}
where the function $f$ is smooth and bounded depending on $(x,\mu_0^{-1}\mu, \xi, a, \theta)$. Differentiating  (\ref{e2:3333}) with respect to $\lambda$, easy but long computations yield that
\begin{equation}\label{e4:66}
\begin{aligned}
\left|\partial_{\lambda}S(u_{A})[\bar{\lambda}]\right|\lesssim\frac{t_0^{-\varepsilon}}{R^{\alpha-2}}\left(\|\bar{\lambda}(t)\|_{1+\sigma}
\frac{\mu_0^{\frac{n-2}{2}-1}(t)}{1+|y|^{\alpha-2}}\right).
\end{aligned}
\end{equation}
By the definition of $S_{out}$ together with \eqref{e4:66}, we obtain
\begin{eqnarray*}\label{e4:67}
\left|\partial_{\lambda}S_{out}[\bar{\lambda}](x, t)\right| \lesssim\frac{t_0^{-\varepsilon}}{R^{\alpha-2}}\left(\|\bar{\lambda}(t)\|_{1+\sigma}\frac{\mu_0^{\frac{n-2}{2}-1}(t)}{1+|y|^{\alpha-2}}\right).
\end{eqnarray*}
Now we estimate the other terms of $g$. When $n = 5, 6$, we have
\begin{equation*}\label{e4:68}
\begin{aligned}
\partial_{\lambda}V_A[\bar{\lambda}](x, t) =& p(p-1)\bigg[|u^*_{A}|^{p-3}u^*_{A}\partial_{\lambda}u^*_{A}[\bar{\lambda}]\\
&-\eta_{R}\left|\mu^{-\frac{n-2}{2}}Q(y)\right|^{p-3}\mu^{-\frac{n-2}{2}}Q(y)\partial_{\lambda}\big(\mu^{-\frac{n-2}{2}}Q(y)\big)[\bar{\lambda}]\bigg].
\end{aligned}
\end{equation*}
Since $\left|\partial_{\lambda}\big(\mu^{-\frac{n-2}{2}}Q(y)\big)\right|\lesssim\mu_0^{-1}\left|\mu^{-\frac{n-2}{2}}Q\left(y\right)\right|$ and $\beta = \frac{n-2}{2(n-4)} + \frac{\sigma}{n-4}$, we obtain
\begin{eqnarray*}\label{e4:69}
\begin{aligned}
\left|\partial_{\lambda}V_A[\bar{\lambda}]\psi(x, t)\right| \lesssim \|\psi\|_{**,\beta,\alpha} \frac{t_0^{-\varepsilon}}{R^{\alpha-2}}\|\bar{\lambda}(t)\|_{1+\sigma}\frac{\mu^{-2}\mu_0^{\frac{n-2}{2}-1+\sigma}}{1+|y|^{\alpha}}.
\end{aligned}
\end{eqnarray*}
Similarly, we estimate the term $p(p-1)|u^*_{A}|^{p-3}u^*_{A}(\psi+\phi^{in})\partial_{\lambda}u^*_{A}[\bar{\lambda}]$ as
\begin{equation*}\label{e4:70}
 \left|p(p-1)|u^*_{A}|^{p-3}u^*_{A}(\psi+\phi^{in})\partial_{\lambda}u^*_{A}[\bar{\lambda}]\right| \lesssim \frac{t_0^{-\varepsilon}}{R^{\alpha-2}}\|\bar{\lambda}\|_{1+\sigma}\frac{\mu^{-2}\mu_0^{\frac{n-2}{2}-1+\sigma}}{1+|y|^{\alpha}}
\end{equation*}
when $n = 5, 6$. The last term $p\left[\left|u^*_{A}+\psi+\phi^{in}\right|^{p-1}u^*_{A} -\left|u^*_{A}\right|^{p-1}u^*_{A}\right]$ can be estimated analogously.

In the set of functions satisfying
$$|Z(x, t)|\leq M \frac{t_0^{-\varepsilon}}{R^{\alpha-2}}\|\bar{\lambda}\|_{1+\sigma}\frac{\mu_0^{\frac{n-2}{2}-1}}{1+|y|^{\alpha-2}}$$
for a fixed large constant $M$, the operator $\mathcal{A}$ defined in (\ref{fixedpointproblem}) has a fixed point. Indeed, $\mathcal{A}$ is a contraction map when $R$ is large in terms of $t_0$. Hence (\ref{e4:64}) holds.
The proof of the other estimates are similar, we omit them.
\end{proof}

Substituting the solution $\psi = \Psi[\lambda, \xi, a, \theta, \dot{\lambda}, \dot{\xi}, \dot{a}, \dot{\theta}, \phi]$ of (\ref{e4:main}) given by Proposition \ref{p4:4.1} into (\ref{e3:10}), the full problem becomes
\begin{eqnarray}\label{e5:1000}
&& \mu_{0}^{2}\partial_t\phi = \Delta_y\phi + p|Q|^{p-1}(y)\phi + H[\lambda,\xi, a, \theta, \dot{\lambda},\dot{\xi}, \dot{a}, \dot{\theta}, \phi](y,t), y\in B_{2R}(0).
\end{eqnarray}
Similar to Section 4.1, using change of variables
\begin{equation*}\label{e5:30}
t = t(\tau),\quad \frac{dt}{d\tau} = \mu_{0}^{2}(t),
\end{equation*}
(\ref{e5:1000}) reduces to
\begin{eqnarray*}\label{e5:40}
\partial_\tau\phi = \Delta_y\phi + p|Q|^{p-1}(y)\phi + H[\lambda,\xi, a, \theta, \dot{\lambda},\dot{\xi}, \dot{a}, \dot{\theta}, \phi](y,t(\tau))
\end{eqnarray*}
for $y\in B_{2R}(0)$, $\tau\geq \tau_0$, $\tau_0$ is the unique positive number such that $t(\tau_0) = t_0$.
We try to find a solution $\phi$ to the equation
\begin{equation}\label{e5:60}
\left\{
\begin{aligned}
&\partial_\tau\phi = \Delta_y\phi + p|Q|^{p-1}(y)\phi\\
&\quad\quad\quad\quad\quad + H[\lambda,\xi, a, \theta, \dot{\lambda},\dot{\xi}, \dot{a}, \dot{\theta}, \phi](y,t(\tau)),\quad y\in B_{2R}(0),\quad\tau\geq\tau_0,\\
&\phi(y,\tau_0) = \sum_{l= 1}^Ke_{0l}Z_l(y),\quad y\in B_{2R}(0),
\end{aligned}
\right.
\end{equation}
for some suitable constants $e_{0l}$, $l = 1, \cdots, K$. To apply the linear theory Proposition \ref{proposition5.1}, the parameter functions $\lambda,\xi, a, \theta$ need to satisfy the following orthogonality conditions
\begin{equation}\label{e5:70}
\int_{B_{2R}}H[\lambda,\xi, a, \theta, \dot{\lambda},\dot{\xi}, \dot{a}, \dot{\theta}, \phi](y,t(\tau))z_l(y)dy = 0, \quad l = 0, 1, \cdots, 3n-1.
\end{equation}

{\bf Step 2. Choosing the parameter functions.} By the Lipschitz properties for $\Psi = \Psi[\lambda, \xi, a, \theta, \dot{\lambda}, \dot{\xi}, \dot{a}, \dot{\theta}, \phi]$ given by Proposition \ref{p4:4.2}, Proposition \ref{l5:1} can be strengthened as
\begin{prop}\label{l5:1000}
(\ref{e5:70}) is equivalent to
\begin{equation}\label{e5:9000}
\left\{
\begin{aligned}
&\dot{\lambda} + \frac{1+(n-4)}{(n-4)t}\lambda = \Pi_0[\lambda,\xi, a, \theta, \dot{\lambda},\dot{\xi}, \dot{a}, \dot{\theta}, \phi, \psi](t),\\
&\dot{\xi}_l = \Pi_l[\lambda,\xi, a, \theta, \dot{\lambda},\dot{\xi}, \dot{a}, \dot{\theta}, \phi, \psi](t),\quad l = 1,\cdots, n,\\
&\dot{\theta}_{12} = \mu_0^{-1}\Pi_{n+1}[\lambda,\xi, a, \theta, \dot{\lambda},\dot{\xi}, \dot{a}, \dot{\theta}, \phi, \psi](t),\\
&\dot{a}_1 = \mu_0^{-1}\Pi_{n+2}[\lambda,\xi, a, \theta, \dot{\lambda},\dot{\xi}, \dot{a}, \dot{\theta}, \phi, \psi](t),\\
&\dot{a}_2 = \mu_0^{-1}\Pi_{n+3}[\lambda,\xi, a, \theta, \dot{\lambda},\dot{\xi}, \dot{a}, \dot{\theta}, \phi, \psi](t),\\
&\dot{\theta}_{1l} = \mu_0^{-1}\Pi_{n+l+1}[\lambda,\xi, a, \theta, \dot{\lambda},\dot{\xi}, \dot{a}, \dot{\theta}, \phi, \psi](t),\quad l = 3, \cdots, n,\\
&\dot{\theta}_{2l} = \mu_0^{-1}\Pi_{2n+l-1}[\lambda,\xi, a, \theta, \dot{\lambda},\dot{\xi}, \dot{a}, \dot{\theta}, \phi, \psi](t), \quad l = 3, \cdots, n.
\end{aligned}
\right.
\end{equation}
The terms in the right hand side of the above system can be expressed as
\begin{equation*}
\begin{aligned}
&\Pi_0[\lambda,\xi, a, \theta, \dot{\lambda}, \dot{\xi}, \dot{a}, \dot{\theta}, \phi, \psi](t) = \frac{t_0^{-\varepsilon}}{R^{\alpha-2}}\mu_0^{n-3 + \sigma}(t)f_0(t)+ \frac{t_0^{-\varepsilon}}{R^{\alpha-2}}\times\\
&\quad \Theta_0\left[\dot{\lambda},\dot{\xi}, \mu_0\dot{a}, \mu_0\dot{\theta}, \mu_0^{n-4}(t)\lambda, \mu_0^{n-4}(\xi-q), \mu_0^{n-3}a, \mu_0^{n-3}\theta, \mu_0^{n-3+\sigma}\phi, \mu_0^{\frac{n-2}{2}+\sigma}\psi\right](t)
\end{aligned}
\end{equation*}
and for $j = 1, \cdots, 3n-1$,
\begin{equation*}
\begin{aligned}
&\Pi_j[\lambda,\xi, a, \theta, \dot{\lambda},\dot{\xi}, \dot{a}, \dot{\theta}, \phi, \psi](t)\\
&= \mu_0^{n-2}c_j\left[b^{n-2}\nabla H(q, q)\right]+ \mu_0^{n-2+\sigma}(t)f_j(t) + \frac{t_0^{-\varepsilon}}{R^{\alpha-2}}\times\\
&\quad \Theta_j\left[\dot{\lambda},\dot{\xi}, \mu_0\dot{a}, \mu_0\dot{\theta}, \mu_0^{n-4}(t)\lambda, \mu_0^{n-4}(\xi-q), \mu_0^{n-3}a, \mu_0^{n-3}\theta, \mu_0^{n-3+\sigma}\phi, \mu_0^{\frac{n-2}{2}+\sigma}\psi\right](t),
\end{aligned}
\end{equation*}
where $f_j(t)$ and $\Theta_j[\cdots](t)$ ($j = 0, \cdots, 3n-1$) are bounded smooth functions for $t\in [t_0,\infty)$, $c_j$ ($j = 0, \cdots, 3n-1$) are suitable constants. Moreover, we have
\begin{equation*}
\left|\Theta_j[\dot{\lambda}_1](t) - \Theta_j[\dot{\lambda}_2](t)\right|\lesssim \frac{t_0^{-\varepsilon}}{R^{\alpha-2}}|\dot{\lambda}_1(t) - \dot{\lambda}_2(t)|
\end{equation*}
\begin{equation*}
\left|\Theta_j[\dot{\xi}_1](t) - \Theta_j[\dot{\xi}_2](t)\right|\lesssim \frac{t_0^{-\varepsilon}}{R^{\alpha-2}}|\dot{\xi}_1(t) - \dot{\xi}_2(t)|,
\end{equation*}
\begin{equation*}
\left|\Theta_j[\mu_0\dot{a}_1^{(1)}](t) - \Theta_j[\mu_0\dot{a}_1^{(2)}](t)\right|\lesssim \frac{t_0^{-\varepsilon}}{R^{\alpha-2}}\mu_0|\dot{a}_1^{(1)}(t) - \dot{a}_1^{(2)}(t)|,
\end{equation*}
\begin{equation*}
\left|\Theta_j[\mu_0\dot{a}_2^{(1)}](t) - \Theta_j[\mu_0\dot{a}_2^{(2)}](t)\right|\lesssim \frac{t_0^{-\varepsilon}}{R^{\alpha-2}}\mu_0|\dot{a}_2^{(1)}(t) - \dot{a}_2^{(2)}(t)|,
\end{equation*}
\begin{equation*}
\left|\Theta_j[\mu_0\dot{\theta}_1](t) - \Theta_j[\mu_0\dot{\theta}_2](t)\right|\lesssim \frac{t_0^{-\varepsilon}}{R^{\alpha-2}}\mu_0|\dot{\theta}_1(t) - \dot{\theta}_2(t)|,
\end{equation*}
\begin{equation*}
\left|\Theta_j[\mu_0^{n-4}\lambda_1](t) - \Theta_j[\mu_0^{n-4}\lambda_2](t)\right|\lesssim \frac{t_0^{-\varepsilon}}{R^{\alpha-2}}|\lambda_1(t) - \lambda_2(t)|,
\end{equation*}
\begin{equation*}
\left|\Theta_j[\mu_0^{n-4}(\xi_1-q)](t) - \Theta_j[\mu_0^{n-4}(\xi_2-q)](t)\right|\lesssim \frac{t_0^{-\varepsilon}}{R^{\alpha-2}}|\xi_1(t) - \xi_2(t)|,
\end{equation*}
\begin{equation*}
\left|\Theta_j[\mu_0^{n-3}a_1^{(1)}](t) - \Theta_j[\mu_0^{n-3}a_1^{(2)}](t)\right|\lesssim \frac{t_0^{-\varepsilon}}{R^{\alpha-2}}\mu_0|a_1^{(1)}(t) - a_1^{(2)}(t)|,
\end{equation*}
\begin{equation*}
\left|\Theta_j[\mu_0^{n-3}a_2^{(1)}](t) - \Theta_j[\mu_0^{n-3}a_2^{(2)}](t)\right|\lesssim \frac{t_0^{-\varepsilon}}{R^{\alpha-2}}\mu_0|a_2^{(1)}(t) - a_2^{(2)}(t)|,
\end{equation*}
\begin{equation*}
\left|\Theta_j[\mu_0^{n-3}\theta_1](t) - \Theta_j[\mu_0^{n-3}\theta_2](t)\right|\lesssim \frac{t_0^{-\varepsilon}}{R^{\alpha-2}}\mu_0|\theta_1(t) - \theta_2(t)|,
\end{equation*}
\begin{equation*}
\left|\Theta[\mu_0^{n-3+\sigma}\phi_1](t) - \Theta[\mu_0^{n-3+\sigma}\phi_2](t)\right|\lesssim \frac{t_0^{-\varepsilon}}{R^{\alpha-2}}\|\phi_1(t) - \phi_2(t)\|_{n-2+\sigma, \alpha}.
\end{equation*}
\end{prop}

System (\ref{e5:9000}) is solvable for $\lambda$, $\xi$, $a$, $\theta$ satisfying (\ref{e4:21}) and (\ref{e4:22}). Indeed, we have
\begin{prop}\label{p5:5.10}
(\ref{e5:9000}) has a solution $\lambda = \lambda[\phi](t)$, $\xi = \xi[\phi](t)$, $a = a[\phi](t)$ and $\theta = \theta[\phi](t)$ satisfying estimates (\ref{e4:21}) and (\ref{e4:22}). Moreover, for $t\in (t_0, \infty)$, there hold
\begin{equation*}\label{e5:120}
\mu_0^{-(1+\sigma)}(t)\big|\lambda[\phi_1](t) - \lambda[\phi_2](t)\big|\lesssim \frac{t_0^{-\varepsilon}}{R^{\alpha-2}}\|\phi_1 - \phi_2\|_{n-2+\sigma, \alpha},
\end{equation*}
\begin{equation*}\label{e5:130}
\mu_0^{-(1+\sigma)}(t)\big|\xi[\phi_1](t) - \xi[\phi_2](t)\big|\lesssim \frac{t_0^{-\varepsilon}}{R^{\alpha-2}}\|\phi_1 - \phi_2\|_{n-2+\sigma, \alpha},
\end{equation*}
\begin{equation*}\label{e5:1301}
\mu_0^{-\sigma}(t)\big|a[\phi_1](t) - a[\phi_2](t)\big|\lesssim \frac{t_0^{-\varepsilon}}{R^{\alpha-2}}\|\phi_1 - \phi_2\|_{n-2+\sigma, \alpha},
\end{equation*}
\begin{equation*}\label{e5:132}
\mu_0^{-\sigma}(t)\big|\theta[\phi_1](t) - \theta[\phi_2](t)\big|\lesssim \frac{t_0^{-\varepsilon}}{R^{\alpha-2}}\|\phi_1 - \phi_2\|_{n-2+\sigma, \alpha}.
\end{equation*}
\end{prop}

Using Proposition \ref{p4:4.2}, the proof of Proposition \ref{l5:1000} and \ref{p5:5.10} is similar to that of \cite{cortazar2016green} and \cite{MussoSireWeiZhengZhou}, we omit it.

{\bf Step 3. Gluing: the inner problem.}
After choosing parameter functions $\lambda = \lambda[\phi](t)$, $\xi = \xi[\phi](t)$, $a = a[\phi](t)$ and $\theta = \theta[\phi](t)$ such that (\ref{e5:70}) hold, we solve problem (\ref{e5:40}) in the class of functions with $\|\phi\|_{n-2+\sigma,\alpha}$ bounded. Problem (\ref{e5:40}) is a fixed point of
\begin{equation*}
\phi = \mathcal{A}_1(\phi): = \mathcal{T}_2(H[\lambda,\xi, a, \theta, \dot{\lambda},\dot{\xi}, \dot{a}, \dot{\theta}, \phi]).
\end{equation*}
It is easy to see that
\begin{equation}\label{e6:10}
\begin{aligned}
&\left|H[\lambda,\xi,\dot{\lambda},\dot{\xi},\phi](y, t)\right|\lesssim t_0^{-\varepsilon}\frac{\mu_0^{n-2+\sigma}}{1+|y|^{2+\alpha}}
\end{aligned}
\end{equation}
and
\begin{equation}\label{e6:20}
\begin{aligned}
&\left|H[\phi^{(1)}]-H[\phi^{(2)}]\right|(y, t)\lesssim t_0^{-\varepsilon}\|\phi^{(1)} - \phi^{(2)}\|_{n-2+\sigma, \alpha}
\end{aligned}
\end{equation}
hold. From (\ref{e6:10}) and (\ref{e6:20}), $\mathcal{A}_1$ has a fixed point $\phi$ in the set of functions $\|\phi\|_{n-2s+\sigma, \alpha}\leq ct_0^{-\varepsilon}$ for suitable large constant $c > 0$. From the Contraction Mapping Theorem, we obtain a solution to (\ref{e3:1}). Then the rest argument to the stability part of Theorem \ref{t:stability} is the same as \cite{cortazar2016green}, we omit it.

\section{Appendix}
\subsection{Proof of Lemma \ref{e:orthogalityofkernels}}
Let us recall from \cite{delpinomussofrankpistoiajde2011} and \cite{MussoWei2015} that
$$
Q_k(x) =
U(x) - \sum_{j=1}^k U_j (x) +\tilde  \phi (x)\quad {\mbox {with}} \quad  U(x) = \left( {2 \over 1+ |x|^2} \right)^{n-2 \over 2}
$$
and
$$
 U_j (x) = \zeta_k^{-\frac{n-2}{2}} U(\zeta_k^{-1} (x-\xi_j )).
$$
Here $\zeta_k$ is a positive constant satisfying $\zeta_k \sim k^{-2}$, $\xi_j = \sqrt{1-\zeta_k^2}(\textbf{n}_j, 0)$, $\textbf{n}_j = (\cos\theta_j, \sin\theta_j, 0)$, $\theta_j = \frac{2\pi}{k}(j-1)$ and $\tilde{\phi}$ is a small term than $U(x) - \sum_{j=1}^k U_j (x)$. Let us introduce the functions
\begin{equation*}
Z_0 (x) = {n-2 \over 2} U(x) + \nabla U(x) \cdot x ,
\end{equation*}
\begin{equation*}
\pi_0(x)={n-2 \over 2}\tilde\phi (x) + \nabla \tilde \phi (x) \cdot x
\end{equation*}
and
\begin{equation*}
Z_\alpha (x) = \frac{\partial}{\partial x_\alpha} U(x), \quad \pi_\alpha(x)=\frac{\partial}{\partial x_\alpha}\tilde{\phi}(x)\quad {\mbox {for}} \quad \alpha = 1, \ldots , n. \end{equation*}
For $l=1, \ldots , k$, define
\begin{equation*}
Z_{0 l} (x) = {n-2 \over 2} U_l (x) + \nabla U_l (x) \cdot (x-\xi_l ) ,
\end{equation*}
From (\ref{capitalzeta0}) and (\ref{capitalzetaj}),
$$
z_0 (x) = Z_0 (x) - \sum_{l=1}^k \left[ Z_{0l}  (x) + \sqrt{1-\zeta_k^2 } \cos \theta_l  {\partial \over \partial x_1} U_l (x) \right.
$$
$$
+ \left. \sqrt{1-\zeta_k^2} \sin \theta_l {\partial \over \partial x_2} U_l (x)  \right] + \pi_0 (x).
$$
For $l=1, \ldots , k$, define
\begin{equation*}
Z_{1 l} (x) =\sqrt{1-\zeta_k^2} \,  \left[\cos \theta_l  {\partial \over \partial x_1} U_l (x)+ \sin \theta_l {\partial \over \partial x_2} U_l (x)  \right],
\end{equation*}
\begin{equation*}
 Z_{2 l} (x) =\sqrt{1-\zeta_k^2} \,  \left[ -\sin  \theta_l  {\partial \over \partial x_1} U_l (x)+ \cos \theta_l {\partial \over \partial x_2} U_l (x)  \right],
\end{equation*}
\begin{equation*}
Z_{\alpha l} (x) = {\partial \over \partial x_\alpha} U_l(x), \quad {\mbox {for}} \quad \alpha = 3, \ldots , n.
\end{equation*}
Then we have
\begin{equation}\label{1ang1}
z_0 (x) = Z_0 (x) - \sum_{l=1}^k \left[ Z_{0 l}  (x) +Z_{1l} (x)  \right] + \pi_0 (x),
\end{equation}
\begin{equation}\label{1ang111}
\begin{aligned}
z_1 (x) &= Z_1 (x) - \sum_{l=1}^k  {\partial \over \partial x_1} U_l (x) + \pi_1 (x) \\
&= Z_1 (x) - \sum_{l=1}^k  {  [ \cos \theta_l Z_{1l} (x) -\sin \theta_l Z_{2l} (x) ] \over \sqrt{1-\zeta_k^2}} + \pi_1 (x),
\end{aligned}
\end{equation}
\begin{equation}\label{1ang222}
\begin{aligned}
z_2 (x) &= Z_2 (x) - \sum_{l=1}^k  {\partial \over \partial x_2} U_2 (x) + \pi_2 (x) \\
&= Z_2 (x) - \sum_{l=1}^k  {  [ \sin \theta_l Z_{1l} (x) +\cos \theta_l Z_{2l} (x) ] \over \sqrt{1-\zeta_k^2}} + \pi_2 (x),
\end{aligned}
\end{equation}
and
\begin{equation}\label{1ang333}
z_\alpha (x) = Z_\alpha (x) - \sum_{l=1}^k Z_{\alpha l} +\pi_\alpha (x) \text{ for }\alpha =3, \cdots , n.
\end{equation}
Moreover, the following identities hold,
\begin{equation}\label{1canada2}
z_{n+1} (x) = \sum_{l=1}^k Z_{2 l} (x)  + x_2 \pi_1 (x) - x_1 \pi_2 (x),
\end{equation}
\begin{equation}\label{1canada30}
\begin{aligned}
z_{n+2} (x) = \sum_{l=1}^k\sqrt{1-\zeta_k^2} \cos \theta_l Z_{0 l} (x)  - &\sum_{l=1}^k\sqrt{1-\zeta_k^2} \cos \theta_l Z_{1 l} (x) \\
&  -2x_1 \pi_0 (x) + |x|^2 \pi_1 (x),
\end{aligned}
\end{equation}
\begin{equation}\label{1canada3}
\begin{aligned}
z_{n+3} (x) =  \sum_{l=1}^k \, \sqrt{1-\zeta_k^2}\sin \theta_l Z_{0 l} (x) -  &\sum_{l=1}^k  \, \sqrt{1-\zeta_k^2}\sin \theta_l Z_{1 l} (x) \\
& - 2x_2 \pi_0 (x) + |x|^2 \pi_2 (x),
\end{aligned}
\end{equation}
\begin{equation}\label{1canada4}
z_{n+\alpha +1} (x) =  \sqrt{1-\zeta_k^2}\,  \sum_{l=1}^k  \cos \theta_l Z_{\alpha l} (x) + x_1 \pi_\alpha (x), \text{ for }\alpha =3, \ldots , n,
\end{equation}
\begin{equation}\label{1canada5}
z_{2n+\alpha -1} (x) =  \sqrt{1-\zeta_k^2} \,  \sum_{l=1}^k \sin \theta_l Z_{\alpha l} (x) + x_2 \pi_\alpha (x), \text{ for }\alpha =3, \ldots , n.
\end{equation}
Then we have the following estimations,
\begin{lemma}\label{lemma9.1}
\begin{equation}\label{1made}
\begin{aligned}
\int_{\mathbb{R}^n} Z_{\alpha l}(x)Z_0(x)dx &= \int_{\mathbb{R}^n} Z_0^2(x) dx + O(k^{-1}) \quad {\mbox {if}} \quad \alpha=0,\,\, l=0, \\
&=O(k^{-1}) \quad {\mbox {otherwise}},
\end{aligned}
\end{equation}
\begin{equation}\label{1made1}
\begin{aligned}
\int_{\mathbb{R}^n} Z_{\alpha l}(x)Z_\beta(x) dx&= \int_{\mathbb{R}^n} Z_1^2(x) dx + O(k^{-1}) \quad {\mbox {if}} \quad \alpha=\beta\in \{1,\cdots, n\},\,\, l=0, \\
&= O(k^{-1}) \quad {\mbox {otherwise}},
\end{aligned}
\end{equation}
\begin{equation}\label{1made2}
\begin{aligned}
\int_{\mathbb{R}^n} Z_{\alpha l}(x)Z_{0j}(x)dx &= \int_{\mathbb{R}^n} Z_0^2(x) dx + O(k^{-1})\quad {\mbox {if}} \quad \alpha=0,\,\, l=j, \\
&= O(k^{-1}) \quad {\mbox {otherwise}},
\end{aligned}
\end{equation}
\begin{equation}\label{1made3}
\begin{aligned}
\int_{\mathbb{R}^n} Z_{\alpha l}(x)Z_{\beta j}(x)dx &= \int_{\mathbb{R}^n} Z_1^2(x) dx + O(k^{-1})\quad {\mbox {if}} \quad \alpha=\beta\in \{1,\cdots, n\},\,\, l=j, \\
&= O(k^{-1}) \quad {\mbox {otherwise}},
\end{aligned}
\end{equation}
\begin{equation}\label{1made4}
\begin{aligned}
&\int_{\mathbb{R}^n}\frac{|x|^2 - 2}{\left(1+|x|^2\right)^{\frac{n-2}{2}+1}}Z_{\beta j}(x)dx\\
&\quad\quad\quad\quad = \left\{
\begin{aligned}
& \int_{\mathbb{R}^n} \frac{(|x|^2 - 2)}{\left(1+|x|^2\right)^{\frac{n-2}{2}+1}}Z_{0}(x)dx + O(k^{-1})\quad {\mbox {if}} \quad \beta=0, \,\, j=0, \\
& O(k^{-1}) \quad {\mbox {otherwise}},
\end{aligned}
\right.
\end{aligned}
\end{equation}

\begin{equation}\label{1made5}
\begin{aligned}
&\int_{\mathbb{R}^n}\frac{x_i}{\left(1+|x|^2\right)^{\frac{n-2}{2}+1}}Z_{\beta j}(x)dx\\
&\quad\quad = \left\{
\begin{aligned}
&\int_{\mathbb{R}^n}\frac{x_i}{\left(1+|x|^2\right)^{\frac{n-2}{2}+1}}Z_{i}(x)dx + O(k^{-1})\quad {\mbox {if}} \quad \beta=0,\,\, j=i\in \{1,\cdots, n\}, \\
&O(k^{-1}) \quad {\mbox {otherwise}}.
\end{aligned}
\right.
\end{aligned}
\end{equation}
\end{lemma}
\begin{proof}
We prove (\ref{1made2}). Let $\eta >0$ be a small fixed real number independent from $k$. Then
\begin{equation*}
\begin{aligned}
\int_{\mathbb{R}^n} Z_{\alpha l }(x) Z_{0 j}(x)dx &=  \int_{ B(\xi_l , \frac{\eta}{k})} Z_{\alpha l}(x) Z_{0 l}(x)dx  + \int_{\R^n \setminus B(\xi_l , \frac{\eta}{k})} Z_{\alpha l}(x)  Z_{0 j}(x)dx\\
&:= i_1 + i_2 .
\end{aligned}
\end{equation*}
Change of variable via $x= \xi_l + \zeta_k y$, we obtain
\begin{equation*}
\begin{aligned}
i_1 & = \int_{B(0, \frac{\eta}{k\zeta_k})}Z_\alpha(x)Z_0(x)dx\\
&=  \left(\int_{\mathbb{R}^n}Z_0^2(x)dx + O((\zeta_k k)^n)\right) \text{ if }\alpha = 0,\\
&= 0 \text{ if }\alpha \neq 0.
\end{aligned}
\end{equation*}
As for the term $i_2$, decompose
$$
i_2 = \int_{\R^n \setminus \cup_{j = 1}^kB(\xi_j, \frac{\eta}{k})} Z_{\alpha l}(x)Z_{0j}(x)dx + \sum_{j\neq l}\int_{B(\xi_j, \frac{\eta}{k})} Z_{\alpha l}(x)Z_{0j}(x)dx = i_{21} + i_{22}.
$$
$i_{21}$ can be estimated as
\begin{equation*}
\begin{aligned}
\left|i_{21}\right| &\leq   C\zeta_k^{n-2}\int_{\{|x|\geq \frac{\eta}{k}\}}^\infty\frac{1}{|x|^{2n-4}}dx =  C\zeta_k^{n-2}\int_{\frac{\eta}{k}}^\infty\frac{r^{n-1}}{r^{2n-4}}dr \\
&\leq C\zeta_k^{n-2}k^{n-4} = O\left(\frac{1}{k}\right).
\end{aligned}
\end{equation*}
And
\begin{equation*}
\begin{aligned}
\left|i_{22}\right| \leq C\sum_{j\neq l}\int_{B(\xi_j,\frac{\eta}{k})}\frac{\zeta_k^{\frac{n-2}{2}}}{|x-\xi_l|^{n-2}}Z_{0j}(x)dx &\leq C\zeta_k^{\frac{n-2}{2}}k^{n-2}\int^{\frac{\eta}{k}}_0\frac{r^{n-1}}{r^{n-2}}dr\\
&\leq C\zeta_k^{\frac{n-2}{2}}k^{n-4}\leq C\zeta_k = O\left(\frac{1}{k}\right),
\end{aligned}
\end{equation*}
where $C$ are generic positive constants independent of $k$. Hence we have (\ref{1made2}). The proofs of (\ref{1made}), (\ref{1made1}), (\ref{1made3}), (\ref{1made4}) and (\ref{1made5}) are similar, we omit them. This concludes the proof.
\end{proof}
Then Lemma \ref{e:orthogalityofkernels} follows from Lemma \ref{lemma9.1}, (\ref{1ang1})-(\ref{1canada5}) and Proposition 2.1 of \cite{MussoWei2015} by long but easy estimates.

\subsection{Proof of (\ref{positiveness})}
First, we claim that
\begin{equation}\label{e:appendix01}
\int_{\mathbb{R}^n}|Q|^{p-1}(y)Z_{0}(y)dy = \int_{\mathbb{R}^n}U^{p-1}(y)Z_{0}(y)dy + O\left(\frac{1}{k^s}\right)
\end{equation}
for some small $s > 0$.
Indeed, we have
\begin{equation*}
\begin{aligned}
&\int_{\mathbb{R}^n}|Q|^{p-1}(y)Z_{0}(y)dy\\
&\quad = \int_{\mathbb{R}^n}\left|U(y) - \sum_{j=1}^k U_j (y) +\tilde\phi (y)\right|^{p-1}Z_{0}(y)dy \\
&\quad = \int_{\mathbb{R}^n}U^{p-1}(y)Z_{0}(y)dy + (p-1)\int_{\mathbb{R}^n}U^{p-2}(y)\left|- \sum_{j=1}^k U_j (y) +\tilde\phi (y)\right|Z_{0}(y)dy \\
&\quad\quad + O\left(\sum_{j = l}^k\int_{\mathbb{R}^n}|U_j|^{p-1}(y)Z_{0}(y)dy\right) + O\left(\int_{\mathbb{R}^n}|\tilde\phi(y)|^{p-1}(y)Z_{0}(y)dy\right)\\
&\quad = \int_{\mathbb{R}^n}U^{p-1}(y)Z_{0}(y)dy + O\left(\sum_{j=1}^k\int_{\mathbb{R}^n}U^{p-2}(y)U_j (y)Z_{0}(y)dy\right)\\
&\quad\quad + O\left(\sum_{j = l}^k\int_{\mathbb{R}^n}|U_j|^{p-1}(y)Z_{0}(y)dy\right) + O\left(\int_{\mathbb{R}^n}|\tilde\phi(y)|^{p-1}(y)Z_{0}(y)dy\right)\\
&\quad\quad + O\left(\int_{\mathbb{R}^n}U^{p-2}(y)\left|\tilde\phi (y)\right|Z_{0}(y)dy\right)
\end{aligned}
\end{equation*}
and
\begin{equation*}
\begin{aligned}
&\int_{\mathbb{R}^n}|U_j|^{p-1}(y)Z_{0}(y)dy = 4\int_{\mathbb{R}^n}\frac{\zeta_k^2}{(\zeta_k^2 + |y-\xi_j|^2)^2}Z_{0}(y)dy\\
&\quad = 4\zeta_k^{n-2}\int_{\mathbb{R}^n}\frac{1}{(1 + |z|^2)^2}Z_{0}(\zeta_k z+\xi_j)dy\\
&\quad \leq C\zeta_k^{n-2}\int_{\mathbb{R}^n}\frac{1}{(1 + |z|^2)^2}\frac{1}{\left|\zeta_k z+\xi_j\right|^{n-2}}dy\\
&\quad = C\zeta_k^{n-2}\int_{|z|\leq \frac{1}{2\zeta_k}}\frac{1}{(1 + |z|^2)^2}\frac{1}{\left|\zeta_k z+\xi_j\right|^{n-2}}dy\\
&\quad\quad + C\zeta_k^{n-2}\int_{|z|\geq \frac{1}{2\zeta_k}}\frac{1}{(1 + |z|^2)^2}\frac{1}{\left|\zeta_k z+\xi_j\right|^{n-2}}dy\\
&\quad = C\zeta_k^{n-2}\int_{|z|\leq \frac{1}{2\zeta_k}}\frac{1}{(1 + |z|^2)^2}\frac{1}{\left|\xi_j\right|^{n-2}}\left(1 + O\left(\frac{\zeta_k z}{|\xi_j|}\right)\right)dy\\
&\quad\quad + C\zeta_k^{n-2}\int_{|z|\geq \frac{1}{2\zeta_k}}\frac{1}{(1 + |z|^2)^2}\frac{1}{\left|\zeta_k z\right|^{n-2}}\left(1 + O\left(\frac{|\xi_j|}{\zeta_k z}\right)\right)dy\\
&\quad = O\left(\zeta_k^{2}\right) = O\left(\frac{1}{k^4}\right),
\end{aligned}
\end{equation*}
\begin{equation*}
\begin{aligned}
&\int_{\mathbb{R}^n}U^{p-2}(y)U_j (y)Z_{0}(y)dy\leq  C\int_{\mathbb{R}^n}\frac{\zeta_k^{\frac{n-2}{2}}}{(\zeta_k^2 + |y-\xi_j|^2)^{\frac{n-2}{2}}}\frac{1}{(1+|y|)^4}dy\\
&\quad = C\zeta_k^{\frac{n}{2}+1}\int_{\mathbb{R}^n}\frac{1}{(1 + |z|^2)^{\frac{n-2}{2}}}\frac{1}{(1+\left|\zeta_k z+\xi_j\right|)^4}dy\\
&\quad \leq C\zeta_k^{\frac{n}{2}+1}\int_{\mathbb{R}^n}\frac{1}{(1 + |z|^2)^{\frac{n-2}{2}}}\frac{1}{\left|\zeta_k z+\xi_j\right|^4}dy\\
&\quad = C\zeta_k^{\frac{n}{2}+1}\int_{|z|\leq \frac{1}{2\zeta_k}}\frac{1}{(1 + |z|^2)^{\frac{n-2}{2}}}\frac{1}{\left|\zeta_k z+\xi_j\right|^4}dy\\
&\quad\quad + C\zeta_k^{\frac{n}{2}+1}\int_{|z|\geq \frac{1}{2\zeta_k}}\frac{1}{(1 + |z|^2)^{\frac{n-2}{2}}}\frac{1}{\left|\zeta_k z+\xi_j\right|^4}dy\\
&\quad= C\zeta_k^{\frac{n}{2}+1}\int_{|z|\leq \frac{1}{2\zeta_k}}\frac{1}{(1 + |z|^2)^{\frac{n-2}{2}}}\frac{1}{\left|\xi_j\right|^{4}}\left(1 + O\left(\frac{\zeta_k z}{|\xi_j|}\right)\right)dy\\
&\quad\quad + C\zeta_k^{\frac{n}{2}+1}\int_{|z|\geq \frac{1}{2\zeta_k}}\frac{1}{(1 + |z|^2)^{\frac{n-2}{2}}}\frac{1}{\left|\zeta_k z\right|^{4}}\left(1 + O\left(\frac{|\xi_j|}{\zeta_k z}\right)\right)dy\\
&\quad = O\left(\zeta_k^{\frac{n}{2}-1}\right),
\end{aligned}
\end{equation*}
\begin{equation*}
\begin{aligned}
\int_{\mathbb{R}^n}|\tilde\phi(y)|^{p-1}(y)Z_{0}(y)dy = O\left(k^{-\frac{n}{q}\frac{4}{n-2}}\int_{\mathbb{R}^n}\frac{1}{(1+|y|)^{n+2}}dy\right) = O\left(k^{-\frac{n}{q}\frac{4}{n-2}}\right),
\end{aligned}
\end{equation*}
\begin{equation*}
\begin{aligned}
\int_{\mathbb{R}^n}U^{p-2}(y)\left|\tilde\phi (y)\right|Z_{0}(y)dy = O\left(k^{-\frac{n}{q}}\int_{\mathbb{R}^n}\frac{1}{(1+|y|)^{n+2}}dy\right) = O\left(k^{-\frac{n}{q}}\right)
\end{aligned}
\end{equation*}
hold. (\ref{e:appendix01}) follows from the above estimates. Similarly, we have
\begin{equation*}\label{e:appendix02}
\int_{\mathbb{R}^n}|Q|^{p-1}(y)Z_{0l}(y)dy = \int_{\mathbb{R}^n}|U_l|^{p-1}(y)Z_{0l}(y)dy + O\left(\frac{1}{k^{1+s}}\right)
\end{equation*}
and
\begin{equation*}\label{e:appendix03}
\int_{\mathbb{R}^n}|Q|^{p-1}(y)Z_{1l}(y)dy = \int_{\mathbb{R}^n}|U_l|^{p-1}(y)Z_{1l}(y)dy + O\left(\frac{1}{k^{1+s}}\right).
\end{equation*}
Moreover,
$$
-p\int_{\mathbb{R}^n}U^{p-1}(y)Z_{0}(y)dy = \frac{n-2}{2}\int_{\mathbb{R}^n}U^{p}(y)dy > 0,
$$
\begin{equation*}
\begin{aligned}
-p\int_{\mathbb{R}^n}|U_l|^{p-1}(y)Z_{0l}(y)dy &= \zeta_k^{\frac{n}{2}-1}\left(-p\int_{\mathbb{R}^n}U^{p-1}(y)Z_{0}(y)dy\right)\\
 &= \zeta_k^{\frac{n}{2}-1}\frac{n-2}{2}\int_{\mathbb{R}^n}U^{p}(y)dy,
\end{aligned}
\end{equation*}
$$
\int_{\mathbb{R}^n}|U_l|^{p-1}(y)Z_{1l}(y)dy = 0.
$$
Then from (\ref{1ang1}),
\begin{equation*}\label{e:appendix05}
\begin{aligned}
&-p\int_{\mathbb{R}^n}|Q|^{p-1}(y)z_{0}(y)dy\\
&= -p\int_{\mathbb{R}^n}|Q|^{p-1}(y)Z_{0}(y)dy +p\sum_{l=1}^k\int_{\mathbb{R}^n}|Q|^{p-1}(y)Z_{0l}(y)dy\\
&\quad +p\sum_{l=1}^k\int_{\mathbb{R}^n}|Q|^{p-1}(y)Z_{1l}(y)dy -p \int_{\mathbb{R}^n}\pi_0(y)Z_{0l}(y)dy\\
&= -p\int_{\mathbb{R}^n}U^{p-1}(y)Z_{0}(y)dy +p\sum_{l=1}^k\int_{\mathbb{R}^n}|U_l|^{p-1}(y)Z_{0l}(y)dy\\
&\quad +p\sum_{l=1}^k\int_{\mathbb{R}^n}|U_l|^{p-1}(y)Z_{1l}(y)dy -p \int_{\mathbb{R}^n}\pi_0(y)Z_{0l}(y)dy + O\left(\frac{1}{k^s}\right)\\
&= (1 + k\zeta_k^{\frac{n}{2}-1} )\frac{n-2}{2}\int_{\mathbb{R}^n}U^{p}(y)dy + O\left(\frac{1}{k^s}\right)
\end{aligned}
\end{equation*}
which is positive when $k$ is large. This proves $c_1 > 0$ when $k_0$ is large enough.

Finally, we prove $c_2 > 0$. From (\ref{1ang1}), $z_0(y) - \frac{D_{n,k}(2-|y|^2)}{\left(1+|y|^2\right)^{\frac{n}{2}}}$ can be written as
\begin{equation*}\label{1ang100}
\begin{aligned}
&z_0(y) - \frac{D_{n,k}(2-|y|^2)}{\left(1+|y|^2\right)^{\frac{n}{2}}}\\ &= \left(Z_0(y) - \frac{n-2}{2}\frac{\alpha_n(2-|y|^2)}{\left(1+|y|^2\right)^{\frac{n}{2}}}\right)\\
&\quad - \sum_{l=1}^k \left(Z_{0l}(y) - \zeta_k^{\frac{n-2}{2}}f_n\frac{(2-|y|^2)}{\left(1+|y|^2\right)^{\frac{n}{2}}}\right)\\
&\quad - \sum_{l=1}^k Z_{1l}(y) + \pi_0(y) - o(1)h_n\frac{(2-|y|^2)}{\left(1+|y|^2\right)^{\frac{n}{2}}},\quad (k\to +\infty).
\end{aligned}
\end{equation*}
Direct computation yields that
\begin{equation*}
\begin{aligned}
&\int_{\mathbb{R}^n}\left(Z_0(y) - \frac{n-2}{2}\frac{\alpha_n(2-|y|^2)}{\left(1+|y|^2\right)^{\frac{n}{2}}}\right)z_0(y)dy\\
&\quad = \int_{\mathbb{R}^n}\left(Z_0(y) - \frac{n-2}{2}\frac{\alpha_n(2-|y|^2)}{\left(1+|y|^2\right)^{\frac{n}{2}}}\right)Z_0(y)dy + O\left(\frac{1}{k}\right)\\
&\quad = \alpha_n\frac{n-2}{2}\frac{\sqrt{\pi } 2^{-n} \Gamma \left(\frac{n}{2}-1\right)}{\Gamma \left(\frac{n+1}{2}\right)} + O\left(\frac{1}{k}\right),
\end{aligned}
\end{equation*}
\begin{equation*}
\begin{aligned}
&\sum_{l=1}^k\int_{\mathbb{R}^n}\left(Z_{0l}(y) - \zeta_k^{\frac{n-2}{2}}f_n\frac{(2-|y|^2)}{\left(1+|y|^2\right)^{\frac{n}{2}}}\right)z_0(y)dy = O\left(\frac{1}{k}\right),
\end{aligned}
\end{equation*}
\begin{equation*}
\begin{aligned}
&\sum_{l=1}^k\int_{\mathbb{R}^n}Z_{1l}(y)z_0(y)dy = O\left(\frac{1}{k}\right),
\end{aligned}
\end{equation*}
\begin{equation*}
\begin{aligned}
&\sum_{l=1}^k\int_{\mathbb{R}^n}\left(\pi_0(y) - o(1)h_n\frac{(2-|y|^2)}{\left(1+|y|^2\right)^{\frac{n}{2}}}\right)z_0(y)dy = O\left(\frac{1}{k}\right).
\end{aligned}
\end{equation*}
Therefore,
$$
\int_{\mathbb{R}^n}\left(z_0(y) - \frac{D_{n,k}(2-|y|^2)}{\left(1+|y|^2\right)^{\frac{n}{2}}}\right)z_0(y)dy = \alpha_n\frac{n-2}{2}\frac{\sqrt{\pi } 2^{-n} \Gamma \left(\frac{n}{2}-1\right)}{\Gamma \left(\frac{n+1}{2}\right)} + O\left(\frac{1}{k}\right)
$$
which is positive when $k$ is large enough. Hence $c_2 > 0$ if $k_0$ is sufficiently large.

\section*{Acknowledgements}
M. del Pino and M. Musso have been partly supported by grants Fondecyt 1160135, 1150066, Fondo Basal CMM. J. Wei is partially supported by NSERC of Canada. Y. Zheng is partially supported by NSF of China (11301374) and China Scholarship Council (CSC).

\end{document}